\documentclass[a4paper,12pt]{article}
\usepackage[margin=1in]{geometry}
\usepackage{amsmath}
\usepackage{amssymb}
\usepackage{amsthm}
\usepackage{bbold}
\usepackage{shuffle}
\usepackage{hyperref}

\usepackage{tikz}

\hypersetup{colorlinks=true}
\hypersetup{linkcolor=black,citecolor=black,urlcolor=black,filecolor=black,menucolor=black}

\newcommand{\Int}{\ensuremath{{\textstyle\int}}}
\newcommand{\Der}{\ensuremath{\partial}}
\newcommand{\id}{\ensuremath{{\mathrm{id}}}}
\newcommand{\E}{\ensuremath{{\mathrm{E}}}}
\newcommand{\Q}{\ensuremath{{\mathrm{Q}}}}
\newcommand{\T}{\ensuremath{{\mathrm{T}}}}
\newcommand{\ep}{\ensuremath{\varepsilon}}
\newcommand{\et}{\ensuremath{1_\mathcal{T}}}
\newcommand{\Sym}{\ensuremath{{\mathrm{Sym}}}}

\DeclareMathOperator{\const}{Const}
\DeclareMathOperator{\linspan}{span}

\DeclareMathOperator{\im}{im}
\DeclareMathOperator{\sh}{sh}

\newtheorem{theorem}{\bf Theorem}
\newtheorem{definition}[theorem]{\bf Definition}
\newtheorem{lemma}[theorem]{\bf Lemma}
\newtheorem{corollary}[theorem]{\bf Corollary}
\newtheorem{conjecture}[theorem]{\bf Conjecture}

\theoremstyle{remark}
\newtheorem{remark}[theorem]{\bf Remark}
\newtheorem{example}[theorem]{\bf Example}

\title{The integro-differential closure of a commutative differential ring}
\author{Clemens G. Raab\textsuperscript{a,}\footnote{Corresponding author}\ \ and Georg Regensburger\textsuperscript{b,}\footnote{Present address: RISC, Johannes Kepler University Linz, Austria}}
\hypersetup{pdfauthor={Clemens G. Raab and Georg Regensburger}}
\date{10 September 2026}

\begin{document}
\maketitle

\begin{center}
\textsuperscript{a}RICAM, Austrian Academy of Sciences, Linz, Austria\\
\href{mailto:clemens.raab@ricam.oeaw.ac.at}{\tt clemens.raab@ricam.oeaw.ac.at}\\
\textsuperscript{b}Institute of Mathematics, University of Kassel, Germany\\
\href{mailto:regensburger@mathematik.uni-kassel.de}{\tt regensburger@mathematik.uni-kassel.de}
\end{center}

\begin{abstract}
 An integro-differential ring is a differential ring that is closed under an integration operation satisfying the fundamental theorem of calculus.
 Via the Newton--Leibniz formula, a generalized evaluation is defined in terms of integration and differentiation.
 The induced evaluation is not necessarily multiplicative, which allows one to model functions with singularities and leads to generalized shuffle relations.
 In general, not every element of a differential ring has an antiderivative in the same ring.
 Starting from a commutative differential ring and a direct decomposition into integrable and non-integrable elements, we construct the free integro-differential ring.
 This integro-differential closure contains all nested integrals over elements of the original differential ring.
 Analyzing the structure of nested integrals independently of concrete integrands, we separate the algebraic structure from analytic properties arising from concrete functions.
 We exhibit the relations satisfied by generalized evaluations of products of nested integrals.
 Investigating these relations of constants, we characterize in terms of Lyndon words certain evaluations of products that determine all others.
 We also analyze the relation of the free integro-differential ring with the shuffle algebra.
 To preserve integrals in the original differential ring for computations in its integro-differential closure, we introduce the notion of quasi-integro-differential rings and give two adapted constructions of the integro-differential closure.
 Finally, in a given integro-differential ring, we consider the internal integro-differential closure of a differential subring and identify it as quotient of the free integro-differential ring by certain constants.
\end{abstract}


\section{Introduction}

Methods to compute with and simplify nested integrals have been considered in many contexts for different classes of integrands.
For example, nested integrals of rational functions were already investigated by Kummer \cite{Kummer} and appear in the study of linear differential equations \cite{Poincare}.
They have been named hyperlogarithms \cite{LappoDanilevski} and satisfy many properties, see e.g.\ \cite{Wechsung}.
In theoretical physics, they were studied as well, where they arise in computations for quantum field theories, see e.g.\ \cite{WeinzierlBook,AbreuBrittoDuhr,BluemleinSchneider} and references therein.
Particular cases of functions defined by nested integrals of rational functions include polylogarithms and their generalizations, see e.g.\ \cite{RemiddiVermaseren,Cyclotomic}.
Nested integrals with integrands involving square roots of rational functions were analyzed more recently in \cite{Binomial}.
Even more general integrands are given by hyperexponential functions, leading to d'Alembertian functions \cite{AbramovPetkovsek}.

In this paper, we analyze from an algebraic point of view the underlying structure of nested integrals that is independent of the class of integrands.
In doing so, we cleanly separate the algebraic structure of nested integrals from analytic properties arising from concrete functions. 
This structure relies on a decomposition of integrands into integrable and non-integrable parts, which depends on the class of integrands and needs to be worked out for specific cases.
Based on such a decomposition of integrands, we work out normal forms of nested integrals and show how to compute with them.
Moreover, we discuss how constants arising in generalized shuffle relations for products of nested integrals have to be taken into account.

Algebraically, integrands and their nested integrals are considered as elements of a commutative ring that is closed under differentiation $\Der$ and integration $\Int$ such that the fundamental theorem of calculus $\Der\Int{f}=f$ holds.
These two operations give rise to the induced evaluation $\E$ on the ring by the Newton--Leibniz formula $\E{f}:=f-\Int\Der{f}$.
Such rings have been studied in the literature already, starting with integro-differential algebras over fields \cite{RosenkranzRegensburgerBVPFactor}, modelling integration and point evaluation of smooth functions, and differential Rota--Baxter algebras over rings \cite{GuoKeigherFreeDRB}.
In \cite{GuoRegensburgerRosenkranz}, these two concepts are compared and the free integro-differential algebra (with weight) over a ring is constructed on a (regular) differential ring.
Nonzero weight allows one to model the discrete analog of difference and summation, which we briefly discuss in Remark~\ref{rem:summation}.
The papers \cite{GaoGuoZheng,GaoGuoRosenkranz} deal with an explicit construction of the 
free integro-differential algebra over a set via Gr\"obner--Shirshov bases in the commutative and in the noncommutative case, see also Remark~\ref{rem:freeonset} later.
Recently, to enable modelling also functions with singularities, integro-differential rings with generalized evaluation have been introduced in \cite{RaabRegensburger}, which lead to generalized shuffle relations involving constants arising from evaluation.

For symbolic integration in commutative differential rings $(\mathcal{R},\Der)$, a basic problem is decomposing elements $f=\Der{g}+h$ such that $h$ is either zero or has no antiderivative in $\mathcal{R}$ and is minimal in some sense.
Considering the rational functions $(\mathcal{C}(x),\frac{d}{dx})$ over a field of characteristic zero, such an $h$ can be obtained by collecting the terms with simple poles in the partial fraction decomposition of $f$ while $g$ sums antiderivatives of each remaining term.
Computation of this decomposition for rational functions without the use of partial fraction decomposition was already worked out in \cite{Ostrogradsky,HermiteRational}, see also \cite[Ch.~2]{Bronstein}, and was extended to square roots of rational functions in \cite{HermiteSqrt}.
Much later, in computer algebra, algorithms to compute decompositions for various other classes of integrands were investigated.
The algorithm presented in \cite{TragerRadical} can be used to decompose algebraic functions in fields generated by radicals of arbitrary order.
This has been generalized to fields generated by an arbitrary algebraic function \cite{TragerPhD,ChenKauersKoutschanAlgebraic,ChenDuKauersLazyHermite}.
Additive decompositions were also developed for fractions of differential polynomials \cite{BoullierEtAlDifferentialFractions} and in differential fields generated by nested logarithms and other primitives \cite{DuGaoLiLi} or by non-nested hyperexponential functions \cite{ChenDuGaoLi}.

Not only in differential rings or fields, but also for integration in modules generated by D-finite functions, algorithms to compute such additive decompositions were developed.
First, this was achieved for modules of a single hyperexponential function \cite{GeddesLeLi,BostanEtAlHyperexponential} and, in recent years, various algorithms for finitely generated modules of D-finite functions have been worked out, see \cite{ChenHoeijKauersKoutschan,BostanEtAlDFinite,Hoeven,ChenDuKauersDFinite}.

In this work, we start from a commutative differential ring $(\mathcal{R},\Der)$ with ring of constants $\mathcal{C}$ and assume that such a decomposition $f=\Der{g}+h$ exists for all $f\in\mathcal{R}$.
In particular, we fix direct complements of $\im(\Der)$ and $\ker(\Der)=\mathcal{C}$ in $\mathcal{R}$ over the ring $\mathcal{C}$ corresponding to a quasi-integration, which is a $\mathcal{C}$-linear map $\Q$ selecting a specific antiderivative.
For instance, considering the additive decomposition $f=\Der{g}+h$ of rational functions $(\mathcal{C}(x),\frac{d}{dx})$ as mentioned above, a natural choice for $g=\Q{f}$ is such that the partial fraction decomposition has no constant term.
In Section~\ref{sec:free}, we construct the free integro-differential ring $\mathrm{IDR}(\mathcal{R})$ on the differential ring $\mathcal{R}$ and prove the following main results characterizing this free object:
\begin{enumerate}
 \item $\mathrm{IDR}(\mathcal{R})$ is an integro-differential ring and $\mathcal{R}$ can be embedded into it as a differential subring, see Theorem~\ref{thm:embedding}.
 \item $\mathrm{IDR}(\mathcal{R})$ has the universal property: any differential ring homomorphism from $\mathcal{R}$ to an integro-differential ring $\mathcal{S}$ can be uniquely extended to an integro-differential homomorphism from $\mathrm{IDR}(\mathcal{R})$ to $\mathcal{S}$, see Theorem~\ref{thm:universal}.
\end{enumerate}
In this sense, $\mathrm{IDR}(\mathcal{R})$ is the (external) integro-differential closure of $\mathcal{R}$.
While the construction is similar to the one in \cite{GuoRegensburgerRosenkranz}, it is considerably more complicated due to the more general properties of integro-differential rings defined in \cite{RaabRegensburger}, see also Remark~\ref{rem:IDA}.
In particular, integration on $\mathrm{IDR}(\mathcal{R})$ is not a Rota--Baxter operator but satisfies the Rota--Baxter identity with evaluation \eqref{eq:mRB}.
In Section~\ref{sec:prelim}, we recall this definition of integro-differential rings and the generalized shuffle relations satisfied by nested integrals in them.
Moreover, we exhibit relations among evaluations of products of nested integrals that hold in any integro-differential ring, see Corollary~\ref{cor:IDRConstRel}.

Related to the general properties of integro-differential rings mentioned, we investigate, in Section~\ref{sec:shufflealgebra}, the generalized shuffle relations in $\mathrm{IDR}(\mathcal{R})$ and, in Section~\ref{sec:constants}, the relations among evaluations.
In particular, Corollary~\ref{cor:basis} allows us to translate an algebra basis of the shuffle algebra to an algebra basis of $\mathrm{IDR}(\mathcal{R})$ and Corollary~\ref{cor:gensC2} shows that certain evaluations of products of nested integrals are sufficient to express all other such evaluations.
Moreover, in Section~\ref{sec:independence}, Theorem~\ref{thm:independence} provides a criterion for the linear independence of nested integrals.

However, the integration operation of the free object $\mathrm{IDR}(\mathcal{R})$ inherently cannot yield any antiderivatives that already exist within $\mathcal{R}$.
Therefore, in Section~\ref{sec:ExtQIDR}, we give an adapted construction in order to also extend a given quasi-integration $\Q$ on $\Der\mathcal{R}$ to the integro-differential closure $\mathrm{IDR}_\Q(\mathcal{R})$.
This leads to Definition~\ref{def:quasiIDR} introducing the notion of quasi-integro-differential rings.
The integro-differential closure $\mathrm{IDR}_\Q(\mathcal{R})$ retaining integrals given by $\Q$ satisfies properties analogous to $\mathrm{IDR}(\mathcal{R})$ above, see Theorem~\ref{thm:IDRext} and Corollary~\ref{cor:universalExt}.
For example, the free integro-differential ring $\mathrm{IDR}(\mathcal{C}(x))$ necessarily has the somewhat unnatural property that $\Int1 \not\in \mathcal{C}(x)$.
In contrast, $\mathrm{IDR}_\Q(\mathcal{C}(x))$ satisfies $\Int1 = \Q1 \in \mathcal{C}(x)$, e.g.\ $\Q1=x$ for the particular $\Q$ mentioned above.
Similarly, this property is relevant for $\Q$ obtained in symbolic computation (see the references given above) by additive decompositions $f=\Der\Q{f}+h$ in other differential rings.

Sometimes, one would like to completely avoid the introduction of new constants when extending the given differential ring $(\mathcal{R},\Der)$.
To this end, we provide a greatly simplified construction of the external integro-differential closure without new constants in Section~\ref{sec:NNC}.

Finally, for a given integro-differential ring $\mathcal{S}$ that contains $\mathcal{R}$, we study in Section~\ref{sec:closure} the internal integro-differential closure of $\mathcal{R}$ in $\mathcal{S}$ as a quotient of the free object $\mathrm{IDR}(\mathcal{R})$, respecting the constants arising from evaluations of nested integrals in $\mathcal{S}$.
In particular, Theorem~\ref{thm:closure} provides under general conditions the constants that need to be factored out.
This is then applied in Example~\ref{ex:hyperlogs} to obtain the ring of hyperlogarithms as quotient of the free integro-differential ring $\mathrm{IDR}(\mathcal{C}(x))$ over rational functions.

\paragraph{Conventions and notation}
Throughout the paper, rings and algebras are implicitly assumed to have a unit element unless otherwise stated.
Consequently, maps between them are only considered homomorphisms if they respect the unit element and modules are assumed to be unitary.
Moreover, integro-differential rings refer to Definition~\ref{def:IDR} below, which was introduced in \cite{RaabRegensburger} and generalizes the definition given in \cite{JSC2018} by dropping the requirement that the evaluation is multiplicative.

Multiplication in given rings $\mathcal{C},\mathcal{R},\mathcal{S}$ and scalar multiplication by constants from $\mathcal{C}$ are denoted by concatenation, all multiplications in constructed rings like tensor, symmetric, and shuffle products $\otimes,\odot,\shuffle$ as well as other multiplications $\cdot,\circledast$ always use the multiplication symbol.
We do not assume any precedence order among different products, so we always use parentheses whenever two different products come together.
For the special linear maps of differentiation $\Der$, integration $\Int$, evaluation $\E$, and quasi-integration $\Q$, we use operator notation except when the argument is the result of a sum or product using an infix symbol.
E.g.\ the Leibniz rule for the derivative of products then reads $\Der{fg}=(\Der{f})g+f\Der{g}$ and the derivative of a tensor product reads $\Der(f_0\otimes{c}\otimes{f})$.

\section{Basics of integro-differential rings}
\label{sec:prelim}

For a differential ring $(\mathcal{R},\Der)$, with derivation $\Der$ satisfying the Leibniz rule
\[
 \Der{fg}=(\Der{f})g+f\Der{g},
\]
we denote its ring of constants by $\const_\Der(\mathcal{R})=\ker(\Der)$.
If the derivation is surjective, every element of $\mathcal{R}$ has an integral in $\mathcal{R}$ and we have the following notions.
\begin{definition}\cite{RaabRegensburger}\label{def:IDR}
 Let $(\mathcal{R},\Der)$ be a differential ring and let $\mathcal{C}:=\const_\Der(\mathcal{R})$.
 A $\mathcal{C}$-linear map $\Int\colon\mathcal{R}\to\mathcal{R}$ is called an \emph{integration} on $\mathcal{R}$ w.r.t.\ $\Der$, if
 \begin{equation}\label{eq:SectionAxiom}
  \Der\Int = \id,
 \end{equation}
 and we call $(\mathcal{R},\Der,\Int)$ an \emph{integro-differential ring}.
 On an integro-differential ring, we define the \emph{(induced) evaluation} $\E:\mathcal{R}\to\mathcal{C}$ by
 \begin{equation}
 \label{eq:EvaluationDef}
  \E := \id-\Int\Der.
 \end{equation}
\end{definition}
Note that the above definition does not require $\mathcal{R}$ or $\mathcal{C}$ to be commutative, so a $\mathcal{C}$-linear map is to be understood as a $\mathcal{C}$-bimodule homomorphism in general.
Properties, examples, and applications of integro-differential rings $(\mathcal{S},\Der,\Int)$ were studied in \cite{RaabRegensburger}.
For example, from Eq.~(18) in that paper, it follows that integration by parts
\begin{equation}\label{eq:IBP}
 \Int(\Der{f})g = fg-\E{fg}-\Int{f}\Der{g}
\end{equation}
holds for all $f,g\in\mathcal{S}$ and Theorem~3.1 in that paper shows that the product of two integrals satisfies the \emph{Rota--Baxter identity with evaluation}
\begin{equation}\label{eq:mRB}
 (\Int{f})\Int{g} = \Int{f}\Int{g} + \Int(\Int{f})g + \E(\Int{f})\Int{g}.
\end{equation}
In Example~2.13 of \cite{RaabRegensburger}, the following basic examples for an integro-differential ring with non-multiplicative evaluation \eqref{eq:EvaluationDef} are given.
\begin{example}\label{ex:IDR}
 Let $\mathcal{C}$ a commutative ring with $\mathbb{Q}\subseteq\mathcal{C}$, i.e.\ $\mathbb{Q}$ is a subring of $\mathcal{C}$, and let $\mathcal{S}=\mathcal{C}[x,x^{-1},\ln(x)]$ or $\mathcal{S}=\mathcal{C}((x))[\ln(x)]$.
 The usual derivation $\Der=\frac{d}{dx}$ is surjective on $\mathcal{S}$.
 An integration $\Int$ on $\mathcal{S}$ was given by a recursive formula in \cite{RaabRegensburger} so that $(\mathcal{S},\Der,\Int)$ is an integro-differential ring.
 The evaluation defined by \eqref{eq:EvaluationDef} is the $\mathcal{C}$-linear map $\E$ acting as
 \begin{equation}\label{eq:LaurentEval}
  \E x^k\ln(x)^n = \begin{cases}1&k=n=0\\0&\text{otherwise}\end{cases},
 \end{equation}
 which is obviously not multiplicative on $\mathcal{S}$.
 The integration $\Int$ can also be understood via the alternative characterization given in \cite[Thm.~2.4]{RaabRegensburger}.
 Namely, if the $\mathcal{C}$-linear map $\E$ is defined by \eqref{eq:LaurentEval} instead, then the same integration $\Int$ can be defined in terms of this $\E$ via $\Int{f}:=g-\E{g}$, where $g\in\mathcal{S}$ is arbitrary such that $\Der{g}=f$.
 In \eqref{eq:mRB}, the evaluation term cannot be omitted in general, as illustrated by choosing $f=1$ and $g=x^{-2}$ leading to $\E(\Int{f})\Int{g}=\E{x}(-x^{-1})=-1$.
\end{example}

The following lemma shows that integro-differential rings can be characterized entirely by equations that hold for all elements of the ring.
In other words, the class of integro-differential rings forms a variety.

\begin{lemma}\label{lem:IDRvariety}
 Let $(\mathcal{R},\Der)$ be a differential ring with constants $\mathcal{C}$.
 Let $\Int:\mathcal{R}\to\mathcal{R}$ be a map such that $\Der\Int=\id$ and define $\E:=\id-\Int\Der$.
 Then, $\Int$ is $\mathcal{C}$-linear if and only if it is additive and satisfies $(\E{g})\Int{f}=\Int(\E{g})f$ and $(\Int{f})\E{g}=\Int{f}\E{g}$ for all $f,g\in\mathcal{R}$.
\end{lemma}
\begin{proof}
 If $\Int$ is $\mathcal{C}$-linear, then the other properties immediately follow from $\Der\E=0$.
 Conversely, $\mathcal{C}$-linearity follows since $\E{g}=g$ holds for $g\in\mathcal{C}$.
\end{proof}

In a commutative integro-differential ring, products of nested integrals can be rewritten as linear combinations of nested integrals by iteratively applying \eqref{eq:mRB}.
For convenient notation of nested integrals and their products, it is standard to represent them in terms of pure tensors and their shuffle products.
For a module $M$ over a commutative ring $\mathcal{C}$, recall that the shuffle product $\shuffle$ on the $\mathcal{C}$-tensor algebra
\[
 T(M)=\bigoplus_{n=0}^{\infty}M^{\otimes n}
\]
can be recursively defined for pure tensors $t \in M^{\otimes n}$ and $s \in M^{\otimes m}$ by
\begin{equation}\label{eq:shufflerec}
 t\shuffle{s} := \begin{cases}
  t \otimes s & n=0 \vee m=0\\
  t_1 \otimes \left((t_2\otimes\dots\otimes{t_n}) \shuffle s\right) + s_1 \otimes \left(t \shuffle (s_2\otimes\dots\otimes{s_m})\right) & n,m\ge1
 \end{cases}
\end{equation}
and extends to arbitrary $t,s \in T(M)$ by biadditivity.
It satisfies $t \shuffle s = s \shuffle t$ and $c(t \shuffle s) = (ct) \shuffle s$ for $c\in\mathcal{C}$ and has the empty tensor $1_{T(M)}$ as unit element.
Moreover, $\shuffle$ is associative and satisfies the distributive law.
Hence, it turns the $\mathcal{C}$-module $T(M)$ into a commutative $\mathcal{C}$-algebra.
Moreover, like the tensor algebra, the shuffle algebra on $M$ is a graded algebra, since, for any $t \in M^{\otimes n}$ and $s \in M^{\otimes m}$, we have $t \shuffle s \in M^{\otimes(n+m)}$.

In the following, let $(\mathcal{S},\Der,\Int)$ be a commutative integro-differential ring with constants $\mathcal{C}$.
Nested integrals can be represented via the $\mathcal{C}$-linear map $\sigma: T(\mathcal{S})\to\mathcal{S}$ defined on pure tensors $f = f_1\otimes{f_2}\otimes\dots\otimes{f_n} \in T(\mathcal{S})$ by
\[
 \sigma(f) := \Int{f_1}\Int{f_2}\dots\Int{f_n} \in \mathcal{S}.
\]
For such pure tensors, we denote shortened versions by 
\[
 f_i^j := f_i\otimes\dots\otimes{f_j} \in T(\mathcal{S})
\]
for any $i \in \{1,\dots,n+1\}$ and $j \in \{0,\dots,n\}$ with $i\le{j+1}$.
The empty tensor arises as $f_i^j=1_{T(\mathcal{S})}$, if $i=j+1$, and we have $\sigma(1_{T(\mathcal{S})})=1$.
For example, if $\mathcal{S}=\mathcal{C}((x))[\ln(x)]$ is as in Example~\ref{ex:IDR}, then we may obtain the series of the dilogarithm $\mathrm{Li}_2(x)=\sum_{k=1}^\infty\frac{x^k}{k^2}$ by $\sigma(\frac{1}{x}\otimes\frac{1}{1-x})=\Int\frac{1}{x}\Int\frac{1}{1-x}$ in $\mathcal{S}$, which corresponds to the analytic representation of the dilogarithm as nested integral $\mathrm{Li}_2(x)=\int_0^x\frac{1}{t}\int_0^t\frac{1}{1-u}\,du\,dt$.

To express products of nested integrals, we also define the map $e:T(\mathcal{S})\times{T(\mathcal{S})}\to\mathcal{C}$ by
\[
 e(f,g):=\E\sigma(f)\sigma(g)
\]
for shorter notation of the constants involved.
By $\mathcal{C}$-linearity of $\sigma$ and $\E$, it follows that $e$ is $\mathcal{C}$-linear in each argument.
Note that, the evaluation $\E\Int{f}$ of the integral of any $f \in \mathcal{S}$ is zero by definition.
Hence, $\E$ is multiplicative if and only if $e(f,g)=0$ holds whenever $f$ or $g$ lies in $\bigoplus_{n=1}^\infty\mathcal{S}^{\otimes n}$.
In \cite[Thm.~3.3]{RaabRegensburger}, it was shown that the product of nested integrals in $\mathcal{S}$ can be expressed via the shuffle product as
\begin{equation}\label{eq:IDRShuffle}
 \sigma(f)\sigma(g) = \sigma(f\shuffle{g})+\sum_{i=0}^{n-1}\sum_{j=0}^{m-1}e\!\left(f_{i+1}^n,g_{j+1}^m\right)\sigma(f_1^i\shuffle{g_1^j}),
\end{equation}
where $f,g \in T(\mathcal{S})$ are pure tensors of length $n$ and $m$, respectively.
For multiplicative evaluation $\E$, the evaluation terms vanish and \eqref{eq:IDRShuffle} yields the standard shuffle relations $\sigma(f)\sigma(g) = \sigma(f\shuffle{g})$ of nested integrals given in \cite{Ree}, see also \cite{GuoBook}.
For illustration, we instantiate \eqref{eq:IDRShuffle} with tensors $f$ and $g$ of length $1$, whereby we obtain explicitly $\sigma(f)\sigma(g)=\sigma(f\otimes{g})+\sigma(g\otimes{f})+e(f,g)\sigma(1_{T(\mathcal{S})})$.
In particular, if $\mathcal{S}=\mathcal{C}((x))[\ln(x)]$ is as in Example~\ref{ex:IDR}, we can set $f=\frac{1}{x^2}$ and $g=\frac{1}{1-x}$, for example, and compute these terms as $\sigma(f)\sigma(g)=\frac{1}{x}\ln(1-x)$, $\sigma(f\otimes{g})=\ln(x)+(\frac{1}{x}-1)\ln(1-x)+1$, $\sigma(g\otimes{f})=\ln(1-x)-\ln(x)$, and $e(f,g)=-1$.
In general, \eqref{eq:IDRShuffle} allows us to immediately deduce the following relations among constants that arise as evaluations of products of nested integrals.

\begin{corollary}\label{cor:IDRConstRel}
 Let $(\mathcal{S},\Der,\Int)$ be a commutative integro-differential ring with constants $\mathcal{C}$.
 Let $f,g,h \in T(\mathcal{S})$ be pure tensors of length $n$, $m$, and $l$, respectively.
 Then, we have
 \begin{multline}\label{eq:IDRConstRel}
  e(f\shuffle{g},h)-e(f,g\shuffle{h})\,+\\
  +\sum_{j=0}^{m-1}\Bigg(\sum_{i=0}^{n-1}e\!\left(f_{i+1}^n,g_{j+1}^m\right)e\!\left(f_1^i\shuffle{g_1^j},h\right)
  -\sum_{k=0}^{l-1}e\!\left(f,g_1^j\shuffle{h_1^k}\right)e\!\left(g_{j+1}^m,h_{k+1}^l\right)\!\Bigg)=0.
\end{multline}
\end{corollary}
\begin{proof}
 We compute the constant $\E\sigma(f)\sigma(g)\sigma(h)$ in two ways.
 First, we use \eqref{eq:IDRShuffle} to obtain
 \begin{align*}
  \E\big(\sigma(f)\sigma(g)\big)\sigma(h) &= \E\!\left(\sigma(f\shuffle{g})+\sum_{i=0}^{n-1}\sum_{j=0}^{m-1}e\!\left(f_{i+1}^n,g_{j+1}^m\right)\sigma(f_1^i\shuffle{g_1^j})\right)\!\sigma(h)\\
  &= e(f\shuffle{g},h)+\sum_{i=0}^{n-1}\sum_{j=0}^{m-1}e\!\left(f_{i+1}^n,g_{j+1}^m\right)e\!\left(f_1^i\shuffle{g_1^j},h\right)
 \end{align*}
 by definition and $\mathcal{C}$-linearity of $e$. 
 Analogously, if we instantiate \eqref{eq:IDRShuffle} to express $\sigma(g)\sigma(h)$, we obtain
 \[
  \E\sigma(f)\big(\sigma(g)\sigma(h)\big) = e(f,g\shuffle{h}) + \sum_{j=0}^{m-1}\sum_{k=0}^{l-1}e\!\left(f,g_1^j\shuffle{h_1^k}\right)e\!\left(g_{j+1}^m,h_{k+1}^l\right).
 \]
 Altogether, taking the difference of these two expressions of $\E\sigma(f)\sigma(g)\sigma(h)$ yields \eqref{eq:IDRConstRel}.
\end{proof}

\section{The free commutative integro-differential ring on a commutative differential ring}
\label{sec:free}

To construct the free commutative integro-differential ring on a commutative differential ring $(\mathcal{R},\Der)$, we assume that, in $\mathcal{R}$, the ring of constants $\mathcal{C}:=\const_\Der(\mathcal{R})$ as well as the {\lq\lq}integrable{\rq\rq} elements $\Der\mathcal{R}$ are complemented $\mathcal{C}$-submodules of $\mathcal{R}$.
In the language of \cite[Def.~4.4]{GuoRegensburgerRosenkranz}, this is equivalent to $(\mathcal{R},\Der)$ being a regular differential $\const_\Der(\mathcal{R})$-algebra.
In this case, a $\mathcal{C}$-linear reflexive generalized inverse of the derivation $\Der$ exists.
In particular, if $\mathcal{C}$ is a semisimple ring (e.g.\ a field), then such a generalized inverse always exists.
This is because modules over semisimple rings are semisimple, i.e.\ every submodule has a direct complement, see e.g.\ Theorems~4.3.4 and 5.2.7 in \cite{CohnBasic}.
\begin{definition}\label{def:quasiintegration}
 Let $(\mathcal{R},\Der)$ be a differential ring with ring of constants $\mathcal{C}$.
 If $\Q:\mathcal{R}\to\mathcal{R}$ is a $\mathcal{C}$-module endomorphism with
 \[
  \Der\Q\Der=\Der \quad\text{and}\quad \Q\Der\Q=\Q,
 \]
 then we call it a \emph{quasi-integration} on $(\mathcal{R},\Der)$.
\end{definition}
In fact, every quasi-integration $\Q$ gives rise to two direct decompositions of $\mathcal{R}$ into $\mathcal{C}$-modules, namely $\mathcal{R}=\mathcal{C}\oplus\mathcal{R}_\mathrm{J}$ with $\mathcal{R}_\mathrm{J}:=\Q\mathcal{R}$ and
\[
 \mathcal{R}=(\Der\mathcal{R})\oplus\mathcal{R}_\T\quad\text{with}\quad \mathcal{R}_\T:=\ker(\Q)
\]
so that $\mathcal{R}_\T$ is a module of {\lq\lq}non-integrable{\rq\rq} elements, following the notation in \cite{GuoRegensburgerRosenkranz}.
Conversely, every pair of direct complements in $\mathcal{R}$ of the $\mathcal{C}$-modules $\mathcal{C}$ and $\Der\mathcal{R}$ determines a unique quasi-integration with these complements as its image and kernel.

\begin{example}\label{ex:Q}
 We illustrate the above notion of quasi-integration using the differential field $(\mathcal{R},\Der)=(\mathcal{C}(x),\frac{d}{dx})$ of rational functions over a field $\mathcal{C}$ of characteristic zero.
 As indicated in the introduction, we define the $\mathcal{C}$-linear map $\Q:\mathcal{R}\to\mathcal{R}$ such that elements in the direct complement
 \[
  \mathcal{R}_\mathrm{J}=\Q\mathcal{R}=\big\{p+\tfrac{f}{g}\ \big|\ p,f,g \in C[x], p(0)=0, \deg(f)<\deg(g)\big\}
 \]
 of $\mathcal{C}$ do not involve a constant term in their partial fraction decomposition and the direct complement
 \[
  \mathcal{R}_\T=\ker(\Q)=\big\{\tfrac{f}{g}\ \big|\ f,g \in C[x], \deg(f)<\deg(g), g \text{ is squarefree}\big\}
 \]
 of $\Der\mathcal{R}$ consists of all proper rational functions having only simple poles, which are non-integrable in $\mathcal{R}$.
 For instance, we obtain
 \[
  \Q\tfrac{x}{(x+2)^2}=\tfrac{2}{x+2},
 \]
 since we have the partial fraction decomposition $\tfrac{x}{(x+2)^2}=\tfrac{1}{x+2}-\tfrac{2}{(x+2)^2}$ and $\Q\tfrac{1}{x+2}=0$.
\end{example}

\begin{remark}\label{rem:freeonset}
 Free algebras (in the sense of universal algebra) on a set always exist in varieties, see e.g.\ \cite[Prop.~1.3.6]{CohnFurther}.
 In particular, the free integro-differential ring on a set exists, as does the free differential ring.
 To apply our construction of the free commutative integro-differential ring laid out below, we first need a commutative differential ring $(\mathcal{R},\Der)$ where both constants $\mathcal{C}$ and derivatives $\Der\mathcal{R}$ form complemented $\mathcal{C}$-submodules.
 The free commutative differential ring $(\mathcal{R},\Der)=(\mathbb{Z}\{X\},\Der)$ on a set $X$ has constants $\mathcal{C}=\mathbb{Z}$ and is given by the differential polynomials in $X$ with coefficients in $\mathbb{Z}$.
 However, $\Der\mathcal{R}$ is not a complemented $\mathbb{Z}$-module in $\mathcal{R}=\mathbb{Z}\{X\}$, since for $x \in X$ and $g,r\in\mathcal{R}$ with $x\Der{x}=\Der{g}+r$ it follows that $r\not\in\Der\mathcal{R}$ and $2r=\Der(x^2-2g)$.
 Instead, we can apply our construction to $(\mathcal{R},\Der)=(\mathbb{Q}\{X\},\Der)$, which has constants $\mathcal{C}=\mathbb{Q}$.
 In fact, for any commutative $\mathbb{Q}$-algebra $\mathcal{K}$, $\mathcal{K}\{X\}=\mathcal{K}\otimes_\mathbb{Q}\mathbb{Q}\{X\}$ inherits the complement of $\im(\Der)$ from $\mathbb{Q}\{X\}$, which allows one to use our construction for obtaining the free commutative integro-differential $\mathcal{K}$-algebra $\mathrm{IDR}(\mathcal{K}\{X\})$ on a set $X$ over such a $\mathcal{K}$.
\par
 We note that the survey \cite{GaoGuoSurvey} states an explicit complement of $\im(\Der)$ in $\mathcal{K}\{X\}$ omitting the necessary requirement that $\mathcal{K}$ is a $\mathbb{Q}$-algebra.
 Similarly, this condition would be needed in \cite{GaoGuoZheng,GaoGuoRosenkranz} for constructing the free (non)commutative integro-differential $\mathcal{K}$-algebra (with multiplicative evaluation) on a set $X$ via Gr\"{o}bner--Shirshov bases, since generators of the ideal need to be made monic.
\end{remark}

For the rest of this section, we fix a commutative differential ring $(\mathcal{R},\Der)$ with constants $\mathcal{C}$ and a quasi-integration $\Q$ on $(\mathcal{R},\Der)$.
First, in Section~\ref{sec:AlgebraStructure}, we construct the free commutative integro-differential ring $\mathrm{IDR}(\mathcal{R})$ as commutative $\mathcal{C}$-algebra only, introducing building blocks that are relevant for all that follows.
Then, in Section~\ref{sec:DerInt}, we define a derivation and integration and show that this algebra is indeed an integro-differential ring.
In Section~\ref{sec:universal}, we show that it is freely generated by $(\mathcal{R},\Der)$.
Hence, using different quasi-integrations in the construction will yield isomorphic integro-differential rings for the same $(\mathcal{R},\Der)$.
After that, we look at simple examples in Section~\ref{sec:examples} where $\mathcal{R}_\T$ is cyclic, like in the case of Laurent series.
We investigate the close relation to the shuffle algebra in Section~\ref{sec:shufflealgebra}.
In particular, we show that the basis of the shuffle algebra constructed in \cite{Radford} can be carried over to yield a basis of $\mathrm{IDR}(\mathcal{R})$.
Finally, we investigate the structure of the relations satisfied by new constants within $\mathrm{IDR}(\mathcal{R})$ in Section~\ref{sec:constants}, whereby we identify a smaller generating set for these constants if $\mathbb{Q}\subseteq\mathcal{C}$.
If $\mathcal{R}_\T$ is a free module, we conjecture that these generators are in fact algebraically independent over $\mathcal{C}$, which we verify for certain subsets of these generators.

\subsection{Construction as $\mathcal{C}$-algebra}
\label{sec:AlgebraStructure}

Definition~\ref{def:DefIDR} below constructs the set $\mathrm{IDR}(\mathcal{R})$ and its $\mathcal{C}$-algebra structure via the $\mathcal{C}$-tensor product of $\mathcal{R}$ with other $\mathcal{C}$-algebras.
These new $\mathcal{C}$-algebras are constructed from the ingredients $\mathcal{C},\mathcal{R}_\mathrm{J},\mathcal{R}_\T\subseteq\mathcal{R}$.
Not only the notation introduced below, but also the meaning of the new algebras that will be implied by the derivation and integration introduced later in Section~\ref{sec:DerInt}, will be relevant for the rest of the paper.
So, we briefly explain the intuitive role now.
Pure tensors in $T(\mathcal{R}_\T)$ represent new nested integrals in $\mathrm{IDR}(\mathcal{R})$.
New constants arise from evaluation of products of elements of $\mathcal{R}_\mathrm{J}$ with such nested integrals, which we model in the $\mathcal{C}$-algebra $\mathcal{C}_1$ defined in \eqref{eq:DefC1}.
Similarly, in \eqref{eq:DefC2}, we define the $\mathcal{C}$-algebra $\mathcal{C}_2$ to model new constants arising from evaluation of products of nested integrals in $\mathrm{IDR}(\mathcal{R})$.
Following Corollary~\ref{cor:IDRConstRel}, we impose relations \eqref{eq:ConstRel} in $\mathcal{C}_2$.

Define the $\mathcal{C}$-module
\begin{equation}\label{eq:DefTensors}
 \mathcal{T} := T(\mathcal{R}_\T) = \bigoplus_{n=0}^\infty \mathcal{R}_\T^{\otimes{n}}
\end{equation}
and endow it with two different products: the shuffle product $\shuffle$ as well as the tensor product $\otimes$ over $\mathcal{C}$.
Either of these products turns $\mathcal{T}$ into a $\mathcal{C}$-algebra with the empty tensor being the unit element $\et$.
Denote
\begin{equation}\label{eq:DefTtilde}
 \widetilde{\mathcal{T}} := \bigoplus_{n=1}^\infty \mathcal{R}_\T^{\otimes{n}}
\end{equation}
the subalgebra without empty tensors (and hence does not typically include the unit element $\et$).
Recall that, for a pure tensor $f = f_1\otimes\dots\otimes{f_n} \in \mathcal{T}$, we abbreviate
\[
 f_i^j := f_i\otimes\dots\otimes{f_j} \in \mathcal{T}
\]
for any $i \in \{1,\dots,n+1\}$ and $j \in \{0,\dots,n\}$ with $i\le{j+1}$, i.e.\ $f_i^j=\et$ for $i=j+1$.
Also define the $\mathcal{C}$-modules
\begin{align}
 M_1 &:= \mathcal{R}_\mathrm{J}\otimes\mathcal{T},\label{eq:DefM1}\\
 M_2 &:= \widetilde{\mathcal{T}}^{\odot2},\label{eq:DefM2}
\end{align}
where $\widetilde{\mathcal{T}}^{\odot2}$ denotes the symmetric product of $\widetilde{\mathcal{T}}$ with itself over $\mathcal{C}$ and can be obtained as the quotient $\widetilde{\mathcal{T}}^{\otimes2}/\linspan_\mathcal{C}\{f\otimes{g}-g\otimes{f} \mid f,g\in \widetilde{\mathcal{T}}\}$.
Note that $\mathcal{R}_\mathrm{J}$ and $\mathcal{R}_\T$ may have a non-trivial intersection.
Hence, for non-zero $f \in \mathcal{R}_\mathrm{J}\cap\mathcal{R}_\T$ and $g \in \widetilde{\mathcal{T}}$, it is important to distinguish $f\otimes{g} \in M_1$ from $f\odot{g} \in M_2$.
To explicitly illustrate the intuitive role of elements of $\mathcal{T},M_1,M_2$ defined above, we briefly consider the concrete setting given in Example~\ref{ex:Q}, which in particular has $\mathcal{R}_\T \subseteq \mathcal{R}_\mathrm{J}$.
The element $g = \tfrac{1}{x}\otimes\tfrac{1}{1-x} \in \mathcal{T}$ can be thought of as nested integral $\Int\tfrac{1}{x}\Int\tfrac{1}{1-x}$.
With $f = \tfrac{1}{x+2} \in \mathcal{R}_\T$, the element $f\otimes{g} \in M_1$ is supposed to represent the product $\tfrac{1}{x+2}\Int\tfrac{1}{x}\Int\tfrac{1}{1-x}$, while the element $f\odot{g} \in M_2$ represents the product $\big(\Int\tfrac{1}{x+2}\big)\Int\tfrac{1}{x}\Int\tfrac{1}{1-x}$ of two nested integrals.

In the general case, the modules $M_1$ and $M_2$ are only used to introduce new constants in our construction, which is reflected in the notation as follows.
Let $\ep(M_1)$ and $\ep(M_2)$ be isomorphic copies of the $\mathcal{C}$-modules $M_1$ and $M_2$, respectively, with isomorphisms $\ep:M_i\to\ep(M_i)$.
Now, we define
\begin{equation}\label{eq:DefC1}
 \mathcal{C}_1:=\Sym(\ep(M_1))
\end{equation}
as the symmetric $\mathcal{C}$-algebra with unit generated by $\ep(M_1)$ and define
\begin{equation}\label{eq:DefC2}
 \mathcal{C}_2:=\Sym(\ep(M_2))/J_2
\end{equation}
as the quotient of the symmetric $\mathcal{C}$-algebra with unit on $\ep(M_2)$ by the ideal $J_2 \subseteq \Sym(\ep(M_2))$ generated by the following relations among elements of $\ep(M_2)$:
\begin{multline}\label{eq:ConstRel}
 \ep((f\shuffle{g})\odot{h})-\ep(f\odot(g\shuffle{h}))\,+\\
 +\sum_{j=0}^{m-1}\Bigg(\sum_{i=\max(0,1-j)}^{n-1}\ep\!\left(f_{i+1}^n\odot{g_{j+1}^m}\right)\!{\cdot}\ep\!\left((f_1^i\shuffle{g_1^j})\odot{h}\right)-\\
 -\sum_{k=\max(0,1-j)}^{l-1}\ep\!\left(f\odot(g_1^j\shuffle{h_1^k})\right)\!{\cdot}\ep\!\left(g_{j+1}^m\odot{h_{k+1}^l}\right)\!\Bigg)
\end{multline}
for all $n,m,l \in \{1,2,\dots\}$ and all pure tensors $f \in \mathcal{R}_\T^{\otimes{n}}$, $g \in \mathcal{R}_\T^{\otimes{m}}$, and $h \in \mathcal{R}_\T^{\otimes{l}}$.

Next, we equip $\mathcal{C}_2\otimes\mathcal{T}$ with the following $\mathcal{C}$-bilinear map.
\begin{definition}\label{def:GeneralizedShuffle}
 With $\mathcal{C}_2$ and $\mathcal{T}$ defined as above, we set
\begin{equation}\label{eq:GeneralizedShuffle}
 (c\otimes{f})\cdot(d\otimes{g}) := (c{\cdot}d)\otimes(f\shuffle{g})+\sum_{i=0}^{n-1}\sum_{j=0}^{m-1}\big(c{\cdot}d{\cdot}\ep(f_{i+1}^n\odot{g_{j+1}^m})\big)\otimes(f_1^i\shuffle{g_1^j})
\end{equation}
for all $c,d \in \mathcal{C}_2$, $n,m \in \{0,1,2,\dots\}$, and all pure tensors $f \in \mathcal{R}_\T^{\otimes{n}}$ and $g \in \mathcal{R}_\T^{\otimes{m}}$, and extend this to multiplication of arbitrary elements of $\mathcal{C}_2\otimes\mathcal{T}$ by $\mathcal{C}$-bilinearity.
\end{definition}

We will prove below in Section~\ref{sec:multiplication} that $\mathcal{C}_2\otimes\mathcal{T}$ with this bilinear map is indeed a commutative $\mathcal{C}$-algebra.
The close relation of this algebra to the shuffle algebra will be detailed later in Section~\ref{sec:shufflealgebra}.
Note that the definition of the algebra $\mathcal{C}_2\otimes\mathcal{T}$ only depends on the $\mathcal{C}$-module $\mathcal{R}_\T$.
Finally, we can define the free commutative integro-differential ring as commutative $\mathcal{C}$-algebra as follows.

\begin{definition}\label{def:DefIDR}
 We define the commutative $\mathcal{C}$-algebra
 \begin{equation}\label{eq:DefIDR}
  \mathrm{IDR}(\mathcal{R}):=\mathcal{R}\otimes\mathcal{C}_1\otimes\mathcal{C}_2\otimes\mathcal{T}
 \end{equation}
 as the $\mathcal{C}$-tensor product of the commutative $\mathcal{C}$-algebras $\mathcal{R}$, $\mathcal{C}_1$, and $\mathcal{C}_2\otimes\mathcal{T}$ defined above.
\end{definition}
Once more, for non-zero $f \in \mathcal{R}_\mathrm{J}\cap\mathcal{R}_\T$ and $g \in \widetilde{\mathcal{T}}$, it is important to distinguish $\ep(f\otimes{g}) \in \mathcal{C}_1$, $\ep(f\odot{g}) \in \mathcal{C}_2$, and $f\otimes{g} \in \mathcal{T}$ in three different tensor factors of $\mathrm{IDR}(\mathcal{R})$.
We also denote the canonical $\mathcal{C}$-algebra homomorphism $\mathcal{R}\to\mathrm{IDR}(\mathcal{R})$ by
\begin{equation}\label{eq:DefIota}
 \iota(f)=f\otimes1\otimes1\otimes\et.
\end{equation}

\subsubsection{Multiplication on $\mathcal{C}_2\otimes\mathcal{T}$}
\label{sec:multiplication}
The $\mathcal{C}$-bilinear map defined by the explicit formula \eqref{eq:GeneralizedShuffle} also satisfies a recursive relation w.r.t.\ $n+m$ based on \eqref{eq:shufflerec}.
To express it in a concise way, let us define a second multiplication on $\mathcal{C}_2\otimes\mathcal{T}$ as the tensor product of the multiplication on $\mathcal{C}_2$ with the multiplication $\otimes$ on $\mathcal{T}$, i.e.
\begin{equation}\label{eq:tensormultiplication}
 (c\otimes{f})\circledast(d\otimes{g}):=(c{\cdot}d)\otimes(f\otimes{g}).
\end{equation}
Note that $\circledast$ is not commutative, in contrast to the bilinear map defined in Definition~\ref{def:GeneralizedShuffle}.
\begin{lemma}\label{lem:recursivemult}
 If \eqref{eq:GeneralizedShuffle} holds for all $c,d \in \mathcal{C}_2$, $n,m \in \{0,1,2,\dots\}$, and all pure tensors $f\in\mathcal{R}_\T^{\otimes{n}}$ and $g\in\mathcal{R}_\T^{\otimes{m}}$, then we also have the recursion
 \begin{equation}\label{eq:GeneralizedShuffleRec2}
  (c\otimes{f})\cdot(d\otimes{g}) =
  \begin{cases}(c\otimes{f})\circledast(d\otimes{g})&n=0\vee{m=0}\\[\smallskipamount]
  \big((c{\cdot}d)\otimes{f_1}\big)\circledast\big((1\otimes{f_2^n})\cdot(1\otimes{g})\big)\,+\\\quad+\,\big((c{\cdot}d)\otimes{g_1}\big)\circledast\big((1\otimes{f})\cdot(1\otimes{g_2^m})\big)\,+&n>0\wedge{m>0}\\\quad+\,(c{\cdot}d{\cdot}\ep(f\odot{g}))\otimes\et
  \end{cases}.
 \end{equation}
\end{lemma}
\begin{proof}
 If $n=0$, then $f=e\et$ for some $e \in \mathcal{C}$ and we have \eqref{eq:GeneralizedShuffleRec2} by $(c\otimes{f})\circledast(d\otimes{g})=(c{\cdot}d{\cdot}e)\otimes{g}$ and $(c\otimes{f})\cdot(d\otimes{g})=(c{\cdot}d)\otimes(f\shuffle{g})=(c{\cdot}d{\cdot}e)\otimes{g}$.
 If $m=0$, then \eqref{eq:GeneralizedShuffleRec2} follows analogously.
 For $n>0\wedge{m>0}$, we evaluate the right hand side of \eqref{eq:GeneralizedShuffleRec2} by first using \eqref{eq:GeneralizedShuffle} for the two products $(1\otimes{f_2^n})\cdot(1\otimes{g})$ and $(1\otimes{f})\cdot(1\otimes{g_2^m})$ to obtain
 \begin{multline*}
  (c{\cdot}d)\otimes(f_1\otimes(f_2^n\shuffle{g})) + \sum_{i=1}^{n-1}\sum_{j=0}^{m-1}(c{\cdot}d{\cdot}\ep(f_{i+1}^n\odot{g_{j+1}^m}))\otimes(f_1\otimes(f_2^i\shuffle{g_1^j}))\,+\\
  +(c{\cdot}d)\otimes(g_1\otimes(f\shuffle{g_2^m})) + \sum_{i=0}^{n-1}\sum_{j=1}^{m-1}(c{\cdot}d{\cdot}\ep(f_{i+1}^n\odot{g_{j+1}^m}))\otimes(g_1\otimes(f_1^i\shuffle{g_2^j}))+(c{\cdot}d{\cdot}\ep(f\odot{g}))\otimes\et.
 \end{multline*}
 Then, by exploiting $f_1\otimes(f_2^i\shuffle{g_1^j})+g_1\otimes(f_1^i\shuffle{g_2^j})=f_1^i\shuffle{g_1^j}$ from \eqref{eq:shufflerec} for every $i\in\{1,\dots,n\}$ and $j\in\{1,\dots,m\}$, we can collect terms yielding
 \begin{multline*}
  (c{\cdot}d)\otimes(f\shuffle{g}) + \sum_{i=1}^{n-1}\sum_{j=1}^{m-1}(c{\cdot}d{\cdot}\ep(f_{i+1}^n\odot{g_{j+1}^m}))\otimes(f_1^i\shuffle{g_1^j}))\,+\\
  +\sum_{i=1}^{n-1}(c{\cdot}d{\cdot}\ep(f_{i+1}^n\odot{g}))\otimes{f_1^i}+\sum_{j=1}^{m-1}(c{\cdot}d{\cdot}\ep(f\odot{g_{j+1}^m}))\otimes{g_1^j}+(c{\cdot}d{\cdot}\ep(f\odot{g}))\otimes\et,
 \end{multline*}
 which agrees with the right hand side of \eqref{eq:GeneralizedShuffle}.
 Therefore, \eqref{eq:GeneralizedShuffleRec2} holds also in this case.
\end{proof}

For showing associativity of the multiplication, we combine \eqref{eq:GeneralizedShuffle} and \eqref{eq:GeneralizedShuffleRec2}.

\begin{theorem}\label{thm:Calgebra}
 The $\mathcal{C}$-bilinear map defined in Definition~\ref{def:GeneralizedShuffle} turns the $\mathcal{C}$-module $\mathcal{C}_2\otimes\mathcal{T}$ into a commutative $\mathcal{C}$-algebra with unit element $1\otimes\et$.
\end{theorem}
\begin{proof}
 First, we show that the map $\cdot$ is well-defined on all of $(\mathcal{C}_2\otimes\mathcal{T})^2$.
 If $t,\tilde{t} \in \mathcal{C}_2\otimes\mathcal{R}_\T^{\otimes{n}}$ and $s,\tilde{s} \in \mathcal{C}_2\otimes\mathcal{R}_\T^{\otimes{m}}$ are pure tensors such that $t+\tilde{t}$ and $s+\tilde{s}$ are pure tensors as well, then for all $\lambda\in\mathcal{C}$ we have $(\lambda{t}+\tilde{t})\cdot{s}=\lambda(t\cdot{s})+\tilde{t}\cdot{s}$ and $t\cdot(\lambda{s}+\tilde{s})=\lambda(t\cdot{s})+t\cdot\tilde{s}$ by \eqref{eq:GeneralizedShuffle}, relying on the properties of $\odot$ and $\shuffle$.
 Therefore, the product of arbitrary elements of $\mathcal{C}_2\otimes\mathcal{R}_\T^{\otimes{n}}$ and $\mathcal{C}_2\otimes\mathcal{R}_\T^{\otimes{n}}$ is well-defined.
 By the definition of $\mathcal{T}$ as direct sum, it immediately follows that $\cdot$ is well-defined on all of $(\mathcal{C}_2\otimes\mathcal{T})^2$.
\par
 The $\mathcal{C}$-bilinear construction of $\cdot$ means in particular that the distributivity laws in $\mathcal{C}_2\otimes\mathcal{T}$ hold.
 From commutativity of $\odot$, $\shuffle$, and the multiplication in $\mathcal{C}_2$, it immediately follows by \eqref{eq:GeneralizedShuffle} that the induced product $\cdot$ is commutative too.
 By \eqref{eq:GeneralizedShuffle}, we directly obtain $(c\otimes{f})\cdot(1\otimes\et) = (c{\cdot}1)\otimes(f\shuffle\et) = c\otimes{f}$ for all $c\in\mathcal{C}_2$, $n\in\mathbb{N}$, and all pure tensors $f\in\mathcal{R}_\T^{\otimes{n}}$, which shows by distributivity that $1\otimes\et$ is the neutral element.
\par
 It remains to show associativity of the product.
 By distributivity, it suffices to show that for all $c,d,e \in \mathcal{C}_2$, $n,m,l\in\mathbb{N}$, and all pure tensors $f\in\mathcal{R}_\T^{\otimes{n}}$, $g\in\mathcal{R}_\T^{\otimes{m}}$, and $h\in\mathcal{R}_\T^{\otimes{l}}$ the associator $((c\otimes{f})\cdot(d\otimes{g}))\cdot(e\otimes{h})-(c\otimes{f})\cdot((d\otimes{g})\cdot(e\otimes{h}))$ is zero.
 We first show this for the special case $n=0$, where we assume $f=\et$ without loss of generality.
 By the definition, we easily verify $((c\otimes\et)\cdot(d\otimes{g}))\cdot(e\otimes{h}) = ((c{\cdot}d)\otimes{g})\cdot(e\otimes{h}) = (c\otimes\et)\cdot((d\otimes{g})\cdot(e\otimes{h}))$.
 By commutativity, this also implies the cases when $m=0$ or $l=0$.
 For the remaining cases, we proceed by induction on the sum $n+m+l$.
 The cases with $nml=0$ proven above in particular cover all cases with $n+m+l\le2$.
 Now we assume $n+m+l\ge3$ with $n,m,l\ge1$.
 For slightly shorter notation, we set $c=d=e=1$ first.
 To evaluate the left factor in $((1\otimes{f})\cdot(1\otimes{g}))\cdot(1\otimes{h})$, we use the definition \eqref{eq:GeneralizedShuffle}.
 In order to apply the recursion \eqref{eq:GeneralizedShuffleRec2} to the product of every summand in \eqref{eq:GeneralizedShuffle} with the right factor $1\otimes{h}$, we first split off $\ep(f\odot{g})\otimes\et$ which yields $(\ep(f\odot{g})\otimes\et)\cdot(1\otimes{h})=\ep(f\odot{g})\otimes{h}$ by the base case of \eqref{eq:GeneralizedShuffleRec2}.
 For multiplying the remaining summands in
 \[
  (1\otimes{f})\cdot(1\otimes{g})-\ep(f\odot{g})\otimes\et = 1\otimes(f\shuffle{g})+\sum_{i=0}^{n-1}\sum_{j=\max(0,1-i)}^{m-1}\ep(f_{i+1}^n\odot{g_{j+1}^m})\otimes(f_1^i\shuffle{g_1^j})
 \]
 by $1\otimes{h}$ using \eqref{eq:GeneralizedShuffleRec2}, we identify the first tensor factor of the terms arising in $f_1^i\shuffle{g_1^j}$ by the recursion \eqref{eq:shufflerec} of the shuffle product, yielding $f_1^i\shuffle{g_1^j}=f_1\otimes(f_2^i\shuffle{g_1^j})+g_1\otimes(f_1^i\shuffle{g_2^j})$ for $i,j\ge1$.
 For $i=0$ resp.\ $j=0$, we have $f_1^0\shuffle{g_1^j}=g_1\otimes(f_1^0\shuffle{g_2^j})$ resp.\ $f_1^i\shuffle{g_1^0}=f_1\otimes(f_2^i\shuffle{g_1^0})$ instead.
 Altogether, we obtain $((1\otimes{f})\cdot(1\otimes{g}))\cdot(1\otimes{h})$ as
 \begin{multline*}
  (1\otimes{f_1})\circledast\bigg((1\otimes(f_2^n\shuffle{g}))\cdot(1\otimes{h})+\sum_{i=1}^{n-1}\sum_{j=0}^{m-1}(\ep(f_{i+1}^n\odot{g_{j+1}^m})\otimes(f_2^i\shuffle{g_1^j}))\cdot(1\otimes{h})\bigg)\\
  +(1\otimes{g_1})\circledast\bigg((1\otimes(f\shuffle{g_2^m}))\cdot(1\otimes{h})+\sum_{i=0}^{n-1}\sum_{j=1}^{m-1}(\ep(f_{i+1}^n\odot{g_{j+1}^m})\otimes(f_1^i\shuffle{g_2^j}))\cdot(1\otimes{h})\bigg)\\
  +(1\otimes{h_1})\circledast\bigg((1\otimes(f\shuffle{g}))\cdot(1\otimes{h_2^l})+\sum_{i=0}^{n-1}\sum_{j=\max(0,1-i)}^{m-1}(\ep(f_{i+1}^n\odot{g_{j+1}^m})\otimes(f_1^i\shuffle{g_1^j}))\cdot(1\otimes{h_2^l})\bigg)\\
  +\ep((f\shuffle{g})\odot{h})\otimes\et+\sum_{i=0}^{n-1}\sum_{j=\max(0,1-i)}^{m-1}(\ep(f_{i+1}^n\odot{g_{j+1}^m}){\cdot}\ep((f_1^i\shuffle{g_1^j})\odot{h}))\otimes\et+\ep(f\odot{g})\otimes{h}.
 \end{multline*}
 Using definition \eqref{eq:GeneralizedShuffle} in reverse for the products $(1\otimes{f_2^n})\cdot(1\otimes{g})$, $(1\otimes{f})\cdot(1\otimes{g_2^m})$, and $(1\otimes{f})\cdot(1\otimes{g})$, we can collect terms above to obtain
 \begin{multline*}
  ((1\otimes{f})\cdot(1\otimes{g}))\cdot(1\otimes{h}) = (1\otimes{f_1})\circledast\big(((1\otimes{f_2^n})\cdot(1\otimes{g}))\cdot(1\otimes{h})\big)\\
  +(1\otimes{g_1})\circledast\big(((1\otimes{f})\cdot(1\otimes{g_2^m}))\cdot(1\otimes{h})\big)+(1\otimes{h_1})\circledast\big(((1\otimes{f})\cdot(1\otimes{g}))\cdot(1\otimes{h_2^l})\big)\\
  +\ep((f\shuffle{g})\odot{h})\otimes\et+\sum_{j=0}^{m-1}\sum_{i=\max(0,1-j)}^{n-1}(\ep(f_{i+1}^n\odot{g_{j+1}^m}){\cdot}\ep((f_1^i\shuffle{g_1^j})\odot{h}))\otimes\et.
 \end{multline*}
 Analogously, we have
 \begin{multline*}
  (1\otimes{f})\cdot((1\otimes{g})\cdot(1\otimes{h})) = (1\otimes{f_1})\circledast\big((1\otimes{f_2^n})\cdot((1\otimes{g})\cdot(1\otimes{h}))\big)\\
  +(1\otimes{g_1})\circledast\big((1\otimes{f})\cdot((1\otimes{g_2^m})\cdot(1\otimes{h}))\big)+(1\otimes{h_1})\circledast\big((1\otimes{f})\cdot((1\otimes{g})\cdot(1\otimes{h_2^l}))\big)\\
  +\ep(f\odot(g\shuffle{h}))\otimes\et+\sum_{j=0}^{m-1}\sum_{k=\max(0,1-j)}^{n-1}(\ep(g_{j+1}^m\odot{h_{k+1}^l}){\cdot}\ep(f\odot(g_1^j\shuffle{h_1^k})))\otimes\et.
 \end{multline*}
 In the difference $((1\otimes{f})\cdot(1\otimes{g}))\cdot(1\otimes{h})-(1\otimes{f})\cdot((1\otimes{g})\cdot(1\otimes{h}))$, all terms in $\mathcal{C}_2\otimes\widetilde{\mathcal{T}}$ cancel by the induction hypothesis and what remains is of the form $\tilde{c}\otimes\et$, where $\tilde{c}\in\mathcal{C}_2$ is exactly \eqref{eq:ConstRel} and hence is zero as an element of $\mathcal{C}_2$.
 Note that for arbitrary $c,d,e\in\mathcal{C}_2$, we have
 \begin{multline*}
  ((c\otimes{f})\cdot(d\otimes{g}))\cdot(e\otimes{h})-(c\otimes{f})\cdot((d\otimes{g})\cdot(e\otimes{h}))\\
  =((c{\cdot}d{\cdot}e)\otimes\et)\cdot\big(((1\otimes{f})\cdot(1\otimes{g}))\cdot(1\otimes{h})-(1\otimes{f})\cdot((1\otimes{g})\cdot(1\otimes{h}))\big).
 \end{multline*}
 This completes the induction.
\end{proof}

Later, in Section~\ref{sec:shufflealgebra}, we will explore the relation of $(\mathcal{C}_2\otimes\mathcal{T},+,\cdot)$ with $(\mathcal{T},+,\shuffle)$ further.

\subsection{Derivation and integration}
\label{sec:DerInt}
On the $\mathcal{C}$-algebra $\mathrm{IDR}(\mathcal{R})$ constructed in Section~\ref{sec:AlgebraStructure}, we define derivation and integration as follows.
\begin{definition}\label{def:DerInt}
 Let $f \in \mathcal{R}_\T^{\otimes{n}}$ be a pure tensor for some $n \in \{0,1,2,\dots\}$ and let $c=c_1\otimes{c_2}\in\mathcal{C}_1\otimes\mathcal{C}_2$ and $f_0 \in \mathcal{R}$.
 We extend the derivation of $\mathcal{R}$ to a $\mathcal{C}$-linear map on $\mathrm{IDR}(\mathcal{R})$ by
 \begin{equation}\label{eq:DefDeriv}
  \Der(f_0\otimes{c}\otimes{f}):=
  \begin{cases}(\Der{f_0})\otimes{c}\otimes{f}&n=0\\
  (\Der{f_0})\otimes{c}\otimes{f}+(f_0f_1)\otimes{c}\otimes{f_2^n}&n>0\end{cases}.
 \end{equation}
 We define the integration recursively by
 \begin{equation}\label{eq:DefInt}
  \Int(f_0\otimes{c}\otimes{f}):=
  \begin{cases}(\Q{f_0})\otimes{c}\otimes{f}-1\otimes(c_1{\cdot}\ep(\Q{f_0}))\otimes{c_2}\otimes{f}+{}\\\quad{}+1\otimes{c}\otimes((f_0-\Der\Q{f_0})\otimes{f})&n=0\\[\smallskipamount]
  (\Q{f_0})\otimes{c}\otimes{f}-1\otimes(c_1{\cdot}\ep((\Q{f_0})\otimes{f}))\otimes{c_2}\otimes\et-{}\\\quad{}-\Int(((\Q{f_0})f_1)\otimes{c}\otimes{f_2^n})+1\otimes{c}\otimes((f_0-\Der\Q{f_0})\otimes{f})&n>0\end{cases}
 \end{equation}
 and extend it $\mathcal{C}$-linearly to all of $\mathrm{IDR}(\mathcal{R})$.
\end{definition}
We also define the $\mathcal{C}$-subalgebra
\begin{equation}\label{eq:Cbar}
 \bar{\mathcal{C}} := \mathcal{C}\otimes\mathcal{C}_1\otimes\mathcal{C}_2\otimes\mathcal{C}\et \subseteq \mathrm{IDR}(\mathcal{R}).
\end{equation}
For every $f_0\in\mathcal{R}$, $c_1,d_1\in\mathcal{C}_1$, $c_2,d_2\in\mathcal{C}_2$, $n\in\mathbb{N}$, and pure tensor $f\in\mathcal{R}_\T^{\otimes{n}}$, it is straightforward to check by the definitions that
\begin{align}
 \Der((1\otimes{c_1}\otimes{c_2}\otimes\et)\cdot(f_0\otimes{d_1}\otimes{d_2}\otimes{f})) &= (1\otimes{c_1}\otimes{c_2}\otimes\et)\cdot\Der(f_0\otimes{d_1}\otimes{d_2}\otimes{f})\label{eq:linDer}\\
 \Int((1\otimes{c_1}\otimes{c_2}\otimes\et)\cdot(f_0\otimes{d_1}\otimes{d_2}\otimes{f})) &= (1\otimes{c_1}\otimes{c_2}\otimes\et)\cdot\Int(f_0\otimes{d_1}\otimes{d_2}\otimes{f}).\label{eq:linInt}
\end{align}
Consequently, $\Der$ and $\Int$ are even $\bar{\mathcal{C}}$-linear.

\begin{example}\label{ex:Int}
 Based on Example~\ref{ex:Q}, we illustrate integration in $\mathrm{IDR}(\mathcal{C}(x))$.
 Using \eqref{eq:DefInt} for $n=1$, we compute
 \begin{multline*}
  \Int\big(\tfrac{x}{(x+2)^2}\otimes1\otimes1\otimes\tfrac{1}{x}\big) = \tfrac{2}{x+2}\otimes1\otimes1\otimes\tfrac{1}{x}-1\otimes\ep\big(\tfrac{2}{x+2}\otimes\tfrac{1}{x}\big)\otimes1\otimes\et-{}\\
  {}-\Int\big(\tfrac{2}{x(x+2)}\otimes1\otimes1\otimes\et\big)+1\otimes1\otimes1\otimes\big(\tfrac{1}{x+2}\otimes\tfrac{1}{x}\big)
 \end{multline*}
 and, recursively, we apply \eqref{eq:DefInt} with $n=0$ to obtain
 \[
  \Int\big(\tfrac{2}{x(x+2)}\otimes1\otimes1\otimes\et\big)=0-0+1\otimes1\otimes1\otimes\tfrac{2}{x(x+2)}.
 \]
 Altogether, this yields $\Int\big(\tfrac{x}{(x+2)^2}\otimes1\otimes1\otimes\tfrac{1}{x}\big)$ as an explicit element of $\mathrm{IDR}(\mathcal{C}(x))$.
 Analytically, if integration is considered as $\int_a^x$, this would correspond to $\int_a^x\tfrac{t}{(t+2)^2}\int_a^t\tfrac{1}{u}\,du\,dt$ being equal to $\tfrac{2}{x+2}\int_a^x\tfrac{1}{t}\,dt-\big(\tfrac{2}{x+2}\int_a^x\tfrac{1}{t}\,dt\big)\big|_{x=a}-\int_a^x\tfrac{1}{t(t+2)}\,dt+\int_a^x\tfrac{1}{t+2}\int_a^t\tfrac{1}{u}\,du\,dt$.
\end{example}

In the following lemmas, we show relevant properties of the above definitions, which are needed for the proof of Theorem~\ref{thm:embedding}.
In particular, Lemmas \ref{lem:Leibniz} and \ref{lem:const} show that $\Der$ satisfies the Leibniz rule and yield $\bar{\mathcal{C}}$ as the ring of constants in $\mathrm{IDR}(\mathcal{R})$.

\begin{lemma}\label{lem:Leibniz}
 The $\mathcal{C}$-linear map $\Der$ defined by \eqref{eq:DefDeriv} is a derivation on $\mathrm{IDR}(\mathcal{R})$.
\end{lemma}
\begin{proof}
 By $\mathcal{C}$-linearity and \eqref{eq:linDer}, it suffices to show that $\Der$ satisfies $\Der(\bar{f}\cdot\bar{g})=(\Der\bar{f})\cdot\bar{g}+\bar{f}\cdot\Der\bar{g}$, where $\bar{f}=f_0\otimes1\otimes1\otimes{f}$ and $\bar{g}=g_0\otimes1\otimes1\otimes{g}$ with $f_0,g_0\in\mathcal{R}$ and pure tensors $f\in\mathcal{R}_\T^{\otimes{n}}$ and $g\in\mathcal{R}_\T^{\otimes{m}}$.
 We distinguish four cases.
 Note that $\bar{f}\cdot\bar{g} = (f_0g_0)\otimes1\otimes1\otimes(f\otimes{g})$ whenever $n=0$ or $m=0$.
 For $n=m=0$ it easily follows from the Leibniz rule on $\mathcal{R}$ that
 \[
  \Der(\bar{f}\cdot\bar{g}) = ((\Der{f_0})g_0+f_0\Der{g_0})\otimes1\otimes1\otimes(f\otimes{g}) = (\Der\bar{f})\cdot\bar{g}+\bar{f}\cdot\Der\bar{g}.
 \]
 If $n=0$ and $m>0$, using $c\in\mathcal{C}$ such that $f=c\et$, we simply rearrange terms to obtain
 \begin{align*}
  \Der(\bar{f}\cdot\bar{g}) &= (c\Der{f_0g_0})\otimes1\otimes1\otimes{g}+(cf_0g_0g_1)\otimes1\otimes1\otimes{g_2^m}\\
  &= (c(\Der{f_0})g_0)\otimes1\otimes1\otimes{g}+(cf_0\Der{g_0})\otimes1\otimes1\otimes{g}+(cf_0g_0g_1)\otimes1\otimes1\otimes{g_2^m}\\
  &= (\Der\bar{f})\cdot\bar{g}+\bar{f}\cdot\Der\bar{g}.
 \end{align*}
 If $n>0$ and $m=0$, the identity is proved analogously.
 Finally, if $n>0$ and $m>0$, we use \eqref{eq:GeneralizedShuffleRec2} to compute
 \begin{align*}
  \Der(\bar{f}\cdot\bar{g}) &= (\Der{f_0g_0})\otimes1\otimes\big((1\otimes{f_1})\circledast\big((1\otimes{f_2^n})\cdot(1\otimes{g})\big)\big)+(f_0g_0f_1)\otimes1\otimes\big((1\otimes{f_2^n})\cdot(1\otimes{g})\big)\\
  &\quad{}+(\Der{f_0g_0})\otimes1\otimes\big((1\otimes{g_1})\circledast\big((1\otimes{f})\cdot(1\otimes{g_2^m})\big)\big)+(f_0g_0g_1)\otimes1\otimes\big((1\otimes{f})\cdot(1\otimes{g_2^m})\big)\\
  &\quad{}+(\Der{f_0g_0})\otimes1\otimes\ep(f\odot{g})\otimes\et\\
  &= (\Der{f_0g_0})\otimes1\otimes\big((1\otimes{f})\cdot(1\otimes{g})\big)\\
  &\quad{}+(f_0g_0f_1)\otimes1\otimes\big((1\otimes{f_2^n})\cdot(1\otimes{g})\big)+(f_0g_0g_1)\otimes1\otimes\big((1\otimes{f})\cdot(1\otimes{g_2^m})\big)\\
  &= ((\Der{f_0})g_0)\otimes1\otimes\big((1\otimes{f})\cdot(1\otimes{g})\big)+(f_0f_1g_0)\otimes1\otimes\big((1\otimes{f_2^n})\cdot(1\otimes{g})\big)\\
  &\quad{}+(f_0\Der{g_0})\otimes1\otimes\big((1\otimes{f})\cdot(1\otimes{g})\big)+(f_0g_0g_1)\otimes1\otimes\big((1\otimes{f})\cdot(1\otimes{g_2^m})\big)\\
  &= (\Der\bar{f})\cdot\bar{g}+\bar{f}\cdot\Der\bar{g}.\qedhere
 \end{align*}
\end{proof}

\begin{lemma}\label{lem:RightInverse}
 On $\mathrm{IDR}(\mathcal{R})$, the $\mathcal{C}$-linear map $\Int$ defined by \eqref{eq:DefInt} is a right inverse of $\Der$.
\end{lemma}
\begin{proof}
 By $\mathcal{C}$-linearity of $\Der$ and $\Int$, it suffices to verify that
 \begin{equation}\label{eq:RightInverse}
  \Der\Int(f_0\otimes{c}\otimes{f})=f_0\otimes{c}\otimes{f}
 \end{equation}
 holds for all $n \in \{0,1,2,\dots\}$, $f_0 \in \mathcal{R}$, $c=c_1\otimes{c_2}\in\mathcal{C}_1\otimes\mathcal{C}_2$, and all pure tensors $f \in \mathcal{R}_\T^{\otimes{n}}$.
 We proceed by induction on $n$.
 If $n=0$, we can assume $f=\et$ and obtain
 \begin{align*}
  \Der\Int(f_0\otimes{c}\otimes\et) &= \Der\big((\Q{f_0})\otimes{c}\otimes\et-1\otimes(c_1{\cdot}\ep(\Q{f_0}))\otimes{c_2}\otimes\et+1\otimes{c}\otimes(f_0-\Der\Q{f_0})\big)\\
  &= (\Der\Q{f_0})\otimes{c}\otimes\et+(f_0-\Der\Q{f_0})\otimes{c}\otimes\et\\
  &= f_0\otimes{c}\otimes\et.
 \end{align*}
 For $n>0$, we similarly compute
 \begin{align*}
  \Der\Int(f_0\otimes{c}\otimes{f}) &= \Der\big((\Q{f_0})\otimes{c}\otimes{f}-1\otimes(c_1{\cdot}\ep((\Q{f_0})\otimes{f}))\otimes{c_2}\otimes\et-{}\\
  &\quad{}-\Int(((\Q{f_0})f_1)\otimes{c}\otimes{f_2^n})+1\otimes{c}\otimes((f_0-\Der\Q{f_0})\otimes{f})\big)\\
  &= (\Der\Q{f_0})\otimes{c}\otimes{f}+((\Q{f_0})f_1)\otimes{c}\otimes{f_2^n}-{}\\
  &\quad{}-\Der\Int(((\Q{f_0})f_1)\otimes{c}\otimes{f_2^n})+(f_0-\Der\Q{f_0})\otimes{c}\otimes{f},
 \end{align*}
 which equals $f_0\otimes{c}\otimes{f}$ by the induction hypothesis.
\end{proof}

On $\mathrm{IDR}(\mathcal{R})$, the derivation and its right inverse defined in Definition~\ref{def:DerInt} induce an operation $\E$ by
\begin{equation}\label{eq:DefEval}
 \E:=\id-\Int\Der
\end{equation}
that maps to $\const_\Der(\mathrm{IDR}(\mathcal{R}))$.
The action of $\E$ can also be characterized explicitly as follows.
\begin{lemma}\label{lem:Evaluation}
 Let $f_0\otimes{c_1}\otimes{c_2}\otimes{f} \in \mathrm{IDR}(\mathcal{R})$ with $f \in \mathcal{R}_\T^{\otimes{n}}$ for some $n \in \{0,1,2,\dots\}$. Then,
 \begin{equation}\label{eq:Evaluation}
  \E(f_0\otimes{c_1}\otimes{c_2}\otimes{f})=
  \begin{cases}1\otimes\big(c_1{\cdot}\big(\underbrace{f_0-\Q\Der{f_0}}_{\in \mathcal{C} \subseteq \mathcal{C}_1}+\ep(\Q\Der{f_0})\big)\big)\otimes{c_2}\otimes{f}&n=0\\
  1\otimes(c_1{\cdot}\ep((\Q\Der{f_0})\otimes{f}))\otimes{c_2}\otimes\et&n>0\end{cases}.
 \end{equation}
\end{lemma}
\begin{proof}
 By $\mathcal{C}$-linearity, we can assume w.l.o.g.\ that $f \in \mathcal{R}_\T^{\otimes{n}}$ is a pure tensor.
 For shorter notation, we abbreviate $c:=c_1\otimes{c_2}$.
 If $n=0$, then by definition we have
 \begin{align*}
  \E(f_0\otimes{c}\otimes{f}) &= f_0\otimes{c}\otimes{f} - \Int((\Der{f_0})\otimes{c}\otimes{f})\\
  &= f_0\otimes{c}\otimes{f} - \big((\Q\Der{f_0})\otimes{c}\otimes{f}-{}\\
  &\quad{}-1\otimes(c_1{\cdot}\ep(\Q\Der{f_0}))\otimes{c_2}\otimes{f}+1\otimes{c}\otimes((\Der{f_0}-\Der\Q\Der{f_0})\otimes{f})\big)\\
  &= 1\otimes(c_1{\cdot}(\underbrace{f_0-\Q\Der{f_0}}_{\in \mathcal{C} \subseteq \mathcal{C}_1}))\otimes{c_2}\otimes{f} + 1\otimes(c_1{\cdot}\ep(\Q\Der{f_0}))\otimes{c_2}\otimes{f}\\
  &= 1\otimes\big(c_1{\cdot}\big(f_0-\Q\Der{f_0}+\ep(\Q\Der{f_0})\big)\big)\otimes{c_2}\otimes{f}.
 \end{align*}
 If $n>0$, then we can exploit $f_0-\Q\Der{f_0} \in \mathcal{C}$ and $\Q{f_1}=0$ in order to compute  the integral $\Int(((f_0-\Q\Der{f_0})f_1)\otimes{c}\otimes{f_2^n}) = (f_0-\Q\Der{f_0})\otimes{c}\otimes{f}$ regardless of $n$.
 Then, we obtain
 \begin{align*}
  \E(f_0\otimes{c}\otimes{f}) &= f_0\otimes{c}\otimes{f} - \Int\big((\Der{f_0})\otimes{c}\otimes{f}+(f_0f_1)\otimes{c}\otimes{f_2^n}\big)\\
  &= f_0\otimes{c}\otimes{f} - \big((\Q\Der{f_0})\otimes{c}\otimes{f} - 1\otimes(c_1{\cdot}\ep((\Q\Der{f_0})\otimes{f}))\otimes{c_2}\otimes\et+{}\\
  &\quad{}+\Int(((f_0-\Q\Der{f_0})f_1)\otimes{c}\otimes{f_2^n}) + 1\otimes{c}\otimes((\Der{f_0}-\Der\Q\Der{f_0})\otimes{f})\big)\\
  &= 1\otimes(c_1{\cdot}\ep((\Q\Der{f_0})\otimes{f}))\otimes{c_2}\otimes\et.\qedhere
 \end{align*}
\end{proof}

\begin{example}\label{ex:Eval}
 Let $(\mathcal{R},\Der)=(\mathcal{C}(x),\frac{d}{dx})$ and $\Q$ be as in Example~\ref{ex:Q}.
 We illustrate \eqref{eq:Evaluation} with $n=0$ by computing the instance
 \[
  \E\big(\tfrac{x}{x+2}\otimes1\otimes1\otimes\et\big)=1\otimes\big(\tfrac{x}{x+2}+\tfrac{2}{x+2}+\ep(-\tfrac{2}{x+2})\big)\otimes1\otimes\et = 1\otimes\big(1-\ep(\tfrac{2}{x+2})\big)\otimes1\otimes\et
 \]
 using $\Q\Der\tfrac{x}{x+2}=\Q\tfrac{2}{(x+2)^2}=-\tfrac{2}{x+2}$.
 Analytically, this would correspond to the trivial identity $\tfrac{x}{x+2}\big|_{x=a}=1-\tfrac{2}{x+2}\big|_{x=a}$.
\end{example}

This explicit characterization allows a simple proof of which elements of $\mathrm{IDR}(\mathcal{R})$ vanish under $\Der$.

\begin{lemma}\label{lem:const}
 The constants of $\mathrm{IDR}(\mathcal{R})$ are given by $\const_\Der(\mathrm{IDR}(\mathcal{R}))=\bar{\mathcal{C}}$.
\end{lemma}
\begin{proof}
 By \eqref{eq:DefEval} and Lemma~\ref{lem:RightInverse}, we have that $\Der\E=0$ and that $\Der{f}=0$ implies $\E{f}=f$ for all $f \in \mathrm{IDR}(\mathcal{R})$.
 Together, this implies $\const_\Der(\mathrm{IDR}(\mathcal{R})) = \E\,\mathrm{IDR}(\mathcal{R})$.
 From \eqref{eq:Evaluation}, it follows that $\E\,\mathrm{IDR}(\mathcal{R})=\bar{\mathcal{C}}$.
\end{proof}

Finally, we can prove that $(\mathrm{IDR}(\mathcal{R}),\Der,\Int)$ indeed is an integro-differential ring and that $\mathcal{R}$ can be embedded accordingly.

\begin{theorem}\label{thm:embedding}
 With $\Der$ and $\Int$ defined in Definition~\ref{def:DerInt}, the following hold for $\mathrm{IDR}(\mathcal{R})$ defined in Definition~\ref{def:DefIDR}:
 \begin{enumerate}
  \item $(\mathrm{IDR}(\mathcal{R}),\Der,\Int)$ is a commutative integro-differential ring.
  \item The $\mathcal{C}$-algebra homomorphism $\iota:\mathcal{R}\to\mathrm{IDR}(\mathcal{R})$ defined by \eqref{eq:DefIota} is an injective differential ring homomorphism.
  \item For $n \in \{1,2,\dots\}$, $f_0\in\mathcal{R}$, and pure tensors $f\in\mathcal{R}_\T^{\otimes{n}}$, we have
   \[
    \iota(f_0)\cdot\Int\iota(f_1){\cdot}\dots{\cdot}\Int\iota(f_n) = f_0\otimes1\otimes1\otimes{f}.
   \]
 \end{enumerate}
\end{theorem}
\begin{proof}
 By Theorem~\ref{thm:Calgebra} and Lemmas \ref{lem:Leibniz} and \ref{lem:const}, $(\mathrm{IDR}(\mathcal{R}),\Der)$ is a commutative differential ring with constants $\bar{\mathcal{C}}$.
 Then, by $\mathcal{C}$-linearity of $\Int$, it follows from \eqref{eq:linInt} that $\Int$ is indeed $\const_\Der(\mathrm{IDR}(\mathcal{R}))$-linear.
 Finally, Lemma~\ref{lem:RightInverse} shows that $(\mathrm{IDR}(\mathcal{R}),\Der,\Int)$ is a commutative integro-differential ring.
\par
 By construction of $\mathrm{IDR}(\mathcal{R})$ as tensor product of commutative $\mathcal{C}$-algebras, $\iota$ is a $\mathcal{C}$-algebra homomorphism.
 Its image is $\mathcal{R}\otimes\mathcal{C}\otimes\mathcal{C}\otimes\mathcal{C}\et$ and $\iota(1)$ is the unit element.
 By construction, $\mathcal{C}$ is a complemented $\mathcal{C}$-submodule in $\mathcal{C}_1$ and $\mathcal{C}\et$ is complemented in $\mathcal{T}$.
 Also, $\mathcal{C}$ is complemented in $\mathcal{C}_2$ by definition since $J_2$ is contained in the canonical complement of $\mathcal{C}$ in $\Sym(\ep(M_2))$.
 By distributivity of the tensor product of $\mathcal{C}$-modules, this implies injectivity of $\iota$.
 In addition, by \eqref{eq:DefDeriv}, it follows that $\Der\iota(f)=\iota(\Der{f})$ for all $f\in\mathcal{R}$.
\par
 For $n=1$, the identity follows since we obtain $\Int\iota(f_1) = 1\otimes1\otimes1\otimes{f_1}$ from \eqref{eq:DefInt} using $\Q{f_1}=0$.
 For $n\ge2$, we assume $\iota(f_1)\cdot\Int\iota(f_2){\cdot}\dots{\cdot}\Int\iota(f_n) = f_1\otimes1\otimes1\otimes{f_2^n}$ by induction and apply $\Int$ to obtain $\Int\iota(f_1){\cdot}\dots{\cdot}\Int\iota(f_n) = 1\otimes1\otimes1\otimes{f}$ from \eqref{eq:DefInt} using $\Q{f_1}=0$ again.
 Then, the identity follows, which completes the induction.
\end{proof}

Note that the nested integrals $\Int\iota(f_1){\cdot}\dots{\cdot}\Int\iota(f_n)$ with $f_i \in \mathcal{R}_\mathrm{T}$ generate $\mathrm{IDR}(\mathcal{R})$ as a module over $\mathcal{R}\bar{\mathcal{C}}=\mathcal{R}\otimes\mathcal{C}_1\otimes\mathcal{C}_2\otimes\mathcal{C}\et$.

\subsection{Universal property}
\label{sec:universal}

\begin{theorem}\label{thm:universal}
 Let $(\mathcal{S},\Der,\Int)$ be a commutative integro-differential ring and let $\varphi:\mathcal{R}\to\mathcal{S}$ be a differential ring homomorphism.
 Then, there exists a unique integro-differential ring homomorphism $\eta:\mathrm{IDR}(\mathcal{R})\to\mathcal{S}$ such that $\eta(\iota(f))=\varphi(f)$ for all $f\in\mathcal{R}$.
\end{theorem}
\begin{proof}
 Since $\mathcal{C}$ is a subring of $\mathcal{R}$, we can consider $\mathcal{S}$ as a $\mathcal{C}$-algebra in a canonical way via $\varphi$.
 Moreover, we have that $\varphi$ is a $\mathcal{C}$-algebra homomorphism and that $\varphi(\mathcal{C})\subseteq\const_\Der(\mathcal{S})$.
 For any $f_0\in\mathcal{R}_\mathrm{J}$, $n\in\mathbb{N}$, and any pure tensor $f\in\mathcal{R}_\T^{\otimes{n}}$, we define
 \[
  \eta_1(\ep(f_0\otimes{f})):=\E\varphi(f_0)\Int\varphi(f_1)\dots\Int\varphi(f_n)\in\mathcal{S}.
 \]
 Since this definition is $\mathcal{C}$-multilinear, we can extend it uniquely to a $\mathcal{C}$-algebra homomorphism $\eta_1:\mathcal{C}_1\to\mathcal{S}$.
 For any $n,m\in\mathbb{N}^+$ and pure tensors $f\in\mathcal{R}_\T^{\otimes{n}}$ and $g\in\mathcal{R}_\T^{\otimes{m}}$, we define
 \[
  \eta_2(\ep(f\odot{g})):=\E\left(\Int\varphi(f_1)\dots\Int\varphi(f_n)\right)\Int\varphi(g_1)\dots\Int\varphi(g_m)\in\mathcal{S}.
 \]
 Since this definition is not only $\mathcal{C}$-multilinear but also symmetric in $f$ and $g$, we can extend it uniquely to a $\mathcal{C}$-algebra homomorphism $\eta_2:\Sym(\ep(M_2))\to\mathcal{S}$.
 In $\mathcal{S}$, we have $\E\Int{f}=0$ for all $f\in\mathcal{S}$, which implies by Corollary~\ref{cor:IDRConstRel} that $\eta_2$ maps \eqref{eq:ConstRel} to zero for all $n,m,l\in\mathbb{N}^+$ and pure tensors $f \in \mathcal{R}_\T^{\otimes{n}}$, $g \in \mathcal{R}_\T^{\otimes{m}}$, and $h \in \mathcal{R}_\T^{\otimes{l}}$.
 Hence, by passing to the quotient by the ideal $J_2 \subseteq \Sym(\ep(M_2))$, $\eta_2$ uniquely induces a $\mathcal{C}$-algebra homomorphism $\tilde{\eta}_2:\mathcal{C}_2\to\mathcal{S}$.
 For any $c\in\mathcal{C}_2$, $n\in\mathbb{N}$, and pure tensor $f\in\mathcal{R}_\T^{\otimes{n}}$, we define
 \[
  \eta_3(c\otimes{f}):=\tilde{\eta}_2(c)\Int\varphi(f_1)\dots\Int\varphi(f_n)\in\mathcal{S}.
 \]
 Since this definition is $\mathcal{C}$-multilinear, we can extend it uniquely to a $\mathcal{C}$-module homomorphism $\eta_3:\mathcal{C}_2\otimes\mathcal{T}\to\mathcal{S}$.
 Moreover, $\eta_3$ is even a $\mathcal{C}$-algebra homomorphism since, for $n,m\in\mathbb{N}$ and pure tensors $f\in\mathcal{R}_\T^{\otimes{n}}$ and $g\in\mathcal{R}_\T^{\otimes{m}}$, we can apply \eqref{eq:IDRShuffle} in $\mathcal{S}$ to obtain
 \[
  \eta_3(1\otimes{f})\eta_3(1\otimes{g}) = \eta_3(1\otimes(f\shuffle{g}))+\sum_{i=0}^{n-1}\sum_{j=0}^{m-1}\eta_2(\ep(f_{i+1}^n\odot{g_{j+1}^m}))\eta_3(1\otimes(f_1^i\shuffle{g_1^j})),
 \]
 which for $c,d\in\mathcal{C}_2$ implies by \eqref{eq:GeneralizedShuffle} that 
 \begin{multline*}
  \eta_3(c\otimes{f})\eta_3(d\otimes{g}) = \tilde{\eta}_2(c)\eta_3(1\otimes{f})\tilde{\eta}_2(d)\eta_3(1\otimes{g})\\
  = \tilde{\eta}_2(c{\cdot}d)\eta_3((1\otimes{f})\cdot(1\otimes{g}))
  = \eta_3((c\otimes{f})\cdot(d\otimes{g})).
 \end{multline*}
 Following the definition of $\mathrm{IDR}(\mathcal{R})$ as tensor product of the $\mathcal{C}$-algebras $\mathcal{R}$, $\mathcal{C}_1$, and $\mathcal{C}_2\otimes\mathcal{T}$, we define the $\mathcal{C}$-algebra homomorphism $\eta:\mathrm{IDR}(\mathcal{R})\to\mathcal{S}$ by
 \[
  \eta:=\varphi\otimes\eta_1\otimes\eta_3.
 \]
 Obviously, we have that $\eta(\iota(f))=\varphi(f)\eta_1(1)\eta_3(1\otimes\et)=\varphi(f)$ for all $f\in\mathcal{R}$.
 It remains to show that $\eta$ is even an integro-differential ring homomorphism.
 Since $\bar{\mathcal{C}}$ defined by \eqref{eq:Cbar} is a subring of $\mathrm{IDR}(\mathcal{R})$, we can consider $\mathcal{S}$ as a $\bar{\mathcal{C}}$-algebra in a canonical way via $\eta$.
 Note that $\eta$ maps $\bar{\mathcal{C}}$ into $\const_\Der(\mathcal{S})$ by definition of $\eta_1$ and $\eta_2$.
 Hence, $\Der$ and $\Int$ are not only $\bar{\mathcal{C}}$-linear on $\mathrm{IDR}(\mathcal{R})$ but also on $\mathcal{S}$.
 Therefore, it suffices to show that, for all $f_0\in\mathcal{R}$, $n\in\mathbb{N}$, and pure tensors $f\in\mathcal{R}_\T^{\otimes{n}}$, we have
 \begin{align*}
  \Der\eta(f_0\otimes1\otimes1\otimes{f})&=\eta(\Der(f_0\otimes1\otimes1\otimes{f})) \quad\text{and}\\
  \Int\eta(f_0\otimes1\otimes1\otimes{f})&=\eta(\Int(f_0\otimes1\otimes1\otimes{f})).
 \end{align*}
 First, if $n=0$, we can assume $f=\et$ w.l.o.g.\ and obtain $\Der\eta(\iota(f_0)) = \Der\varphi(f_0) = \varphi(\Der{f_0}) = \eta(\iota(\Der{f_0})) = \eta(\Der\iota(f_0))$ by \eqref{eq:DefDeriv} and definition of $\varphi$ and $\eta$.
 If $n>0$, by the Leibniz rule in $\mathcal{S}$, we have
 \begin{align*}
  \Der\eta(f_0\otimes1\otimes1\otimes{f}) &= \Der\varphi(f_0)\Int\varphi(f_1)\eta_3(1\otimes{f_2^n})\\
  &= (\Der\varphi(f_0))\Int\varphi(f_1)\eta_3(1\otimes{f_2^n})+\varphi(f_0)\varphi(f_1)\eta_3(1\otimes{f_2^n})\\
  &= \varphi(\Der{f_0})\eta_3(1\otimes{f})+\varphi(f_0f_1)\eta_3(1\otimes{f_2^n})\\
  &= \eta((\Der{f_0})\otimes1\otimes1\otimes{f})+\eta((f_0f_1)\otimes1\otimes1\otimes{f_2^n}).
 \end{align*}
 This is equal to $\eta(\Der(f_0\otimes1\otimes1\otimes{f}))$ by \eqref{eq:DefDeriv}.
 Second, we abbreviate $f_\T := f_0-\Der\Q{f_0} \in \mathcal{R}_\T$ and we proceed by induction on $n$.
 For $n=0$, we can assume $f=\et$ w.l.o.g.\ and we have
 \begin{multline*}
  \Int\eta(\iota(f_0)) = \Int\varphi(f_0)
  = \Int\Der\varphi(\Q{f_0})+\Int\varphi(f_\T)
  = \varphi(\Q{f_0})-\E\varphi(\Q{f_0})+\Int\varphi(f_\T)\\
  = \eta(\iota(\Q{f_0}))-\eta(1\otimes\ep(\Q{f_0})\otimes1\otimes\et)+\eta(1\otimes1\otimes1\otimes{f_\T})
  = \eta(\Int\iota(f_0))
 \end{multline*}
 by \eqref{eq:EvaluationDef} in $\mathcal{S}$ and by \eqref{eq:DefInt}.
 For $n\ge1$, on the one hand, we compute 
 \begin{align*}
  \Int\eta(f_0\otimes1\otimes1\otimes{f}) &= \Int\varphi(f_0)\Int\varphi(f_1)\eta_3(1\otimes{f_2^n})\\
  &= \Int(\Der\varphi(\Q{f_0}))\Int\varphi(f_1)\eta_3(1\otimes{f_2^n}) + \Int\varphi(f_\T)\Int\varphi(f_1)\eta_3(1\otimes{f_2^n})\\
  &= \varphi(\Q{f_0})\Int\varphi(f_1)\eta_3(1\otimes{f_2^n}) - \E\varphi(\Q{f_0})\Int\varphi(f_1)\eta_3(1\otimes{f_2^n})\\
  &\quad - \Int\varphi(\Q{f_0})\varphi(f_1)\eta_3(1\otimes{f_2^n}) + \Int\varphi(f_\T)\Int\varphi(f_1)\eta_3(1\otimes{f_2^n})
 \end{align*}
 by \eqref{eq:IBP} in $\mathcal{S}$ and, on the other hand, \eqref{eq:DefInt} yields
 \begin{align*}
  \eta(\Int(f_0\otimes1\otimes1\otimes{f})) &= \eta((\Q{f_0})\otimes1\otimes1\otimes{f}) - \eta(1\otimes\ep((\Q{f_0})\otimes{f})\otimes1\otimes\et)\\
  &\quad - \eta(\Int(((\Q{f_0})f_1)\otimes1\otimes1\otimes{f_2^n})) + \eta(1\otimes1\otimes1\otimes(f_\T\otimes{f}))\\
  &= \varphi(\Q{f_0})\Int\varphi(f_1)\eta_3(1\otimes{f_2^n}) - \eta_1(\ep((\Q{f_0})\otimes{f}))\\
  &\quad - \eta(\Int(((\Q{f_0})f_1)\otimes1\otimes1\otimes{f_2^n})) + \Int\varphi(f_\T)\Int\varphi(f_1)\eta_3(1\otimes{f_2^n}).
 \end{align*}
 For completing the induction, we observe that these two elements of $\mathcal{S}$ are equal, since we obtain $\Int\varphi(\Q{f_0})\varphi(f_1)\eta_3(1\otimes{f_2^n}) = \eta(\Int(((\Q{f_0})f_1)\otimes1\otimes1\otimes{f_2^n}))$ from the induction hypothesis.
\par
 Finally, to show uniqueness, assume $\eta:\mathrm{IDR}(\mathcal{R})\to\mathcal{S}$ is an arbitrary integro-differential ring homomorphism such that $\eta(\iota(f))=\varphi(f)$ for all $f\in\mathcal{R}$.
 Let $n\in\mathbb{N}^+$ and let $f\in\mathcal{R}_\T^{\otimes{n}}$ be a pure tensor, then by \eqref{eq:DefInt} and $\Q{f_1}=0$ we obtain
 \[
  \eta(1\otimes1\otimes1\otimes{f}) = \eta(\Int(\iota(f_1)\cdot(1\otimes1\otimes1\otimes{f_2^n}))) = \Int\varphi(f_1)\eta(1\otimes1\otimes1\otimes{f_2^n}).
 \]
 Then, we inductively obtain $\eta(1\otimes1\otimes1\otimes{f})=\Int\varphi(f_1)\dots\Int\varphi(f_n)$, using $\eta(\iota(1))=1$ as base case.
 Let $f_0\in\mathcal{R}_\mathrm{J}$, then by \eqref{eq:Evaluation} we have
 \[
  \eta(1\otimes\ep(f_0\otimes{f})\otimes1\otimes\et) = \eta(\E(\iota(f_0)\cdot(1\otimes1\otimes1\otimes{f}))) = \E\varphi(f_0)\Int\varphi(f_1)\dots\Int\varphi(f_n).
 \]
 Let $m\in\mathbb{N}^+$ and let $g\in\mathcal{R}_\T^{\otimes{m}}$ be a pure tensor, then by \eqref{eq:GeneralizedShuffle} and \eqref{eq:Evaluation} we have
 \begin{multline*}
  \eta(1\otimes1\otimes\ep(f\odot{g})\otimes\et) = \eta(\E((1\otimes1\otimes1\otimes{f})\cdot(1\otimes1\otimes1\otimes{g})))\\
  = \E\left(\Int\varphi(f_1)\dots\Int\varphi(f_n)\right)\Int\varphi(g_1)\dots\Int\varphi(g_m).
 \end{multline*}
 Since $\mathrm{IDR}(\mathcal{R})$ is generated as $\mathcal{C}$-algebra by $\iota(\mathcal{R})$ and elements of the above forms, the above identities show that $\eta$ is unique.
\end{proof}

\begin{remark}\label{rem:closure1}
The image $\eta(\mathrm{IDR}(\mathcal{R})) \subseteq \mathcal{S}$ of the above integro-differential ring homomorphism given by the universal property of $\mathrm{IDR}(\mathcal{R})$ agrees with the internal integro-differential ring closure of $\varphi(\mathcal{R})$ in $(\mathcal{S},\Der,\Int)$.
This is straightforward because, for any integro-differential subring $\widetilde{\mathcal{R}}$ of $(\mathcal{S},\Der,\Int)$ that contains $\varphi(\mathcal{R})$, the universal property yields an integro-differential ring homomorphism $\tilde{\eta}:\mathrm{IDR}(\mathcal{R})\to\widetilde{\mathcal{R}}$ and, by uniqueness of $\eta$, we have $\eta(\mathrm{IDR}(\mathcal{R})) = \tilde{\eta}(\mathrm{IDR}(\mathcal{R})) \subseteq \widetilde{\mathcal{R}}$.
Trivially, the homomorphism theorem allows one to identify this internal closure $\eta(\mathrm{IDR}(\mathcal{R})) \subseteq \mathcal{S}$ with the quotient $\mathrm{IDR}(\mathcal{R})/\ker(\eta)$.
For the special case that $(\mathcal{S},\Der)$ is a differential ring extension of $(\mathcal{R},\Der)$ and $\varphi:=\id$ is the inclusion map of $\mathcal{R}$ into $\mathcal{S}$, we will analyze the quotient of $\mathrm{IDR}(\mathcal{R})$ in more detail in Section~\ref{sec:closure}.
\end{remark}

\subsection{Basic examples}
\label{sec:examples}

In the following, we illustrate and discuss further properties in simple special cases of the free commutative integro-differential ring constructed above.
To this end, we assume that $\mathcal{R}_\T$ is a cyclic $\mathcal{C}$-module and $x_1 \in \mathcal{R}_\T$ is such that $\mathcal{R}_\T=\mathcal{C}x_1$.
Then, we have $\mathcal{R}_\T^{\otimes{n}}=\mathcal{C}x_n$ with
\[
 x_n:=x_1\otimes\ldots\otimes{x_1}
\]
for all $n\in\mathbb{N}$, including the special case $x_0=\et$.
Consequently, the shuffle product satisfies $x_n\shuffle{x_m}=\binom{n+m}{n}x_{n+m}$.

If $\mathcal{R}_\T$ is freely generated by $x_1$, i.e.\ has rank $1$, then $\mathrm{IDR}(\mathcal{R})$ is freely generated as $\bar{\mathcal{C}}\iota(\mathcal{R})$-module by $\iota(1)=1\otimes1\otimes1\otimes{x_0}$ and all nested integrals $\Int\iota(x_1){\cdot}\dots{\cdot}\Int\iota(x_1)=1\otimes1\otimes1\otimes{x_n}$.
In particular, we have the following examples.

\begin{example}\label{ex:FreeConstant}
 Starting with the commutative differential ring $(\mathcal{R},\Der)=(\mathcal{C},0)$ having trivial derivation, the only quasi-integration is $\Q=0$.
 In particular, we have $\mathcal{R}_\mathrm{J}=\{0\}$ and $\mathcal{R}_\T=\mathcal{C}$ is cyclically generated by $x_1=1$.
 Then, $\mathrm{IDR}(\mathcal{R})$ is generated as $\bar{\mathcal{C}}$-module by $\{1\otimes1\otimes1\otimes{x_n}\ |\ n\in\mathbb{N}\}$.
 In fact, it is the free $\bar{\mathcal{C}}$-module generated by $\iota(1)$ and all nested integrals $\Int\iota(1){\cdot}\dots{\cdot}\Int\iota(1)$.
\end{example}

\begin{example}\label{ex:FreeLaurent}
 Starting with the ring of formal Laurent series $(\mathcal{R},\Der)=(\mathcal{C}((x)),\frac{d}{dx})$ where $\mathbb{Q}\subseteq\mathcal{C}$, it is natural to choose the quasi-integration $\Q$ such that $\mathcal{R}_\mathrm{J}$ consists of all series without constant term and $\mathcal{R}_\T=\mathcal{C}\frac{1}{x}$.
 Then, with $x_1=\frac{1}{x}$, $\mathrm{IDR}(\mathcal{R})$ is generated as $\bar{\mathcal{C}}$-module by $\{f\otimes1\otimes1\otimes{x_n}\ |\ f\in\mathcal{R},n\in\mathbb{N}\}$.
 In fact, it is the free $\mathcal{R}\otimes\mathcal{C}_1\otimes\mathcal{C}_2\otimes\mathcal{C}\et$-module generated by $\iota(1)$ and all nested integrals $\Int\iota(\frac{1}{x}){\cdot}\dots{\cdot}\Int\iota(\frac{1}{x})$.
 Recall the integro-differential ring $(\mathcal{S},\Der,\Int):=(\mathcal{C}((x))[\ln(x)],\frac{d}{dx},\Int)$ from Example~\ref{ex:IDR}, which contains $(\mathcal{R},\Der)$ as a differential subring.
 As a $\bar{\mathcal{C}}$-module, $\mathrm{IDR}(\mathcal{R})$ is isomorphic to $\bar{\mathcal{C}}\otimes\mathcal{S}$ via the straightforward module isomorphism given by $\mu(f\otimes{c}\otimes{x_n})=c\otimes{f\frac{\ln(x)^n}{n!}}$ for $f\in\mathcal{R}$ and $c \in \bar{\mathcal{C}}$.
 While $\mu$ satisfies $\mu(\iota(f))=f$ for all $f\in\mathcal{R}$ and even commutes with $\Der$, it is not a ring homomorphism, since it does not respect the product \eqref{eq:GeneralizedShuffle} on $\mathcal{C}_2\otimes\mathcal{T}$.
 In contrast, the unique integro-differential ring homomorphism $\eta:\mathrm{IDR}(\mathcal{R})\to\bar{\mathcal{C}}\otimes\mathcal{S}$ s.t.\ $\eta(\iota(f))=f$ for all $f\in\mathcal{R}$ from Theorem~\ref{thm:universal} is not injective, since we have e.g.\ $\eta(1\otimes\ep(x\otimes{x_1})\otimes1\otimes\et) = \eta(\E{x\Int\iota(\frac{1}{x})}) = \E{x\Int\frac{1}{x}} = \E{x\ln(x)} = 0$ by \eqref{eq:LaurentEval}.
 In Section~\ref{sec:closure}, we will analyze the kernel of such homomorphisms and determine it for this concrete setting in Example~\ref{ex:LaurentClosure}.
\end{example}

On the other hand, $\mathcal{R}_\T$ need not be a free $\mathcal{C}$-module of rank $1$.

\begin{example}\label{ex:FreeSurjective}
 Starting with a commutative differential ring that is already an integro-differential ring $(\mathcal{R},\Der,\Int)$, it is natural to choose $\Q:=\Int$ such that $\mathcal{R}_\mathrm{J}=\Int\mathcal{R}$ and $\mathcal{R}_\T$ is the trivial module, i.e.\ cyclically generated by $x_1=0$.
 Then, $\mathcal{T}=\mathcal{C}\et$ implies that $\mathcal{C}_1=\Sym(\ep(\mathcal{R}_\mathrm{J}))$ and $\mathcal{C}_2=\mathcal{C}$.
 So, $\mathrm{IDR}(\mathcal{R})=\mathcal{R}\otimes\mathcal{C}_1\otimes\mathcal{C}\otimes\mathcal{C}\et$ is just $\mathcal{R}$ extended by new constants $\Sym(\ep(\mathcal{R}_\mathrm{J}))$.
 Note that \eqref{eq:DefInt} implies $\Int\iota(f)=\iota(\Int{f})-1\otimes\ep(\Int{f})\otimes1\otimes\et$ for the integration on $\mathrm{IDR}(\mathcal{R})$.
\end{example}

Finally, we take a closer look at the relations \eqref{eq:ConstRel} of constants when $\mathcal{R}_\T$ is cyclic.
For $n,m\ge1$, we abbreviate the constants
\[
 c_{n,m}:=\ep(x_n\odot{x_m}).
\]
Setting $f=x_n$, $g=x_m$, and $h=x_l$ in \eqref{eq:ConstRel}, we immediately obtain that
\begin{multline}\label{eq:ConstRel1}
 \binom{n+m}{m}c_{n+m,l}-\binom{m+l}{m}c_{n,m+l}\,+\\
 +\sum_{j=0}^{m-1}\Bigg(\sum_{i=\max(0,1-j)}^{n-1}\binom{i+j}{j}c_{n-i,m-j}c_{i+j,l}-\sum_{k=\max(0,1-j)}^{l-1}\binom{j+k}{j}c_{n,j+k}c_{m-j,l-k}\!\Bigg)
\end{multline}
is zero for $n,m,l\ge1$.
Using $f=x_1$, $g=x_{m-1}$, and $h=x_n$ with $m\ge2$ instead, we can express \eqref{eq:ConstRel} as
\begin{equation}\label{eq:ConstRel2}
 mc_{m,n}-\binom{m+n-1}{m-1}c_{1,m+n-1}-\sum_{j=0}^{m-2}\sum_{k=1}^{n-1}\binom{j+k}{j}c_{1,j+k}c_{m-j-1,n-k}.
\end{equation}
In \cite{RaabRegensburger}, we used the relations \eqref{eq:ConstRel2} to show that all $c_{m,n}$ can be expressed in terms of all $c_{1,n}$ if $\mathbb{Q}\subseteq\mathcal{R}$.
In Section~\ref{sec:constants}, we will generalize this to arbitrary $\mathcal{R}_\mathrm{T}$ in Corollary~\ref{cor:gensC2}.
Working in the polynomial ring $\mathbb{Q}[c_{1,1},c_{1,2},c_{2,1},c_{1,3},\dots,c_{d-1,1}]$ generated by symbols $c_{m,n}$ with $m,n\ge1$ and $m+n \le d$, let $G_d$ be the set consisting of the polynomials \eqref{eq:ConstRel2} for $2 \le m \le n \le d-m$ together with $c_{n,m}-c_{m,n}$ for $1 \le m < n \le d-m$.
This set straightforwardly forms a lexicographic Gr\"obner basis for the ideal $(G_d)$ it generates, since the leading monomials of elements in $G_d$ are pairwise distinct variables $c_{m,n}$.
Furthermore, it can be checked for fixed $d$ by computer algebra software whether $(G_d)$ agrees with the ideal generated by all polynomials \eqref{eq:ConstRel1} with $n+m+l \le d$.
We verified that this is the case for $d=60$, which consumed about 121 days of CPU-time in addition to about 3 days of CPU-time needed for $d=50$.
This leads to the conjecture that it is the case for all $d\ge3$.
More generally, we formulate Conjecture~\ref{conj:freeC2} below.

\subsection{Relation to the shuffle algebra}
\label{sec:shufflealgebra}

The goal of this section is to obtain an algebra basis of $\mathrm{IDR}(\mathcal{R})$ based on an algebra basis of the shuffle algebra $\mathcal{T}$, see Corollary~\ref{cor:basis}.
This can be achieved by exploiting that the multiplication \eqref{eq:GeneralizedShuffle} defined on $\mathcal{C}_2\otimes\mathcal{T}$ is closely related to the shuffle product \eqref{eq:shufflerec} on $\mathcal{T}$.

By the definition of $\mathrm{IDR}(\mathcal{R})$, the module structure is straightforward, since the $\mathcal{C}$-module $\mathcal{T}$ is isomorphic to the $\mathcal{C}$-submodule $\mathcal{C}\otimes\mathcal{C}\otimes\mathcal{C}\otimes\mathcal{T}$ of $\mathrm{IDR}(\mathcal{R})$.
Hence, using the canonical isomorphism between these two modules, any set $B\subseteq\mathcal{T}$ that generates $\mathcal{T}$ as a $\mathcal{C}$-module also generates $\mathrm{IDR}(\mathcal{R})$ as a $\mathcal{R}\bar{\mathcal{C}}$-module and, if $B$ additionally is $\mathcal{C}$-linearly independent, then it is also $\mathcal{R}\bar{\mathcal{C}}$-linearly independent.
Note the change of the ground ring from $\mathcal{C}$ to $\mathcal{R}\bar{\mathcal{C}}$.
The algebra structure is not so straightforward, since considering $\mathcal{T}$ canonically as a $\mathcal{C}$-submodule of $\mathrm{IDR}(\mathcal{R})$ as above does not yield a $\mathcal{C}$-subalgebra as it is not closed under multiplication in $\mathrm{IDR}(\mathcal{R})$.
However, the multiplication \eqref{eq:GeneralizedShuffle} is closely related to the shuffle product, under which $\mathcal{T}$ is closed.
Corollary~\ref{cor:basis} shows under certain general conditions on $\mathcal{R}$ and $B$ that the corresponding set $\{1\otimes1\otimes1\otimes{b}\ |\ b\in{B}\}$ indeed generates $\mathrm{IDR}(\mathcal{R})$ freely as a $\mathcal{R}\bar{\mathcal{C}}$-algebra.
For example, if $\mathbb{Q}\subseteq\mathcal{C}$ and if $\mathcal{R}_\T$ is a free module of rank $1$ generated by some $x_1$ as in Examples \ref{ex:FreeConstant} and \ref{ex:FreeLaurent}, then $\{x_1\}$ trivially is a $\mathcal{C}$-algebra basis of $(\mathcal{T},+,\shuffle)$.
Moreover, we have that $\{1\otimes1\otimes1\otimes{x_1}\}$ is a $\mathcal{R}\bar{\mathcal{C}}$-algebra basis of $\mathrm{IDR}(\mathcal{R})$, cf.~\cite[Thm.~3.2]{RaabRegensburger}.

For the formal treatment in the general case, we exploit the graded structure of the shuffle algebra given by the submodules $\mathcal{R}_\T^{\otimes{n}}$.
We abbreviate for every $n\in\mathbb{N}$ the submodule $\mathcal{T}_n:=\bigoplus_{i=0}^n\mathcal{R}_\T^{\otimes{i}}$ of $\mathcal{T}$ and we say that $f\in\mathcal{T}$ has degree $\deg(f)=n$ if $n$ is minimal such that $f\in\mathcal{T}_n$.
We also extend the shuffle product $\shuffle$ to $\mathcal{C}_2\otimes\mathcal{T}$ by tensoring it with the multiplication on $\mathcal{C}_2$.
\begin{equation}\label{eq:extendedshuffle}
 (c\otimes{f})\shuffle(d\otimes{g}):=(c{\cdot}d)\otimes(f\shuffle{g})
\end{equation}
Note that this yields a commutative multiplication, in contrast to the analogous definition of $\circledast$ in \eqref{eq:tensormultiplication}.
In fact, this is the shuffle product on the $\mathcal{C}_2$-tensor algebra $T(\mathcal{C}_2\otimes\mathcal{R}_\T)$.
Furthermore, let $\pi_n:\mathcal{C}_2\otimes\mathcal{T}\to\mathcal{C}_2\otimes\mathcal{R}_\T^{\otimes{n}}$ be the canonical projection.
The straightforward observation formulated in the next lemma is key to what follows, so we prove it explicitly.

\begin{lemma}\label{lem:leadingtensor}
 Let $n,m\in\mathbb{N}$.
 Let $f\in\mathcal{C}_2\otimes\mathcal{R}_\T^{\otimes{n}}$ and $g\in\mathcal{C}_2\otimes\mathcal{R}_\T^{\otimes{m}}$, then $f\cdot{g}\in\mathcal{C}_2\otimes\mathcal{T}_{n+m}$ and
 \[
  \pi_{n+m}(f\cdot{g})=f\shuffle{g}.
 \]
 Moreover, for $f\in\mathcal{C}_2\otimes\mathcal{T}_n$ and $g\in\mathcal{C}_2\otimes\mathcal{T}_m$, we have $f\cdot{g},f\shuffle{g}\in\mathcal{C}_2\otimes\mathcal{T}_{n+m}$ as well as $\pi_{n+m}(f\cdot{g})=\pi_{n+m}(\pi_n(f)\cdot\pi_m(g))$ and $\pi_{n+m}(f\shuffle{g})=\pi_{n+m}(\pi_n(f)\shuffle\pi_m(g))$.
\end{lemma}
\begin{proof}
 Since the shuffle product of an element of $\mathcal{R}_\T^{\otimes{n}}$ and an element of $\mathcal{R}_\T^{\otimes{m}}$ lies in $\mathcal{R}_\T^{\otimes(n+m)}$, the first statement follows immediately from definitions \eqref{eq:GeneralizedShuffle} and \eqref{eq:extendedshuffle}.
 For the same reason, $f\cdot{g},f\shuffle{g}\in\mathcal{C}_2\otimes\mathcal{T}_{n+m}$ holds for $f\in\mathcal{C}_2\otimes\mathcal{T}_n$ and $g\in\mathcal{C}_2\otimes\mathcal{T}_m$.
 Based on this, the remaining equations follow trivially by linearity of $\pi_{n+m}$ and distributivity of $\cdot$ and $\shuffle$.
\end{proof}

From Lemma~\ref{lem:leadingtensor}, it follows immediately that $(\mathcal{C}_2\otimes\mathcal{T}_n)_{n\in\mathbb{N}}$ turns $(\mathcal{C}_2\otimes\mathcal{T},+,\cdot)$ into a filtered $\mathcal{C}_2$-algebra and that its associated graded $\mathcal{C}_2$-algebra is $(\mathcal{C}_2\otimes\mathcal{T},+,\shuffle)$, if the quotient module $(\mathcal{C}_2\otimes\mathcal{T}_{n+1})/(\mathcal{C}_2\otimes\mathcal{T}_n)$ is canonically identified with $\pi_{n+1}(\mathcal{C}_2\otimes\mathcal{T}_{n+1})$.

The properties given in Theorem~\ref{thm:associatedgraded} hold for commutative filtered $\mathcal{K}$-algebras in general.
Since we did not find a proof in the literature, we include it here explicitly.
Related results are stated as Corollary~7.6.14 in \cite{McConnellRobson} and as Remark~5.1.37 in \cite{Rowen}.
\begin{theorem}\label{thm:associatedgraded}
 Let $\mathcal{K}$ be a commutative ring, let $\mathcal{A}$ be a commutative filtered $\mathcal{K}$-algebra with (ascending) filtration $(F_n)_{n\in\mathbb{N}}$ and let $\mathcal{G}$ be its associated graded $\mathcal{K}$-algebra.
 On each $F_n$, denote by $\pi_n$ the canonical $\mathcal{K}$-linear map to the corresponding homogeneous component of $\mathcal{G}$.
 Let the map $\pi:\mathcal{A}\to\mathcal{G}$ be defined by $\pi(a):=\pi_n(a)$ with $n=\deg(a)$ being minimal such that $a \in F_n$.
 For any subset $B\subseteq\mathcal{A}$, the following hold.
 \begin{enumerate}
  \item If $\{\pi(b)\ |\ b\in{B}\}$ is a generating set of the $\mathcal{K}$-algebra $\mathcal{G}$, then $B$ is a generating set of the $\mathcal{K}$-algebra $\mathcal{A}$.
  \item If $\{\pi(b)\ |\ b\in{B}\}$ is algebraically independent in $\mathcal{G}$ over $\mathcal{K}$, then $B$ is algebraically independent in $\mathcal{A}$ over $\mathcal{K}$.
 \end{enumerate}
\end{theorem}
\begin{proof}
 Let $\mathcal{F}$ be the free commutative $\mathcal{K}$-algebra on $B$.
 We turn $\mathcal{F}$ into a graded $\mathcal{K}$-algebra by assigning to the elements of $B$ the degrees $\deg(b)$ inherited from $\mathcal{A}$.
 Let the $\mathcal{K}$-algebra homomorphisms $\varphi_1:\mathcal{F}\to\mathcal{A}$ and $\varphi_2:\mathcal{F}\to\mathcal{G}$ be defined via $\varphi_1(b)=b$ and $\varphi_2(b)=\pi(b)$, respectively.
 Then, both $\varphi_1$ and $\varphi_2$ respect the filtration, i.e.\ $\deg(p) \le n$ implies $\deg(\varphi_i(p)) \le n$.
 Moreover, if $p \in \mathcal{F}$ is homogeneous of degree $n$, we have $\varphi_2(p)=\pi_n(\varphi_1(p))$.
 In particular, $\varphi_2$ is graded, i.e.\ it maps homogeneous elements to homogeneous elements of the same degree or to zero.
 Now, we show that $\varphi_1$ is surjective resp.\ injective whenever $\varphi_2$ is.
 \begin{enumerate}
  \item Let $a\in\mathcal{A}$ be nonzero and let $n:=\deg(a)$.
   Since $\varphi_2$ is surjective, there is $p \in \mathcal{F}$ with $\pi_n(a)=\varphi_2(p)$.
   Since $\varphi_2$ is graded, we can choose $p$ to be homogeneous of degree $n$.
   Hence, we get $\varphi_2(p)=\pi_n(\varphi_1(p))$, which implies that $a-\varphi_1(p)$ is zero or has degree strictly less than $n$.
   Inductively, we obtain $a \in \im(\varphi_1)$.
  \item Let $p \in \mathcal{F}$ be nonzero.
   Let $\tilde{p} \in \mathcal{F}$ be the homogeneous component of $p$ of degree $n:=\deg(p)$.
   Then, we have $\varphi_2(\tilde{p})=\pi_n(\varphi_1(\tilde{p}))$ and $\pi_n(\varphi_1(p-\tilde{p}))=0$, since $p-\tilde{p}$ is zero or has degree strictly less than $n$.
   Since $\varphi_2$ is injective, $\varphi_2(\tilde{p})$ is nonzero.
   Altogether, this implies that $\varphi_1(p)$ is nonzero.\qedhere
 \end{enumerate}
\end{proof}

We also need the following basic fact about linear independence being preserved in tensor products, which we state for later reference.
This property follows immediately from Proposition~4.8.6 in \cite{CohnBasic}, for example.

\begin{lemma}\label{lem:tensorproduct}
 Let $\mathcal{K}$ be a commutative ring, let $N,M$ be $\mathcal{K}$-modules, and let the subset $\{m_i\ |\ i \in I\}\subseteq{M}$ be $\mathcal{K}$-linearly independent such that $\widetilde{M}:=\linspan_\mathcal{K}\{m_i\ |\ i \in I\}$ is complemented in $M$.
 Then, any element of $N\otimes_\mathcal{K}\widetilde{M} \subseteq N\otimes_\mathcal{K}M$ has a unique representation of the form $\sum_{i \in I}n_i\otimes_\mathcal{K}m_i$ with $n_i \in N$.
\end{lemma}

This allows us to obtain the following result relating $(\mathcal{C}_2\otimes\mathcal{T},+,\cdot)$ to the standard shuffle algebra $(\mathcal{T},+,\shuffle)$.

\begin{theorem}\label{thm:basis}
 Let $B\subseteq\mathcal{T}$ be a set of homogeneous elements.
 \begin{enumerate}
  \item If $B$ is a generating set of the $\mathcal{C}$-algebra $(\mathcal{T},+,\shuffle)$, then $\{1\otimes{b}\ |\ b\in{B}\}$ is a generating set of the $\mathcal{C}_2$-algebra $(\mathcal{C}_2\otimes\mathcal{T},+,\cdot)$.
  \item Assume that the $\mathcal{C}$-module $\mathcal{C}_2$ is semisimple or the $\mathcal{C}$-algebra generated by $B$ has a direct complement in $\mathcal{T}$.
   If $B$ is algebraically independent in $(\mathcal{T},+,\shuffle)$ over $\mathcal{C}$, then $\{1\otimes{b}\ |\ b\in{B}\}$ is algebraically independent in $(\mathcal{C}_2\otimes\mathcal{T},+,\cdot)$ over $\mathcal{C}_2$.
 \end{enumerate}
\end{theorem}
\begin{proof}
 Trivially, if $B$ generates the $\mathcal{C}$-algebra $(\mathcal{T},+,\shuffle)$, then $\{1\otimes{b}\ |\ b\in{B}\}$ generates the $\mathcal{C}_2$-algebra $(\mathcal{C}_2\otimes\mathcal{T},+,\shuffle)$.
 Now, the first claim follows by the first part of Theorem~\ref{thm:associatedgraded}.
\par
 The second claim is trivial, if $B$ is empty or $\mathcal{C}$ is the zero ring.
 Now, assume $B$ is nonempty and $\mathcal{C}$ (and hence also $\mathcal{T}$) is not zero.
 For proving the second claim, let $b_1,\dots,b_k \in B$ and let $p \in \mathcal{C}_2[z_1,\dots,z_k]$ be a nonzero polynomial $p=\sum_\alpha c_\alpha z^\alpha$.
 Since $p$ is nonzero, we can choose $\beta\in\mathbb{N}^k$ such that $c_\beta$ is nonzero.
 To show that the evaluation $p(b_1,\dots,b_k):=\sum_{\alpha}c_\alpha\mathop{\shuffle}_{i=1}^k(1\otimes{b_i})^{\shuffle\alpha_i}$ of $p$ at $z_i=1\otimes{b_i}$ in $(\mathcal{C}_2\otimes\mathcal{T},+,\shuffle)$ yields a nonzero element, we distinguish two cases.
\par
 First, assuming $\mathcal{C}_2$ is a semisimple $\mathcal{C}$-module, we can fix a direct complement of $\mathcal{C}c_\beta$ in $\mathcal{C}_2$.
 For each $\alpha$, this yields a unique $\tilde{c}_\alpha\in\mathcal{C}$ such that $c_\alpha-\tilde{c}_\alpha c_\beta$ lies in that complement.
 Note that $\tilde{c}_\beta=1$ and hence $\sum_\alpha \tilde{c}_\alpha z^\alpha$ is nonzero.
 Since $b_1,\dots,b_k$ are algebraically independent in $(\mathcal{T},+,\shuffle)$ over $\mathcal{C}$, we have that $\sum_{\alpha}\tilde{c}_\alpha c_\beta\mathop{\shuffle}_{i=1}^k(1\otimes{b_i})^{\shuffle\alpha_i}=c_\beta\otimes\sum_{\alpha}\tilde{c}_\alpha\mathop{\shuffle}_{i=1}^kb_i^{\shuffle\alpha_i}$ is nonzero.
 Consequently, $p(b_1,\dots,b_k)$ is nonzero as well, since the direct complement of $\mathcal{C}c_\beta$ in $\mathcal{C}_2$ carries over to a direct complement of $\mathcal{C}c_\beta\otimes\mathcal{T}$ in $\mathcal{C}_2\otimes\mathcal{T}$.
 Alternatively, if the $\mathcal{C}$-algebra generated by $B$ has a direct complement in $(\mathcal{T},+,\shuffle)$, we exploit that $p(b_1,\dots,b_k)=\sum_{\alpha}c_\alpha\otimes\mathop{\shuffle}_{i=1}^kb_i^{\shuffle\alpha_i}$.
 Since $B$ is algebraically independent in $(\mathcal{T},+,\shuffle)$ over $\mathcal{C}$, Lemma~\ref{lem:tensorproduct} implies that $p(b_1,\dots,b_k)$ is zero if and only if all $c_\alpha$ are zero.
\par
 Altogether, this shows that the set $\{1\otimes{b}\ |\ b\in{B}\}$ is algebraically independent in $(\mathcal{C}_2\otimes\mathcal{T},+,\shuffle)$ over $\mathcal{C}_2$.
 Then, the second part of Theorem~\ref{thm:associatedgraded} concludes the proof.
\end{proof}

\begin{corollary}\label{cor:basis}
 If a $\mathcal{C}$-algebra basis $B$ of $(\mathcal{T},+,\shuffle)$ consisting of homogeneous elements is given, then $\{1\otimes1\otimes1\otimes{b}\ |\ b\in{B}\}$ is a $\mathcal{R}\otimes\mathcal{C}_1\otimes\mathcal{C}_2\otimes\mathcal{C}\et$-algebra basis of $\mathrm{IDR}(\mathcal{R})$.
\end{corollary}

\begin{proof}
 Since $\mathcal{T}$ has the trivial complement in $\mathcal{T}$, Theorem~\ref{thm:basis} implies that $\{1\otimes{b}\ |\ b\in{B}\}$ is a $\mathcal{C}_2$-algebra basis of $(\mathcal{C}_2\otimes\mathcal{T},+,\cdot)$.
 Since $\mathrm{IDR}(\mathcal{R})$ is the $\mathcal{C}_2$-tensor product of the $\mathcal{C}_2$-modules $\mathcal{R}\otimes\mathcal{C}_1\otimes\mathcal{C}_2$ and $\mathcal{C}_2\otimes\mathcal{T}$, it follows by Lemma~\ref{lem:tensorproduct} that the power products of $\{1\otimes1\otimes1\otimes{b}\ |\ b\in{B}\}$ are a $\mathcal{R}\otimes\mathcal{C}_1\otimes\mathcal{C}_2$-module basis of $\mathrm{IDR}(\mathcal{R})$ whenever the power products of $\{1\otimes{b}\ |\ b\in{B}\}$ are a $\mathcal{C}_2$-module basis of $\mathcal{C}_2\otimes\mathcal{T}$.
 Hence, the claim follows.
\end{proof}

If $\mathcal{C}$ is a field of characteristic zero and an ordered $\mathcal{C}$-vector space basis of $\mathcal{R}_\T$ is chosen, then \cite{Radford} gives a $\mathcal{C}$-algebra basis $B$ of $(\mathcal{T},+,\shuffle)$ consisting of pure tensors.
Corollary~\ref{cor:basis} then yields a $\mathcal{R}\otimes\mathcal{C}_1\otimes\mathcal{C}_2\otimes\mathcal{C}\et$-algebra basis of $\mathrm{IDR}(\mathcal{R})$.

\begin{remark}\label{rem:summation}
Similarly, one can show that the associated graded algebras of free commutative Rota--Baxter algebras with weight are shuffle algebras as well.
For free commutative Rota--Baxter algebras, see \cite{GuoBook} and references therein.
They have been used in \cite{GuoRegensburgerRosenkranz} to construct the free commutative integro-differential algebra with weight (and with multiplicative evaluation).
Over fields of characteristic zero, the algebra basis of the shuffle algebra given in \cite{Radford} again carries over by arguments that are analogous to above.
Evidently, the sum-difference case is obtained by choosing weight $1$ or $-1$ as follows.
The usual forward difference operator $\Delta_n$ satisfies the Leibniz rule of weight $1$ and the summation operator $\sum_0^{n-1}$ is a Rota--Baxter operator of weight $1$ that satisfies $\Delta_n\sum_{i=0}^{n-1}a_i=a_n$ and $\sum_{i=0}^{n-1}\Delta_ia_i=a_n-a_0$ for any sequence $(a_k)$ in a commutative ring.
Likewise, the usual backward difference operator $\nabla_n$ satisfies the Leibniz rule of weight $-1$ and the summation operator $\sum_1^n$ is a Rota--Baxter operator of weight $-1$ that satisfies $\nabla_n\sum_{i=1}^na_i=a_n$ and $\sum_{i=1}^n\nabla_ia_i=a_n-a_0$.
So, the free commutative integro-differential algebra with weight $\pm1$ models all nested sums arising from a given commutative difference ring.
\end{remark}

\subsection{Linear independence of nested integrals}
\label{sec:independence}

Considering integrals over a given set $\{a_i\ |\ i\in{I}\}\subseteq\mathcal{R}$ of integrands, independence of these integrals is an important property.
Algebraic independence of (onefold) integrals \cite{Ostrowski} as well as linear independence of nested integrals \cite{DeneufchatelEtAl} have been studied in the literature.
In this subsection, we let $(\mathcal{S},\Der,\Int)$ be a commutative integro-differential ring such that $(\mathcal{S},\Der)$ is a differential ring extension of $(\mathcal{R},\Der)$ and we present a refined characterization of $\mathcal{R}$-linear independence of nested integrals in $\mathcal{S}$.
This independence over $\mathcal{R}$ relies only on the choice of integrands modulo $\Der\mathcal{R}$, hence we restrict to integrands from the complement $\mathcal{R}_\T$.

To conveniently discuss nested integrals over a chosen set of integrands, we use the following short-hand notations.
For any set $I$, the free monoid $\langle{I}\rangle$ consists of words $W=w_1\dots{w_n}$ of arbitrary length $|W|=n\in\mathbb{N}$ over the alphabet $I$.
Trivially, the unit element of concatenation of words in $\langle{I}\rangle$ is given by the empty word with $|W|=0$.
Fixing a subset $\{a_i\ |\ i\in{I}\}\subseteq\mathcal{R}_\mathrm{T}$, we define $\sigma:\langle{I}\rangle\to\mathcal{S}$ mapping words $W=w_1\dots{w_{|W|}} \in \langle{I}\rangle$ to nested integrals
\[
 \sigma(W):=\Int{a_{w_1}}\dots\Int{a_{w_{|W|}}},
\]
i.e.\ $\sigma(W)=1$ if $W$ is the empty word.
Analogously, we also define $\tau:\langle{I}\rangle\to\mathcal{T}$ by
\[
 \tau(W):=a_{w_1}\otimes\dots\otimes{a_{w_{|W|}}} \in \mathcal{T}
\]
to abbreviate tensors, where we have $\tau(W)=\et$ if $W$ is the empty word.

If $\mathcal{R}$ is a field of characteristic zero, then it follows by \cite[Thm.~1]{DeneufchatelEtAl} that $\{\sigma(W)\ |\ W \in \langle{I}\rangle\}$ is $\mathcal{R}$-linearly independent in $\mathcal{S}$ if $\{a_i\ |\ i\in{I}\}$ is $\mathcal{C}$-linearly independent.
For later use in Section~\ref{sec:closure}, we show in Theorem~\ref{thm:independence} by an adapted proof that this holds for more general $\mathcal{R}$ where every nonzero differential ideal contains a nonzero constant.
In particular, it holds for simple differential rings or, e.g., for Laurent series over any commutative Noetherian ring of coefficients, which follows from Lemma~\ref{lem:Noetherian}.

\begin{lemma}\label{lem:Noetherian}
 Let $\mathcal{C}$ be a commutative Noetherian ring with $\mathbb{Q}\subseteq\mathcal{C}$ and consider the usual derivation $\Der=\frac{d}{dx}$ on $\mathcal{C}((x))$.
 Then, for any nonzero $f\in\mathcal{C}((x))$, there are $n\in\mathbb{N}$, $m\in\mathbb{Z}$, and $b_0,\dots,b_n\in\mathcal{C}[[x]]$ such that $\sum_{i=0}^nb_i\Der^ix^{-m}f \in \mathcal{C}\setminus\{0\}$ with $x^{-m}f \in \mathcal{C}[[x]]$.
\end{lemma}
\begin{proof}
 For any nonzero $f=\sum_{i=m}^{\infty}f_ix^i\in\mathcal{C}((x))$, there is some $n\in\mathbb{N}$ such that any $f_i$ lies in the ideal $(f_m,\dots,f_{m+n}) \subseteq \mathcal{C}$.
 In particular, the coefficients $f_m,\dots,f_{m+n}$ are not all zero.
 After choosing $b_{0,0},\dots,b_{n,0} \in \mathcal{C}$ such that $\sum_{i=0}^ni!b_{i,0}f_{m+i}\neq0$, this allows us to choose $b_{0,j},\dots,b_{n,j} \in \mathcal{C}$ recursively for all $j\in\mathbb{N}$ such that $\sum_{i=0}^ni!b_{i,j}f_{m+i}=-\sum_{i=0}^n\sum_{k=1}^j\frac{(k+i)!}{k!}b_{i,j-k}f_{m+k+i}$.
 Finally, with these $b_{i,j}$, one can verify by a straightforward calculation that $\sum_{i=0}^n(\sum_{j=0}^{\infty}b_{i,j}x^j)\Der^ix^{-m}f$ is equal to $\sum_{i=0}^ni!b_{i,0}f_{m+i} \in \mathcal{C}\setminus\{0\}$.
\end{proof}

\begin{theorem}\label{thm:independence}
 Let $(\mathcal{S},\Der,\Int)$ be a commutative integro-differential ring such that $(\mathcal{S},\Der)$ is a differential ring extension of $(\mathcal{R},\Der)$.
 Let $\{a_i\ |\ i\in{I}\}\subseteq\mathcal{R}_\mathrm{T}$.
 \begin{enumerate}
  \item\label{item:independence} If $\{a_i\ |\ i\in{I}\}$ is $\mathcal{C}$-linearly independent, the $\mathcal{C}$-submodule generated by it is complemented in $\mathcal{R}_\mathrm{T}$, and the differential ring $(\mathcal{R},\Der)$ is such that, for every nonzero element, the differential ideal generated by it contains a nonzero constant, then $\{\sigma(W)\ |\ W \in \langle{I}\rangle\}\subseteq\mathcal{S}$ is $\mathcal{R}$-linearly independent.
  \item\label{item:converse} Conversely, if $\{\sigma(w)\ |\ w \in I\}\subseteq\mathcal{S}$ is $\mathcal{C}$-linearly independent and $\const_\Der(\mathcal{S})\mathcal{R}\subseteq\mathcal{S}$ is such that any finitely generated $\mathcal{R}$-submodule is contained in a free $\mathcal{R}$-submodule of $\const_\Der(\mathcal{S})\mathcal{R}$, then $\{a_i\ |\ i\in{I}\}\subseteq\mathcal{S}$ is $\const_\Der(\mathcal{S})$-linearly independent.
 \end{enumerate}
\end{theorem}
\begin{proof}
 First, assume that $\{a_i\ |\ i\in{I}\}$ is $\mathcal{C}$-linearly independent and generates a complemented submodule in $\mathcal{R}_\mathrm{T}$.
 Let us consider the $\mathcal{R}$-submodule
 \begin{equation}\label{eq:defMgen}
  M:=\mathcal{R}\otimes\mathcal{C}\otimes\mathcal{C}\otimes{T(\linspan_\mathcal{C}\{a_i\ |\ i\in{I}\})}\subseteq\mathrm{IDR}(\mathcal{R}).
 \end{equation}
 Then, by Lemma~\ref{lem:tensorproduct}, any $f \in M$ can be uniquely written as $f=\sum_{j=1}^rf_j\otimes1\otimes1\otimes\tau(W_j)$ where $f_j\in\mathcal{R}\setminus\{0\}$ and $W_j \in \langle{I}\rangle$ are pairwise distinct with $|W_j|\le|W_1|$.
 For such a representation, we let $l:=|W_1|$ and we let $n$ be the number of words $W_j$ with $|W_j|=l$.
 We let $\eta:\mathrm{IDR}(\mathcal{R})\to\mathcal{S}$ be the unique integro-differential ring homomorphism such that $\eta(\iota(f))=f$ for all $f \in \mathcal{R}$, given by Theorem~\ref{thm:universal}.
 Note that $\eta(f)=\sum_{i=1}^rf_j\sigma(W_j)$.
 Assume there is a nonzero $f=\sum_{i=1}^rf_j\otimes1\otimes1\otimes\tau(W_j) \in M$ such that $\eta(f)=0$.
 Let such an $f$ be chosen such that $l$ is minimal and $n$ is minimal for this $l$.
 By the assumption on $(\mathcal{R},\Der)$, there are $b_0,\dots,b_m\in\mathcal{R}$ such that $\sum_{k=0}^mb_k\Der^kf_1\in\mathcal{C}\setminus\{0\}$.
 So, by $\Der{M} \subseteq M$, we can assume w.l.o.g.\ that $f_1\in\mathcal{C}\setminus\{0\}$, because otherwise we can replace $f$ by $\sum_{k=0}^mb_k\Der^kf\in\ker(\eta)\cap{M}$, which satisfies $\sum_{k=0}^mb_k\Der^kf-\sum_{j=1}^r\big(\sum_{k=0}^mb_k\Der^kf_j\big)\otimes1\otimes1\otimes\tau(W_j) \in \mathcal{R}\otimes\mathcal{C}\otimes\mathcal{C}\otimes\bigoplus_{k=0}^{l-1}\mathcal{R}_\T^{\otimes{k}}$ by \eqref{eq:DefDeriv}.
 Now, $0=\eta(\Der{f})=\eta\big(\sum_{j=2}^r(\Der{f_j})\otimes1\otimes1\otimes\tau(W_j)+\sum_{j=1}^rf_j\Der(1\otimes1\otimes1\otimes\tau(W_j))\big)$.
 Since $\Der(1\otimes1\otimes1\otimes\tau(W_j)) \in \mathcal{R}\otimes\mathcal{C}\otimes\mathcal{C}\otimes\bigoplus_{k=0}^{l-1}\mathcal{R}_\T^{\otimes{k}}$ by \eqref{eq:DefDeriv}, we obtain $\Der{f}=0$ by minimality of $l$ and $n$.
 By $M\cap\bar{\mathcal{C}}=\iota(\mathcal{C})$, this implies $f\in\iota(\mathcal{C})$ and hence $l=0$, $r=n=1$, and $f=\iota(f_1)$.
 By definition of $\eta$, this implies that $\eta(f)$ is nonzero in contradiction to $\eta(f)=0$.
\par
 Conversely, assume that $\const_\Der(\mathcal{S})\mathcal{R}$ is a free $\mathcal{R}$-module and $\{\sigma(w)\ |\ w \in I\}\subseteq\mathcal{S}$ is $\mathcal{C}$-linearly independent.
 Let $c_1,\dots,c_n \in \const_\Der(\mathcal{S})$ and pairwise distinct $w_1,\dots,w_n \in I$ such that $\sum_{j=1}^nc_ja_{w_j}=0$.
 There are $\mathcal{R}$-linearly independent $d_1,\dots,d_m \in \const_\Der(\mathcal{S})\mathcal{R}$ such that we have $c_j=\sum_{k=1}^mc_{j,k}d_k$ with $c_{j,k}\in\mathcal{C}$ for all $j\in\{1,\dots,n\}$.
 Now, $0=\sum_{j=1}^nc_ja_{w_j} = \sum_{k=1}^md_k\sum_{j=1}^nc_{j,k}a_{w_j}$.
 Since $d_1,\dots,d_m$ are $\mathcal{R}$-linearly independent, we obtain $\sum_{j=1}^nc_{j,k}a_{w_j}=0$ for every $k\in\{1,\dots,m\}$.
 Applying $\Int$, it follows that all $c_{j,k}$ are zero.
\end{proof}

\begin{remark}\label{rem:independence}
On the one hand, the formulation of Theorem~\ref{thm:independence} is more general than \cite[Thm.~1]{DeneufchatelEtAl}, since we do not require $\mathcal{R}$ to be a field of characteristic zero and we assume only $\mathcal{C}$-linear independence of $\{a_i\ |\ i\in{I}\}$ resp.\ $\{\sigma(w)\ |\ w \in I\}$ instead of independence over $\const_\Der(\mathcal{S})$ resp.\ $\mathcal{R}$.
On the other hand, the formulation of Theorem~\ref{thm:independence} is less general than \cite[Thm.~1]{DeneufchatelEtAl}, since the latter requires only a commutative differential ring $(\mathcal{S},\Der)$ instead of an integro-differential ring $(\mathcal{S},\Der,\Int)$ and does not assume that $\linspan_\mathcal{C}\{a_i\ |\ i\in{I}\}$ lies in a complement of $\Der\mathcal{R}$.
In the proof of Theorem~\ref{thm:independence}, we use $M \subseteq \mathrm{IDR}(\mathcal{R})$ defined by \eqref{eq:defMgen} instead of $\mathcal{R}\langle\langle{I}\rangle\rangle=\mathcal{R}^{\langle{I}\rangle}$, but the statement and proof could be carried over, since the differential module $M$ is canonically isomorphic to $\mathcal{R}\langle{I}\rangle\subseteq\mathcal{R}\langle\langle{I}\rangle\rangle$ by $f\otimes1\otimes1\otimes\tau(W)\mapsto{fW}$.
\par
Note that, in \cite{DeneufchatelEtAl}, the proof that independence of nested integrals implies independence of $\{a_i\ |\ i\in{I}\}$ modulo $\Der\mathcal{R}$ actually quietly assumes that the nested integrals are $\const_\Der(\mathcal{S})\mathcal{R}$-linearly independent, while the statement of \cite[Thm.~1]{DeneufchatelEtAl} only assumes $\mathcal{R}$-linear independence.
Indeed, \cite[Thm.~1]{DeneufchatelEtAl} need not hold if $\const_\Der(\mathcal{S})\neq\mathcal{C}$.
For example, if the field $\mathcal{R}=\mathbb{Q}((x))$ and its purely transcendental ring extension $\mathcal{S}=\mathcal{R}[z,\ln(x)]$ are considered with the usual derivation $\Der=\frac{d}{dx}$, then $\const_\Der(\mathcal{S})=\mathcal{C}[z]$ differs from $\mathcal{C}=\mathbb{Q}$.
Letting $I=\{0,1\}$, $a_0=\frac{1}{x}$, and $a_1=a_0+1$, we then can choose integration constants $z^i$ in the nested integrals over $\{a_0,a_1\}$ to obtain $\mathcal{R}$-linearly independent antiderivatives $\ln(x),\ln(x)+x+z,\frac{1}{2}\ln(x)^2+z^2,\frac{1}{2}\ln(x)^2+x+z^3,\dots$ even though the integrands $\{a_0,a_1\}$ are not even $\mathcal{C}$-linearly independent modulo $\Der\mathcal{R}$.
\end{remark}

\subsection{The relations of constants in $\mathcal{C}_2$}
\label{sec:constants}

The definition \eqref{eq:DefC2} of the ring $\mathcal{C}_2$ as quotient of the symmetric algebra $\Sym(\ep(M_2))$ allows us to express some of these constants in terms of others via the relations \eqref{eq:ConstRel}.
In this section, we aim to identify the remaining constants, which can be used to express all others in $\mathcal{C}_2$.
To this end, we introduce an order on the generators of $\ep(M_2)$ to make use of the relations \eqref{eq:ConstRel} in a systematic way.

Our definition of this order will follow the three layers in the construction of elements of $\ep(M_2)$.
Since elements in $\ep(M_2)$ represent evaluations of products of nested integrals, the order will be based on an order of nested integrals.
Integrands appearing in these nested integrals are in $\mathcal{R}_\mathrm{T}$, i.e.\ they are not integrable in $\mathcal{R}$.
So, the order of nested integrals will be based on an order of these non-integrable elements.
Altogether, we start from an ordered basis of $\mathcal{R}_\mathrm{T}$, words formed from this basis give rise to nested integrals, and pairs of such words give rise to generators of $\ep(M_2)$.

Writing a relation \eqref{eq:ConstRel} in terms of these ordered generators of $\ep(M_2)$ allows us to split it into the greatest appearing generator and the remainder arising only from smaller generators of $\ep(M_2)$, cf.\ Lemma~\ref{lem:leadingmonomial}.
This {\lq\lq}leading term{\rq\rq} actually arises from maximal shuffles of words and appears with an integer coefficient.
If we can divide by this integer coefficient, then the relation enables to directly replace the {\lq\lq}leading term{\rq\rq} by the {\lq\lq}remainder{\rq\rq}.
Generators of $\ep(M_2)$ that cannot be reduced this way arise from pairs of words that involve Lyndon words, see Theorem~\ref{thm:leadingmonomial} below for details.

In the following syntactic analysis of relations, we assume that $\mathcal{R}_\mathrm{T}$ is a free $\mathcal{C}$-module, since the order of generators of $\ep(M_2)$ relies on an ordered basis of $\mathcal{R}_\mathrm{T}$.
Then, Corollary~\ref{cor:gensC2} below states the consequences for arbitrary $\mathcal{R}_\mathrm{T}$, which is the goal of this section.
Moreover, if $\mathcal{R}_\mathrm{T}$ is a free $\mathcal{C}$-module, we conjecture a free set of generators of $\mathcal{C}_2$ in Conjecture~\ref{conj:freeC2} below.

Now, when $\mathcal{R}_\mathrm{T}$ is a free $\mathcal{C}$-module, the module $M_2$ defined by \eqref{eq:DefM2} is a free $\mathcal{C}$-module as well, as is any isomorphic copy $\ep(M_2)$.
In this case, $\Sym(\ep(M_2))$ is a polynomial ring over $\mathcal{C}$ with its generating set being a $\mathcal{C}$-module basis of $\ep(M_2)$.
Explicitly, for any basis $\{a_i\ |\ i\in{I}\}\subseteq\mathcal{R}_\mathrm{T}$, $\Sym(\ep(M_2))$ is a (commutative) polynomial ring in the set of all
\[
 c(V,W):=\ep((a_{v_1}\otimes\dots\otimes{a_{v_{|V|}}})\odot(a_{w_1}\otimes\dots\otimes{a_{w_{|W|}}})) \in \ep(M_2),
\]
where $V=v_1\dots{v_{|V|}}$ and $W=w_1\dots{w_{|W|}}$ are nonempty words in the word monoid $\langle{I}\rangle$.
Since $c(V,W)=c(W,V)$ by commutativity of $\odot$, we obtain the same set if we consider only those $c(V,W)$ where $V$ is not greater than $W$.
More precisely, any total order $\le$ on the index set $I$ induces a degree-lexicographic order $\le_\mathrm{dlex}$ on $\langle{I}\rangle$ by comparing words first by their length and then, in case of a tie, according to the lexicographic order.
Moreover, on the set
\begin{equation}\label{eq:varsC2}
 \{c(V,W)\ | \ V,W\in\langle{I}\rangle^*,V\le_\mathrm{dlex}W\} \subseteq \ep(M_2)
\end{equation}
that generates the polynomial ring $\Sym(\ep(M_2))$, we define a total order $\preceq$ by ordering $(|V|+|W|,V,W)$ according to the lexicographic order on $\mathbb{N}\times\langle{I}\rangle^2$ arising from the usual order on $\mathbb{N}$ together with $(\langle{I}\rangle,\le_\mathrm{dlex})$.
For the remainder of this section, whenever we write $c(V,W)$, we implicitly assume that $V,W\in\langle{I}\rangle^*$ are such that $V\le_\mathrm{dlex}W$.

For shorter notation in the following, recall the definition of $\tau:\langle{I}\rangle\to\mathcal{T}$ from Section~\ref{sec:independence}.
In addition, for $V_1,V_2,V_3\in\langle{I}\rangle^*$, we abbreviate by $r(V_1,V_2,V_3)$ the element of $\Sym(\ep(M_2))$ obtained by setting $f=\tau(V_1)$, $g=\tau(V_2)$, and $h=\tau(V_3)$ in \eqref{eq:ConstRel}.
By commutativity of $\shuffle$ and $\odot$, swapping $f$ and $h$ in \eqref{eq:ConstRel} just changes sign of the element given by \eqref{eq:ConstRel}.
Similarly, \eqref{eq:ConstRel} is zero in $\Sym(\ep(M_2))$ whenever $f=h$, since every summand in \eqref{eq:ConstRel} is compensated by a duplicate with opposite sign in that case.
Since \eqref{eq:ConstRel} is $\mathcal{C}$-multilinear in $f,g,h$, it follows straightforwardly that the ideal $J_2$ defining $\mathcal{C}_2$ by \eqref{eq:DefC2} is generated by
\begin{equation}\label{eq:gensJ2}
 \{r(V_1,V_2,V_3)\ |\ V_1,V_2,V_3\in\langle{I}\rangle^*,V_1<_\mathrm{dlex}V_3\} \subseteq \Sym(\ep(M_2))
\end{equation}
whenever $\{a_i\ |\ i\in{I}\}$ generates $\mathcal{R}_\mathrm{T}$ as a $\mathcal{C}$-module.

Closely related to the shuffle product of tensors, in the following, we also need to argue about shuffles of words.
For words $V,W\in\langle{I}\rangle$, the set $\sh(V,W)\subseteq\langle{I}\rangle$ of shuffles of $V$ and $W$ can be defined by $\sh(V,W):=\{VW\}$, if $V$ or $W$ is empty, and recursively by
\[
 \sh(V,W):=\{v_1U\ |\ U\in{\sh(v_2\dots{v_{|V|}},W)}\}\cup\{w_1U\ |\ U\in{\sh(V,w_2\dots{w_{|W|}})}\}
\]
if $V,W$ are nonempty.
Note that every word in $\sh(V,W)$ has length $|V|+|W|$.
By \eqref{eq:shufflerec}, the relation with the shuffle product $\shuffle$ is obviously that
\[
 \tau(V)\shuffle\tau(W)=\sum_{U\in{\sh(V,W)}}c_U\tau(U)
\]
with coefficients $c_U$ being positive integers depending on $V$ and $W$.

\begin{lemma}\label{lem:leadingmonomial}
 Assume that $\mathcal{R}_\mathrm{T}$ is a free $\mathcal{C}$-module with basis $\{a_i\ |\ i\in{I}\}\subseteq\mathcal{R}_\mathrm{T}$.
 Let $\le$ be a total order on $I$.
 Consider the sets \eqref{eq:varsC2} and \eqref{eq:gensJ2} defined above and let $\preceq$ be the total order on \eqref{eq:varsC2} defined above.
 Then, for any $V_1,V_2,V_3\in\langle{I}\rangle^*$ with $V_1<_\mathrm{dlex}V_3$, the greatest $c(V,W)$ appearing in $r(V_1,V_2,V_3)$ arises only from the term $\ep((f\shuffle{g})\odot{h})$ in \eqref{eq:ConstRel}.
 In particular, this $c(V,W)$ appears linearly in $r(V_1,V_2,V_3)$ with integer coefficient and is not multiplied by any other elements of \eqref{eq:varsC2}.
\end{lemma}
\begin{proof}
 Let $V_1,V_2,V_3\in\langle{I}\rangle^*$ with $V_1<_\mathrm{dlex}V_3$ and let $f:=\tau(V_1)$, $g:=\tau(V_2)$, and $h:=\tau(V_3)$ so that $r(V_1,V_2,V_3)$ equals \eqref{eq:ConstRel}.
 From the properties of the shuffle product, it is immediate that those $c(V,W)$ in \eqref{eq:varsC2} arising from the terms $\ep((f\shuffle{g})\odot{h})$ and $\ep(f\odot(g\shuffle{h}))$ satisfy $|V|+|W|=|V_1|+|V_2|+|V_3|$, while those arising from the remaining terms in \eqref{eq:ConstRel} satisfy $|V|+|W|<|V_1|+|V_2|+|V_3|$.
 Consequently, the greatest element of \eqref{eq:varsC2} appearing in $r(V_1,V_2,V_3)$ arises from $\ep((f\shuffle{g})\odot{h})-\ep(f\odot(g\shuffle{h}))$, if $\ep((f\shuffle{g})\odot{h})\neq\ep(f\odot(g\shuffle{h}))$.
 Note that $\ep((f\shuffle{g})\odot{h})-\ep(f\odot(g\shuffle{h}))$ is a linear combination of elements of \eqref{eq:varsC2} with integer coefficients and we have that $\ep((f\shuffle{g})\odot{h})\neq0$.
 In the following, we will show that every element of \eqref{eq:varsC2} arising from $\ep((f\shuffle{g})\odot{h})$ is strictly greater w.r.t.\ $\preceq$ than any element arising from $\ep(f\odot(g\shuffle{h}))$.
 Altogether, this allows us to conclude that $\ep((f\shuffle{g})\odot{h})\neq\ep(f\odot(g\shuffle{h}))$ and that the greatest element of \eqref{eq:varsC2} appearing in $r(V_1,V_2,V_3)$ indeed arises from $\ep((f\shuffle{g})\odot{h})$.
\par
 First, from $V_1<_\mathrm{dlex}V_3$ and $|V_2|\ge1$, it follows that $|V_1|<|V_2|+|V_3|$.
 Hence, all $c(V,W)$ with $V\le_\mathrm{dlex}W$ arising from $\ep(f\odot(g\shuffle{h}))$ are such that $V=V_1$ and $W$ is a shuffle of $V_2$ and $V_3$, i.e.\ $|V|=|V_1|<|V_2|+|V_3|=|W|$.
 To complete the proof, we distinguish three cases to show that every element $c(V,W)$ of \eqref{eq:varsC2} arising from $\ep((f\shuffle{g})\odot{h})$ satisfies $V_1<_\mathrm{dlex}V$ and hence is strictly greater than any element arising from $\ep(f\odot(g\shuffle{h}))$.
\par
 If $|V_1|+|V_2|<|V_3|$, then all $c(V,W)$ with $V\le_\mathrm{dlex}W$ arising from $\ep((f\shuffle{g})\odot{h})$ are such that $W=V_3$ and $V$ is a shuffle of $V_1$ and $V_2$, i.e.\ $|V|=|V_1|+|V_2|<|V_3|=|W|$.
 Since $|V_2|\ge1$, this implies that $|V_1|<|V|$, i.e.\ $V_1<_\mathrm{dlex}V$.
 If $|V_1|+|V_2|=|V_3|$, then all $c(V,W)$ with $V\le_\mathrm{dlex}W$ arising from $\ep((f\shuffle{g})\odot{h})$ satisfy $|V|=|W|=|V_3|$.
 Since $|V_2|\ge1$, we have $|V_1|<|V|$ again.
 If $|V_1|+|V_2|>|V_3|$, then all $c(V,W)$ with $V\le_\mathrm{dlex}W$ arising from $\ep((f\shuffle{g})\odot{h})$ are such that $V=V_3$ and $W$ is a shuffle of $V_1$ and $V_2$, i.e.\ $|V|=|V_3|<|V_1|+|V_2|=|W|$.
 The general assumption $V_1<_\mathrm{dlex}V_3$ directly yields $V_1<_\mathrm{dlex}V$.
\end{proof}

If $\mathbb{Q}\subseteq\mathcal{C}$, then Lemma~\ref{lem:leadingmonomial} implies that the relation $r(V_1,V_2,V_3)$ allows us to express the greatest $c(V,W)$ appearing in it by smaller elements of \eqref{eq:varsC2} modulo the ideal $J_2$, i.e.\ when working in $\mathcal{C}_2$.
Therefore, it is relevant to know if a given $c(V,W)$ arises as the greatest one in $r(V_1,V_2,V_3)$ for some $V_1,V_2,V_3\in\langle{I}\rangle^*$ with $V_1<_\mathrm{dlex}V_3$.
Theorem~\ref{thm:leadingmonomial} below answers precisely this question.
Before, we need to recall Lyndon words and discuss their relation with shuffling.
\par
Recall that $W \in \langle{I}\rangle$ is a Lyndon word if and only if it is not empty and all proper nonempty suffixes of $W$ are lexicographically greater than $W$, see e.g.\ \cite[p.~105]{ReutenauerFree}.
A characterization in terms of shuffles is that $W \in \langle{I}\rangle$ is a Lyndon word if and only if it is neither empty nor the (lexicographically) maximal shuffle of two nonempty words $W_1,W_2\in \langle{I}\rangle^*$.
This follows from Theorem~5.1.5 in \cite{Lothaire} together with Theorem~6.1.ii in \cite{ReutenauerFree}, for example, and will be used in the sequel.
A related characterization, which additionally imposes $W=W_1W_2$, is given e.g.\ as Corollary~6.2 in \cite{ReutenauerFree}.

\begin{theorem}\label{thm:leadingmonomial}
 Assume that $\mathcal{R}_\mathrm{T}$ is a free $\mathcal{C}$-module with basis $\{a_i\ |\ i\in{I}\}\subseteq\mathcal{R}_\mathrm{T}$.
 Let $\le$ be a total order on $I$.
 Consider the sets \eqref{eq:varsC2} and \eqref{eq:gensJ2} defined above and let $\preceq$ be the total order on \eqref{eq:varsC2} defined above.
 Let $V,W \in \langle{I}\rangle^*$ with $V\le_\mathrm{dlex}W$, then the following are equivalent.
 \begin{enumerate}
  \item There are $V_1,V_2,V_3\in\langle{I}\rangle^*$ with $V_1<_\mathrm{dlex}V_3$ such that $c(V,W)$ is the greatest element of \eqref{eq:varsC2} appearing in $r(V_1,V_2,V_3)$.
  \item $V$ is not a Lyndon word or $W$ is the maximal shuffle of some $W_1,W_2 \in \langle{I}\rangle^*$ with $W_1<_\mathrm{dlex}V$.
 \end{enumerate}
\end{theorem}
\begin{proof}
 First, assume that $V_1,V_2,V_3\in\langle{I}\rangle^*$ with $V_1<_\mathrm{dlex}V_3$ are such that $c(V,W)$ is the greatest element of \eqref{eq:varsC2} appearing in $r(V_1,V_2,V_3)$.
 By Lemma~\ref{lem:leadingmonomial}, $c(V,W)$ arises from the term $\ep((f\shuffle{g})\odot{h})$, where $f:=\tau(V_1)$, $g:=\tau(V_2)$, and $h:=\tau(V_3)$.
 All elements $c(\tilde{V},\tilde{W})$ arising from this $\ep((f\shuffle{g})\odot{h})$ either have $\tilde{W}=V_3$ and $\tilde{V}$ is a shuffle of $V_1$ and $V_2$ or they have $\tilde{V}=V_3$ and $\tilde{W}$ is a shuffle of $V_1$ and $V_2$.
 Hence, there are two cases how $V$ and $W$ might be formed from $V_1,V_2,V_3$.
 If the maximal shuffle $V_0$ of $V_1$ and $V_2$ satisfies $V_0\le_\mathrm{dlex}V_3$, then all elements $c(\tilde{V},\tilde{W})$ arising from $\ep((f\shuffle{g})\odot{h})$ have $\tilde{W}=V_3$ and $\tilde{V}$ is a shuffle of $V_1$ and $V_2$.
 In particular, $V=V_0$ and $W=V_3$.
 So, $V$ is not a Lyndon word in this case.
 If the maximal shuffle $V_0$ of $V_1$ and $V_2$ satisfies $V_0>_\mathrm{dlex}V_3$, then all elements $c(\tilde{V},\tilde{W})$ arising from $\ep((f\shuffle{g})\odot{h})$ satisfy $\tilde{V}\le_\mathrm{dlex}V_3$.
 In particular, $V=V_3$ and $W=V_0$.
 So, $W$ is the maximal shuffle of $V_1$ and $V_2$ and we have $V_1<_\mathrm{dlex}V$.
\par
 For the converse, we again distinguish two cases.
 If $V$ is not a Lyndon word, then we can choose $V_1,V_2\in\langle{I}\rangle^*$ such that $V$ is the maximal shuffle of $V_1$ and $V_2$ and we set $V_3:=W$.
 From $V\le_\mathrm{dlex}W$ and $|V|=|V_1|+|V_2|$ it follows that $|V_1|<|W|$, i.e.\ $V_1<_\mathrm{dlex}V_3$.
 So, from the first part of the proof, we obtain that $c(V,W)$ is the greatest element of \eqref{eq:varsC2} appearing in $r(V_1,V_2,V_3)$.
 If $V$ is a Lyndon word and $W$ is the maximal shuffle of some $W_1,W_2 \in \langle{I}\rangle^*$ with $W_1<_\mathrm{dlex}V$, then we set $V_1:=W_1$, $V_2:=W_2$, and $V_3:=V$, which trivially satisfy $V_1<_\mathrm{dlex}V_3$.
 Since $V$ is not the maximal shuffle of any two nonempty words, we have $V<_\mathrm{dlex}W$, i.e.\ the maximal shuffle $V_0=W$ of $V_1$ and $V_2$ satisfies $V_0>_\mathrm{dlex}V_3$.
 So, from the first part of the proof, we obtain that $c(V,W)$ is the greatest element of \eqref{eq:varsC2} appearing in $r(V_1,V_2,V_3)$.
\end{proof}

For brevity, given a totally ordered set $(I,\le)$, we let
\begin{multline}\label{eq:gensC2}
 S:=\{(V,W)\in(\langle{I}\rangle^*)^2\ |\ \text{$V\le_\mathrm{dlex}W$, $V$ is a Lyndon word, and $W$ is not}\\
 \text{the maximal shuffle of some $W_1,W_2 \in \langle{I}\rangle^*$ with $W_1<_\mathrm{dlex}V$}\}.
\end{multline}
Altogether, this straightforwardly yields the following.

\begin{corollary}\label{cor:gensC2}
 Let $\mathbb{Q}\subseteq\mathcal{C}$ and let $\mathcal{R}_\mathrm{T}$ be generated as a $\mathcal{C}$-module by $\{a_i\ |\ i\in{I}\}\subseteq\mathcal{R}_\mathrm{T}$.
 Let $\le$ be a total order on $I$ and consider the sets \eqref{eq:varsC2} and \eqref{eq:gensC2} defined above.
 Then, $\mathcal{C}_2$ is generated as $\mathcal{C}$-algebra by the set $\{c(V,W)\ |\ (V,W)\in{S}\}$.
\end{corollary}
\begin{proof}
 Let $\widetilde{\mathcal{R}}_\mathrm{T}$ be the free $\mathcal{C}$-module generated by $\{a_i\ |\ i\in{I}\}$.
 Let the $\mathcal{C}$-module $\widetilde{M}_2:=\big(\bigoplus_{n=1}^\infty \widetilde{\mathcal{R}}_\T^{\otimes{n}}\big)^{\odot2}$ be defined analogously to \eqref{eq:DefM2} and let $\ep(\widetilde{M}_2)$ be an isomorphic copy.
 Analogously to $J_1\subseteq\Sym(\ep(M_2))$, let the ideal $\tilde{J}_1\subseteq\Sym(\ep(\widetilde{M}_2))$ be generated by the elements \eqref{eq:ConstRel} in $\Sym(\ep(\widetilde{M}_2))$.
 For words $V,W\in\langle{I}\rangle^*$ with $V\le_\mathrm{dlex}W$, we define $\tilde{c}(V,W):=\ep((a_{v_1}\otimes\dots\otimes{a_{v_{|V|}}})\odot(a_{w_1}\otimes\dots\otimes{a_{w_{|W|}}})) \in \ep(\widetilde{M}_2)$ analogously to $c(V,W)\in\ep(M_2)$.
 By Theorem~\ref{thm:leadingmonomial}, we have that
 \[
  \Sym(\ep(\widetilde{M}_2))=\Sym(\linspan_\mathcal{C}\{\tilde{c}(V,W)\ |\ (V,W)\in{S}\})+\tilde{J}_1.
 \]
 The canonical $\mathcal{C}$-module epimorphism $\widetilde{\mathcal{R}}_\mathrm{T}\to\mathcal{R}_\mathrm{T}$ yields a canonical $\mathcal{C}$-module epimorphism $\psi:\ep(\widetilde{M}_2)\to\ep(M_2)$ mapping each $\tilde{c}(V,W)$ to the corresponding $c(V,W)$.
 The latter epimorphism extends uniquely to a $\mathcal{C}$-algebra epimorphism $\psi:\Sym(\ep(\widetilde{M}_2))\to\Sym(\ep(M_2))$.
 In particular, $\psi$ hence maps $\tilde{J}_1$ onto $J_1$.
 Therefore, by applying $\psi$ to the above identity, we obtain
 \[
  \Sym(\ep(M_2))=\Sym(\linspan_\mathcal{C}\{c(V,W)\ |\ (V,W)\in{S}\})+J_1.
 \]
 Taking the quotient by $J_1$ concludes the proof.
\end{proof}

We conjecture that the statement of Corollary~\ref{cor:gensC2} remains true when generation is replaced by free generation.
\begin{conjecture}\label{conj:freeC2}
 Let $\mathbb{Q}\subseteq\mathcal{C}$ and let $\mathcal{R}_\mathrm{T}$ be freely generated by $\{a_i\ |\ i\in{I}\}\subseteq\mathcal{R}_\mathrm{T}$ as a $\mathcal{C}$-module.
 Let $\le$ be a total order on $I$ and consider the sets \eqref{eq:varsC2} and \eqref{eq:gensC2} defined above.
 Then, $\mathcal{C}_2$ is freely generated as commutative $\mathcal{C}$-algebra by the set $\{c(V,W)\ |\ (V,W)\in{S}\}$.
\end{conjecture}
Using computer algebra software, we have computationally verified restricted versions of this conjecture where $\ep(M_2)$ is truncated at finite {\lq\lq}weight{\rq\rq}.
For example, for free $\mathcal{R}_\mathrm{T}$ of arbitrary rank, we have verified the conjecture up to weight $6$, i.e.\ that the quotient of $\Sym(\ep(\bigoplus_{w=2}^6\bigoplus_{i=1}^{\lfloor{w/2}\rfloor}\mathcal{R}_\mathrm{T}^{\otimes{i}}\odot\mathcal{R}_\mathrm{T}^{\otimes{w-i}}))$ by the ideal generated by all $r(V_1,V_2,V_3)$ with $|V_1|+|V_2|+|V_3|\le6$ is indeed freely generated by those $c(V,W)$ with $|V|+|W|\le6$ which satisfy $(V,W)\in{S}$.
This verification utilizes the polynomial ring generated over $\mathbb{Q}$ by indeterminates $c(V,W)$ with $|V|+|W|\le6$ and $|I|=6$, where we order these $134499$ indeterminates by $\preceq$, just like \eqref{eq:varsC2} above.
Any $r(V_1,V_2,V_3)$ with $|V_1|+|V_2|+|V_3|\le6$ canonically is considered an element of this ring by replacing the free generators $c(V,W)$ of $\ep(M_2)$ by the corresponding indeterminates $c(V,W)$.
In analogy to \eqref{eq:ConstRel2}, we choose a maximal subset $G \subseteq \{r(V_1,V_2,V_3)\ |\ |V_1|+|V_2|+|V_3|\le6\}$ such that each $c(V,W)$ arises at most for one element of $G$ as the maximal indeterminate appearing within that element of $G$.
By Lemma~\ref{lem:leadingmonomial} and Theorem~\ref{thm:leadingmonomial}, it follows that the polynomial subring generated by indeterminates $c(V,W)$ with $(V,W)\in{S}$ and $|V|+|W|\le6$ is a direct complement of the ideal generated by $G$.
Then, it suffices to verify computationally that all remaining $r(V_1,V_2,V_3)$ with $|V_1|+|V_2|+|V_3|\le6$ lie in the ideal generated by $G$.
Altogether, this computation took 21 hours on a laptop, consuming 93 hours of CPU-time.
Analogously, for free $\mathcal{R}_\mathrm{T}$ of rank $1$, $2$, and $3$, we have verified the conjecture up to weight $60$, $12$, and $9$, respectively.

\section{Imposing additional relations of constants}
\label{sec:Ext}

Having a free commutative integro-differential ring as presented in the previous section, one can obtain further integro-differential rings by forming quotients by integro-differential ideals.
In differential algebra, differential ideals are well-known as ideals that are closed under derivation.
Integro-differential ideals are defined as ideals that are closed under both $\Der$ and $\Int$.
In general, any (ring) ideal generated by constants in any integro-differential ring is automatically an integro-differential ideal, since integration $\Int$ is linear over constants.
We will use this fact implicitly in the following.

When factoring a differential ring $(\mathcal{R},\Der)$ by a differential ideal $J\subseteq\mathcal{R}$,  a differential ring $(\mathcal{R}/J,\tilde{\Der})$ is obtained where all constants in $\mathcal{R}$ give rise to (not necessarily distinct) constants in the quotient.
However, not all constants of the quotient differential ring necessarily arise this way, as can be illustrated by differential polynomials $\mathcal{R}=\mathcal{C}\{y\}$ and the differential ideal $J$ generated by $\Der{y}$, which yields a quotient consisting only of constants while $y-c\not\in J$ for all $c\in\mathcal{C}$.

For integro-differential rings $(\mathcal{R},\Der,\Int)$, this is different.
Since integro-differential rings form a variety as mentioned in Section~\ref{sec:prelim}, the quotient $(\mathcal{R}/J,\tilde{\Der},\tilde{\Int})$ by an integro-differential ideal $J\subseteq\mathcal{R}$ is again an integro-differential ring and the canonical map $\pi:\mathcal{R}\to\mathcal{R}/J$ is an integro-differential ring homomorphism.
In particular, this implies that all constants in the quotient arise from constants of the original ring, since $\const_{\tilde{\Der}}(\mathcal{R}/J) = \tilde{\E}\pi(\mathcal{R}) = \pi(\E\mathcal{R})$.
Moreover, every integro-differential ring homomorphism $\psi:\mathcal{R}\to\mathcal{S}$ with $J\subseteq\ker(\psi)$ induces an integro-differential ring homomorphism $\bar{\psi}:\mathcal{R}/J\to\mathcal{S}$ and every integro-differential ring homomorphism $\bar{\psi}:\mathcal{R}/J\to\mathcal{S}$ arises this way.

In applications, one often has an explicitly computable quasi-integration $\Q:\mathcal{R}\to\mathcal{R}$ in addition to the differential ring $(\mathcal{R},\Der)$.
In such cases, a natural requirement for the integration defined on the integro-differential closure is to reproduce the antiderivatives that already exist in $\mathcal{R}$ and are given by $\Q$, i.e.\ $\Int\iota(f)=\iota(\Q{f})$ for all $f\in\Der\mathcal{R}$.
Obviously, this could be achieved by factoring $\mathrm{IDR}(\mathcal{R})$ by the integro-differential ideal generated by these relations.
Instead of factoring, a more direct approach can be taken by slightly modifying the construction of $\mathrm{IDR}(\mathcal{R})$ given in Section~\ref{sec:free}, which we will detail below in Section~\ref{sec:ExtQIDR}.
For $f\in\Der\mathcal{R}$, we have $\Der\Q{f}=f$ and hence $\Int\iota(f)=\iota(\Q{f})$ is equivalent to $\ep(\Q{f})=0$ by \eqref{eq:DefInt}.
So, the above integro-differential ideal is in fact generated by constants
\begin{equation}\label{eq:constgenerators}
 \mathcal{C}\otimes\ep(\Q\Der\mathcal{R})\otimes\mathcal{C}\otimes\mathcal{C}\et \subseteq \bar{\mathcal{C}}.
\end{equation}

More generally, if in a concrete setting additional information can be inferred about how evaluation interacts with multiplication of nested integrals modelled by $\mathcal{T}$, we may want to factor out other new constants in $\bar{\mathcal{C}}\setminus\iota(\mathcal{C})$ introduced by the construction of the free object $\mathrm{IDR}(\mathcal{R})$.
For instance, imposing multiplicativity of evaluation is explained in Remark~\ref{rem:IDA} below.
Later, in Section~\ref{sec:closure}, we treat the case that evaluation is fully determined by a given integro-differential ring $(\mathcal{S},\Der,\Int)$ that has $(\mathcal{R},\Der)$ as differential subring of $(\mathcal{S},\Der)$.

Evidently, we can compute in the quotient of $\mathrm{IDR}(\mathcal{R})$ by an arbitrary integro-differential ideal by choosing representatives in $\mathrm{IDR}(\mathcal{R})$ and modifying all operations accordingly.
However, in the special situation treated by the following lemma, all operations except integration can remain exactly the same.
This situation arises for example when, on a complemented $\mathcal{C}$-submodule $N_1 \subseteq \ep(M_1)$, we can assign values via a $\mathcal{C}$-module homomorphism $\varphi_1:N_1\to\mathcal{C}$, which, in the lemma below, yields $C=\{c-\varphi_1(c)\ |\ c\in{N}\}$ and $\mathcal{D}$ generated by a complement of $N_1$.

\begin{lemma}\label{lem:FactorConst}
 Let $C,\mathcal{D} \subseteq \mathcal{C}_1$ s.t.\ $\mathcal{D}$ is a subalgebra and the ideal $(C)$ satisfies $\mathcal{C}_1=(C)\oplus\mathcal{D}$.
 Let $\bar{\pi}:=\id_{\mathcal{R}}\otimes\pi_{\mathcal{D}}\otimes\id_{\mathcal{C}_2\otimes\mathcal{T}}$ where $\pi_{\mathcal{D}}:\mathcal{C}_1\to\mathcal{D}$ is the canonical projection.
 Let $J_1 \subseteq \mathrm{IDR}(\mathcal{R})$ be the integro-differential ideal generated by $\{1\otimes{c}\otimes1\otimes\et\ |\ c\in{C}\}$ and let $\pi:\mathrm{IDR}(\mathcal{R})\to\mathrm{IDR}(\mathcal{R})/J_1$ be the canonical projection.
 Then, 
 \begin{enumerate}
  \item $(\mathcal{R}\otimes\mathcal{D}\otimes\mathcal{C}_2\otimes\mathcal{T},\Der,\bar{\pi}\circ\Int)$ is an integro-differential ring and
  \item $\psi:\mathrm{IDR}(\mathcal{R})/J_1\to\mathcal{R}\otimes\mathcal{D}\otimes\mathcal{C}_2\otimes\mathcal{T}$ defined by $\psi(\pi(f)):=\bar{\pi}(f)$ is a bijective integro-differential ring homomorphism.
 \end{enumerate}
\end{lemma}
\begin{proof}
 By assumption, the canonical projection $\pi_{\mathcal{D}}:\mathcal{C}_1\to\mathcal{D}$ is a $\mathcal{C}$-algebra homomorphism.
 Hence, $\bar{\pi}$ is a surjective $\mathcal{C}$-algebra homomorphism as well.
 By construction, we have $\ker(\bar{\pi})=\mathcal{R}\otimes(C)\otimes\mathcal{C}_2\otimes\mathcal{T}$, which is generated as an ideal by $\{1\otimes{c}\otimes1\otimes\et\ |\ c\in{C}\} \subseteq \bar{\mathcal{C}}$, and hence is closed under $\Der$ and $\Int$, i.e.\ $\ker(\bar{\pi})=J_1$.
 Therefore, $(\bar{\pi}(\mathrm{IDR}(\mathcal{R})),\bar{\pi}\circ\Der,\bar{\pi}\circ\Int)$ is an integro-differential ring and $\bar{\pi}$ is an integro-differential ring homomorphism, since integro-differential rings form a variety.
\begin{center}
 \begin{tikzpicture}
  \node at (0,0) (r) {$\mathcal{R}$};
  \node at (4,0) (gq) {$\mathcal{R}\otimes\mathcal{D}\otimes\mathcal{C}_2\otimes\mathcal{T}$};
  \node at (0,-2) (g) {$\mathrm{IDR}(\mathcal{R})$};
  \node at (4,-2) (gj) {$\mathrm{IDR}(\mathcal{R})/J_1$};
  \draw (r) edge[->] node[above] {$\iota$} (gq);
  \draw (r) edge[->] node[left] {$\iota$} (g);
  \draw (g) edge[->] node[above] {$\pi$} (gj);
  \draw (gj) edge[->] node[right] {$\psi$} (gq);
  \draw (g) edge[->] node[above] {$\bar{\pi}$} (gq);
 \end{tikzpicture}
\end{center}
 Since $\mathcal{R}\otimes\mathcal{C}\otimes\mathcal{C}\otimes\mathcal{T}$ is closed under $\Der$ and $\bar{\pi}$ acts as the identity on it, $\bar{\mathcal{C}}$-linearity of $\Der$ implies $\bar{\pi}\circ\Der=\Der$ on $\bar{\pi}(\mathrm{IDR}(\mathcal{R}))$.
 Factoring by $\ker(\bar{\pi})=J_1$, we obtain a bijective integro-differential ring homomorphism $\psi:\mathrm{IDR}(\mathcal{R})/J_1\to\mathcal{R}\otimes\mathcal{D}\otimes\mathcal{C}_2\otimes\mathcal{T}$ satisfying $\psi(\pi(f))=\bar{\pi}(f)$.
\end{proof}

Note that in general, i.e.\ if the effect of factoring out constants is not essentially limited to $\mathcal{C}_1$ as in the lemma, the multiplication \eqref{eq:GeneralizedShuffle} on $\mathcal{C}\otimes\mathcal{C}_1\otimes\mathcal{C}_2\otimes\mathcal{T}$ will need to be adapted as well.
In particular, this happens not only in the situation treated in Section~\ref{sec:closure}, but also when we recover the construction of the free integro-differential algebra (with weight $0$) presented in \cite{GuoRegensburgerRosenkranz} as follows.
\begin{remark}\label{rem:IDA}
 To obtain the free integro-differential algebra $\mathrm{ID}(\mathcal{R})^*$ of \cite{GuoRegensburgerRosenkranz}, we additionally need to impose multiplicativity of evaluation.
 Explicitly, it suffices to factor out from $\mathrm{IDR}(\mathcal{R})$ the ideal generated by the following constants.
 \[
  \{1\otimes(\ep(\Q\Der{fg})-\ep(f)\ep(g))\otimes1\otimes\et\ |\ f,g \in \mathcal{R}_\mathrm{J}\}, \quad\{1\otimes\ep(f\otimes{t})\otimes1\otimes\et\ |\ f \in \mathcal{R}_\mathrm{J},t \in \widetilde{\mathcal{T}}\},
 \]
 \[
  \{1\otimes1\otimes\ep(s\odot{t})\otimes\et\ |\ s,t \in \widetilde{\mathcal{T}}\}
 \]
 Effectively, this eliminates the involved definition of $(\mathcal{C}_2\otimes\mathcal{T},+,\cdot)$ from the construction of $\mathrm{IDR}(\mathcal{R})$, as it corresponds to replacing $(\mathcal{C}_2\otimes\mathcal{T},+,\cdot)$ by the shuffle algebra $(\mathcal{T},+,\shuffle)$ in addition to replacing $\mathcal{C}_1$ by an isomorphic copy of $\mathcal{R}$.
 We note that the necessary requirement of $\mathcal{R}$ being commutative is not explicitly mentioned in \cite[Sec.~4.2]{GuoRegensburgerRosenkranz} when constructing $\mathrm{ID}(\mathcal{R})^*$.
\end{remark}

\subsection{The integro-differential closure of a commutative quasi-integro-differential ring}
\label{sec:ExtQIDR}

In Section~\ref{sec:free}, we considered a commutative differential ring that admits a quasi-integration and we constructed the free commutative integro-differential ring on it.
In the construction below, we consider a commutative differential ring with a distinguished quasi-integration instead and we want to retain the integrals given by it.
To this end, we introduce the concept of a (not necessarily commutative) quasi-integro-differential ring.
Note that the following definition and lemma are a generalization of Definition~\ref{def:IDR} and Lemma~\ref{lem:IDRvariety}, no longer requiring the derivation to be surjective and turning integration into quasi-integration.

\begin{definition}\label{def:quasiIDR}
 Let $(\mathcal{R},\Der)$ be a differential ring with ring of constants $\mathcal{C}$.
 Let $\Q:\mathcal{R}\to\mathcal{R}$ be a quasi-integration.
 Then, $(\mathcal{R},\Der,\Q)$ is called a \emph{quasi-integro-differential ring} and we denote the two induced projectors onto the kernels of $\Der$ and $\Q$ by
 \[
  \E:=\id-\Q\Der \quad\text{and}\quad \T:=\id-\Der\Q.
 \]
\end{definition}
Quasi-integro-differential rings form a variety based on the following lemma whose proof is exactly as the proof of Lemma~\ref{lem:IDRvariety}.
\begin{lemma}\label{lem:QIDRvariety}
 Let $(\mathcal{R},\Der)$ be a differential ring with constants $\mathcal{C}$.
 Let $\Q:\mathcal{R}\to\mathcal{R}$ be a map such that $\Der\Q\Der=\Der$ and $\Q\Der\Q=\Q$ and define $\E:=\id-\Q\Der$.
 Then, $\Q$ is $\mathcal{C}$-linear if and only if it is additive and satisfies $(\E{g})\Q{f}=\Q(\E{g})f$ and $(\Q{f})\E{g}=\Q{f}\E{g}$ for all $f,g\in\mathcal{R}$.
\end{lemma}
\begin{proof}
 If $\Q$ is $\mathcal{C}$-linear, then the other properties immediately follow from $\Der\E=0$.
 Conversely, $\mathcal{C}$-linearity follows since $\E{g}=g$ holds for $g\in\mathcal{C}$.
\end{proof}

The differential rings which the construction in Section~\ref{sec:free} is based on are precisely those arising as reducts of commutative quasi-integro-differential rings.

\begin{remark}\label{rem:restriction}
In an integro-differential ring $(\mathcal{S},\Der,\Int)$ with evaluation $\E$, we can obtain a quasi-integration $\Q$ on a differential subring $\mathcal{R}\subseteq\mathcal{S}$ as follows, if $\E\mathcal{R}\subseteq\mathcal{R}$ and $\Der\mathcal{R}$ has a complement in $\mathcal{R}$ over $\E\mathcal{R}$.
For any $\mathcal{R}_\T$ that is such a complement of $\Der\mathcal{R}$ in $\mathcal{R}$, we can define $\Q{f}:=\Int(f-f_\T)$ for any $f\in\mathcal{R}$, where $f_\T\in\mathcal{R}_\T$ such that $f-f_\T\in\Der\mathcal{R}$.
Then, $\Q$ is $\mathcal{C}$-linear and we have $\Q\Der\Q{f}=\Q\Der\Int(f-f_\T)=\Q(f-f_\T)=\Q{f}$.
By $\Q\Der{f}=\Int\Der{f}$, we obtain $\Der\Q\Der{f}=\Der\Int=\Der{f}$ and $\Q\mathcal{R}=\Q\Der\mathcal{R}=\Int\Der\mathcal{R}=(\id-\E)\mathcal{R}\subseteq\mathcal{R}$.
So, $\Q$ is indeed a quasi-integration on $\mathcal{R}$ and the induced projector $\T{f}=f-\Der\Int(f-f_\T)=f_\T$ maps onto $\mathcal{R}_\T$.
Note that, in this construction, the induced projector $\id-\Q\Der$ is the restriction of the evaluation $\E$ from $\mathcal{S}$ and that $\Q\mathcal{R}=\Int\Der\mathcal{R}$ too is independent of the choice of $\mathcal{R}_\T$.
\end{remark}

For the rest of this section, we fix a commutative quasi-integro-differential ring $(\mathcal{R},\Der,\Q)$ and construct the commutative integro-differential closure $\mathrm{IDR}_\Q(\mathcal{R})$ of it that respects $\Q$.
Recall that this can be done by factoring $\mathrm{IDR}(\mathcal{R})$ by the integro-differential ideal generated by the constants \eqref{eq:constgenerators} in order to make $\Int$ agree with $\Q$ on $\Der\mathcal{R}$.
Lemma~\ref{lem:FactorConst} already allows us to construct $\mathrm{IDR}_\Q(\mathcal{R})$ equivalently by restricting to a subring of $\mathrm{IDR}(\mathcal{R})$ instead of taking a quotient, but the integration operation is still defined as projection of the integration of the full ring $\mathrm{IDR}(\mathcal{R})$.
In contrast, we now present a more direct construction of $(\mathrm{IDR}_\Q(\mathcal{R}),\Der,\Int)$ entirely without first constructing $\mathrm{IDR}(\mathcal{R})$ and performing underlying operations there.
In short, compared to the construction of $\mathrm{IDR}(\mathcal{R})$ in Section~\ref{sec:free}, we only need to modify the definitions of $M_1$ and $\Int$ slightly, as predicted by Lemma~\ref{lem:FactorConst}.

First, we replace the module $M_1$ defined in \eqref{eq:DefM1} by the submodule
\begin{equation}\label{eq:DefM1tilde}
 \widetilde{M}_1:=\mathcal{R}_\mathrm{J}\otimes\widetilde{\mathcal{T}}
\end{equation}
of it, which consequently leads to the definitions
\begin{align}
 \widetilde{\mathcal{C}}_1 &:= \Sym(\ep(\widetilde{M}_1))\label{eq:DefC1tilde}\\
 \mathrm{IDR}_\Q(\mathcal{R}) &:= \mathcal{R}\otimes\widetilde{\mathcal{C}}_1\otimes\mathcal{C}_2\otimes\mathcal{T}\label{eq:DefIDRQ}\\
 \widetilde{\mathcal{C}} &:= \mathcal{C}\otimes\widetilde{\mathcal{C}}_1\otimes\mathcal{C}_2\otimes\mathcal{C}\et \subseteq \mathrm{IDR}_\Q(\mathcal{R})\label{eq:DefCtilde}
\end{align}
instead of \eqref{eq:DefC1}, \eqref{eq:DefIDR}, and \eqref{eq:Cbar}, respectively.
Together with the other definitions and results of Section~\ref{sec:AlgebraStructure}, this completely explains $\mathrm{IDR}_\Q(\mathcal{R})$ as a $\mathcal{C}$-algebra.

Second, we modify the definition \eqref{eq:DefInt} of the integration by dropping from the $n=0$ case the term that involves an element of $M_1$ which lies in the complement of $\widetilde{M}_1$.
More explicitly, for $n\in\mathbb{N}$, $c=c_1\otimes{c_2}\in\widetilde{\mathcal{C}}_1\otimes\mathcal{C}_2$, and pure tensors $f\in\mathcal{R}_\T^{\otimes{n}}$, we use the formula
\begin{equation}\label{eq:DefIntQ}
 \Int(f_0\otimes{c}\otimes{f}):=
 \begin{cases}(\Q{f_0})\otimes{c}\otimes{f}+1\otimes{c}\otimes((f_0-\Der\Q{f_0})\otimes{f})&n=0\\[\smallskipamount]
 (\Q{f_0})\otimes{c}\otimes{f}-1\otimes(c_1{\cdot}\ep((\Q{f_0})\otimes{f}))\otimes{c_2}\otimes\et-{}\\\quad{}-\Int(((\Q{f_0})f_1)\otimes{c}\otimes{f_2^n})+1\otimes{c}\otimes((f_0-\Der\Q{f_0})\otimes{f})&n>0\end{cases}
\end{equation}
instead of \eqref{eq:DefInt}.
Together with definitions \eqref{eq:DefDeriv} and \eqref{eq:DefEval}, the lemmas stated in Section~\ref{sec:DerInt} remain true in $\mathrm{IDR}_\Q(\mathcal{R})$, except for formula \eqref{eq:Evaluation} being replaced by \eqref{eq:EvaluationExt} below.

For the sake of completeness, we explicitly state the analogs of the results in Section~\ref{sec:DerInt} and we mention what changes in their proofs, before we formulate and prove the main results Theorem~\ref{thm:IDRext} and Corollary~\ref{cor:universalExt} about $(\mathrm{IDR}_\Q(\mathcal{R}),\Der,\Int)$.
First, on $(\mathrm{IDR}_\Q(\mathcal{R}),\Der,\Int)$, $\Der$ and $\Int$ are $\widetilde{\mathcal{C}}$-linear, since \eqref{eq:linDer} and \eqref{eq:linInt} hold for all $f_0\in\mathcal{R}$, $c_1,d_1\in\widetilde{\mathcal{C}}_1$, $c_2,d_2\in\mathcal{C}_2$, $n\in\mathbb{N}$, and pure tensors $f\in\mathcal{R}_\T^{\otimes{n}}$.

\begin{lemma}\label{lem:Ext}
 For $\mathrm{IDR}_\Q(\mathcal{R})$, $\Der$, $\Int$, and $\E$ defined above, we have the following properties.
 \begin{enumerate}
  \item\label{item:LeibnizExt} The $\mathcal{C}$-linear map $\Der$ defined by \eqref{eq:DefDeriv} is a derivation on $\mathrm{IDR}_\Q(\mathcal{R})$.
  \item\label{item:RightInverseExt} On $\mathrm{IDR}_\Q(\mathcal{R})$, the $\mathcal{C}$-linear map $\Int$ defined by \eqref{eq:DefIntQ} is a right inverse of $\Der$.
  \item\label{item:EvaluationExt} Let $f_0\otimes{c_1}\otimes{c_2}\otimes{f} \in \mathrm{IDR}_\Q(\mathcal{R})$ with $f \in \mathcal{R}_\T^{\otimes{n}}$ for some $n \in \{0,1,2,\dots\}$. Then,
   \begin{equation}\label{eq:EvaluationExt}
    \E(f_0\otimes{c_1}\otimes{c_2}\otimes{f})=
    \begin{cases}1\otimes\big(c_1{\cdot}\big(\underbrace{f_0-\Q\Der{f_0}}_{\in \mathcal{C} \subseteq \widetilde{\mathcal{C}}_1}\big)\big)\otimes{c_2}\otimes{f}&n=0\\
    1\otimes(c_1{\cdot}\ep((\Q\Der{f_0})\otimes{f}))\otimes{c_2}\otimes\et&n>0\end{cases}.
   \end{equation}
  \item\label{item:constExt} The constants of $\mathrm{IDR}_\Q(\mathcal{R})$ are given by $\const_\Der(\mathrm{IDR}_\Q(\mathcal{R}))=\widetilde{\mathcal{C}}$.
 \end{enumerate}
\end{lemma}
\begin{proof}
 The proof of statement~\ref{item:LeibnizExt} is literally the same as that of Lemma~\ref{lem:Leibniz}.
 For statement~\ref{item:RightInverseExt}, the proof of Lemma~\ref{lem:RightInverse} carries over with $c=c_1\otimes{c_2}\in\widetilde{\mathcal{C}}_1\otimes\mathcal{C}_2$.
 For the base case $n=0$, the computation now reads
 \begin{align*}
  \Der\Int(f_0\otimes{c}\otimes\et) &= \Der\big((\Q{f_0})\otimes{c}\otimes\et+1\otimes{c}\otimes(f_0-\Der\Q{f_0})\big)\\
  &= (\Der\Q{f_0})\otimes{c}\otimes\et+(f_0-\Der\Q{f_0})\otimes{c}\otimes\et\\
  &= f_0\otimes{c}\otimes\et.
 \end{align*}
 For statement~\ref{item:EvaluationExt}, the proof of Lemma~\ref{lem:Evaluation} carries over.
 For the case $n=0$, the computation now reads
 \begin{align*}
  \E(f_0\otimes{c}\otimes{f}) &= f_0\otimes{c}\otimes{f} - \Int((\Der{f_0})\otimes{c}\otimes{f})\\
  &= f_0\otimes{c}\otimes{f} - \big((\Q\Der{f_0})\otimes{c}\otimes{f}+1\otimes{c}\otimes((\Der{f_0}-\Der\Q\Der{f_0})\otimes{f})\big)\\
  &= 1\otimes\big(c_1{\cdot}\big(\underbrace{f_0-\Q\Der{f_0}}_{\in \mathcal{C} \subseteq \widetilde{\mathcal{C}}_1}\big)\big)\otimes{c_2}\otimes{f}.
 \end{align*}
 Finally, for statement~\ref{item:constExt}, the proof of Lemma~\ref{lem:const} carries over, using statement~\ref{item:RightInverseExt} and \eqref{eq:EvaluationExt} instead of Lemma~\ref{lem:RightInverse} and \eqref{eq:Evaluation}, respectively.
\end{proof}

\begin{theorem}\label{thm:IDRext}
 With the above definitions, the following hold.
 \begin{enumerate}
  \item $(\mathrm{IDR}_\Q(\mathcal{R}),\Der,\Int)$ is a commutative integro-differential ring.
  \item\label{item:embeddingExt} $\iota:\mathcal{R}\to\mathrm{IDR}_\Q(\mathcal{R})$ defined by \eqref{eq:DefIota} is an injective differential ring homomorphism such that $\Int\iota(f)=\iota(\Q{f})$ holds for all $f \in \Der\mathcal{R}$.
  \item Let $J_1\subseteq\mathrm{IDR}(\mathcal{R})$ be the integro-differential ideal generated by \eqref{eq:constgenerators}.
   Then, there is a unique bijective integro-differential ring homomorphism $\psi:\mathrm{IDR}(\mathcal{R})/J_1\to\mathrm{IDR}_\Q(\mathcal{R})$ such that $\psi(\pi(\iota(f)))=\iota(f)$, where $\pi:\mathrm{IDR}(\mathcal{R})\to\mathrm{IDR}(\mathcal{R})/J_1$ is the canonical projection.
 \end{enumerate}
\end{theorem}
\begin{proof}
 Using Lemma~\ref{lem:Ext} instead of Lemmas \ref{lem:Leibniz}, \ref{lem:RightInverse}, and \ref{lem:const}, the proof of Theorem~\ref{thm:embedding} carries over to show that $(\mathrm{IDR}_\Q(\mathcal{R}),\Der,\Int)$ is a commutative integro-differential ring and that $\iota$ is an injective differential ring homomorphism.
 Moreover, by \eqref{eq:DefIntQ}, it follows that $\Int\iota(\Der{f})=\iota(\Q\Der{f})$ holds for all $f \in \mathcal{R}$.
\par
 Finally, we show that $\psi:\mathrm{IDR}(\mathcal{R})/J_1\to\mathrm{IDR}_\Q(\mathcal{R})$ satisfying the claimed properties exists.
 By Theorem~\ref{thm:universal}, there exists a unique integro-differential homomorphism $\bar{\pi}:\mathrm{IDR}(\mathcal{R})\to\mathrm{IDR}_\Q(\mathcal{R})$ such that $\bar{\pi}(\iota(f))=\iota(f)$ for all $f\in\mathcal{R}$.
 Since $\bar{\pi}(1\otimes\ep(f)\otimes1\otimes\et)=\bar{\pi}(\E\iota(f))=\E\bar{\pi}(\iota(f))=\E\iota(f)=0$ for $f\in\Q\Der\mathcal{R}$ by \eqref{eq:Evaluation} and \eqref{eq:EvaluationExt}, it follows that $J_1\subseteq\ker(\bar{\pi})$.
 Hence, by factoring out $J_1$, $\bar{\pi}$ induces a unique integro-differential ring homomorphism $\psi:\mathrm{IDR}(\mathcal{R})/J_1\to\mathrm{IDR}_\Q(\mathcal{R})$ with $\psi\circ\pi=\bar{\pi}$.
 Lemma~\ref{lem:FactorConst} implies that $\psi$ is bijective.
 Finally, $\psi$ is unique with $\psi\circ\pi\circ\iota=\iota$, since any such integro-differential ring homomorphism $\psi$ arises from the corresponding $\bar{\psi}:=\psi\circ\pi$ by factoring out $J_1$.
\end{proof}

\begin{remark}
If we consider the free integro-differential ring $\mathrm{IDR}(\mathcal{R})$ over $(\mathcal{R},\Der)$, we have that the evaluation of any non-constant element $f \in \mathcal{R}\setminus\mathcal{C}$ yields a new constant $\E\iota(f)\not\in\iota(\mathcal{R})$ by \eqref{eq:Evaluation}.
In contrast, in $\mathrm{IDR}_\Q(\mathcal{R})$, the property $\Int\iota(f)=\iota(\Q{f})$ on $\Der\mathcal{R}$ implies instead that the evaluation of all $f \in \mathcal{R}$ yields a constant from $\mathcal{C}$ by $\E\iota(f)=\iota(f-\Q\Der{f})\in\iota(\mathcal{C})$.
\end{remark}

In particular, the construction described in Remark~\ref{rem:restriction} can be applied to obtain a quasi-integration on $\mathcal{R}$ from the integration of $\mathrm{IDR}_\Q(\mathcal{R})$, whereby we recover the same $\Q$ only if we choose the same $\mathcal{R}_\T$.
From Theorems \ref{thm:IDRext} and \ref{thm:universal}, we can deduce the following universal property.
While $\mathrm{IDR}(\mathcal{R})$ is always isomorphic for different $\Q$ on the same $(\mathcal{R},\Der)$, we obtain isomorphic integro-differential rings $\mathrm{IDR}_\Q(\mathcal{R})$ only if quasi-integrations have the same image $\mathcal{R}_\mathrm{J}=\Q\mathcal{R}$.

\begin{corollary}\label{cor:universalExt}
 Let $(\mathcal{S},\Der,\Int)$ be a commutative integro-differential ring and let $\varphi:\mathcal{R}\to\mathcal{S}$ be a differential ring homomorphism such that $\varphi(\Q\Der{f})=\Int\Der\varphi(f)$.
 Then, there exists a unique integro-differential ring homomorphism $\eta:\mathrm{IDR}_\Q(\mathcal{R})\to\mathcal{S}$ such that $\eta(\iota(f))=\varphi(f)$ for all $f\in\mathcal{R}$.
\end{corollary}
\begin{proof}
 By Theorem~\ref{thm:universal}, there is a unique integro-differential ring homomorphism $\bar{\eta}:\mathrm{IDR}(\mathcal{R})\to\mathcal{S}$ such that $\bar{\eta}(\iota(f))=\varphi(f)$ for all $f\in\mathcal{R}$.
 By the assumption on $\varphi$, we have $\bar{\eta}(1\otimes\ep(\Q\Der{f})\otimes1\otimes\et) = \bar{\eta}(\E\iota(\Q\Der{f})) = \E\bar{\eta}(\iota(\Q\Der{f})) = \E\varphi(\Q\Der{f}) = \E\Int\Der\varphi(f) = 0$ for all $f\in\mathcal{R}$.
 With $J_1,\psi,\pi$ from Theorem~\ref{thm:IDRext}, we have $J_1=\ker(\psi\circ\pi)$ and $\bar{\eta}(J_1)=0$.
 Therefore, we obtain a unique integro-differential ring homomorphism $\eta:\mathrm{IDR}_\Q(\mathcal{R})\to\mathcal{S}$ such that $\bar{\eta}=\eta\circ\psi\circ\pi$.
 By Theorem~\ref{thm:IDRext}, on $\mathcal{R}$, we have $\eta\circ\iota=\eta\circ\psi\circ\pi\circ\iota=\bar{\eta}\circ\iota=\varphi$.
\begin{center}
 \begin{tikzpicture}
  \node at (0,0) (r) {$\mathcal{R}$};
  \node at (4,0) (s) {$\mathcal{S}$};
  \node at (0,-2) (g) {$\mathrm{IDR}(\mathcal{R})$};
  \node at (4,-2) (gq) {$\mathrm{IDR}_\Q(\mathcal{R})$};
  \draw (r) edge[->] node[above] {$\varphi$} (s);
  \draw (r) edge[->] node[left] {$\iota$} (g);
  \draw (r) edge[->] node[pos=0.3,above] {$\iota$} (gq);
  \draw (g) edge[->] node[above] {$\psi\circ\pi$} (gq);
  \draw (g) edge[->] node[pos=0.8,below] {$\bar{\eta}$} (s);
  \draw (gq) edge[->] node[right] {$\eta$} (s);
 \end{tikzpicture}
\end{center}
\par
 Now, assume that $\eta:\mathrm{IDR}_\Q(\mathcal{R})\to\mathcal{S}$ is an arbitrary integro-differential ring homomorphism such that $\eta(\iota(f))=\varphi(f)$ for all $f\in\mathcal{R}$.
 Then, $\tilde{\eta}:=\eta\circ\psi\circ\pi$ is an integro-differential ring homomorphism $\mathrm{IDR}(\mathcal{R})\to\mathcal{S}$ such that $\tilde{\eta}\circ\iota = \eta\circ\psi\circ\pi\circ\iota = \eta\circ\iota = \varphi$ on $\mathcal{R}$.
 By Theorem~\ref{thm:universal}, $\tilde{\eta}$ is identical with $\bar{\eta}$ above.
 Hence $\eta$ satisfies $\bar{\eta}=\eta\circ\psi\circ\pi$, which makes it unique as shown above.
\end{proof}

Note that, in Corollary~\ref{cor:universalExt}, $\varphi$ can in particular be any quasi-integro-differential ring homomorphism.
For such a $\varphi$, we necessarily have that $\varphi(\T{f})=\varphi(f)-\Der\Int\varphi(f)=0$ for all $f \in \mathcal{R}$, i.e.\ $\mathcal{R}_\T \subseteq \ker(\varphi)$.
Specifically, a quasi-integro-differential ring homomorphism $\varphi:\mathcal{R}\to\mathrm{IDR}_\Q(\mathcal{R})$ need not exist, however.
E.g., for $(\mathcal{R},\Der,\Q)=(\mathcal{C}((x)),\frac{d}{dx},\Q)$ in Example~\ref{ex:LaurentExt} below, we would obtain $\iota(1)=\varphi(x)\varphi(\frac{1}{x})=\varphi(\Q1)\varphi(\T\frac{1}{x})=(\Int\iota(1))0=0$ from $\mathcal{R}_\T \subseteq \ker(\varphi)$ and $\varphi(1)=\iota(1)$ for such a $\varphi$.

We briefly revisit the three basic examples treated in Section~\ref{sec:examples}.

\begin{example}
 If $(\mathcal{R},\Der,\Q)=(\mathcal{C},0,0)$ is a quasi-integro-differential ring with trivial derivation and quasi-integration, it follows that $\mathcal{C}_1=\widetilde{\mathcal{C}}_1=\mathcal{C}$, since $\mathcal{R}_\mathrm{J}=\{0\}$.
 So, nothing changes compared to Example~\ref{ex:FreeConstant} and we have $\mathrm{IDR}_\Q(\mathcal{R})=\mathrm{IDR}(\mathcal{R})$.
\end{example}

\begin{example}\label{ex:LaurentExt}
 For $(\mathcal{R},\Der,\Q)=(\mathcal{C}((x)),\frac{d}{dx},\Q)$ and $\Q$ as in Example~\ref{ex:FreeLaurent}, we have $\widetilde{\mathcal{C}}_1=\Sym(\ep(\mathcal{R}_\mathrm{J}\otimes\bigoplus_{n=1}^\infty\mathcal{C}x_n))$.
 Moreover, in contrast to Example~\ref{ex:FreeLaurent}, we have $\E\iota(f)=\iota(c)$ for $f\in\mathcal{R}$, where $c\in\mathcal{C}$ is the constant term of $f$, cf.~\eqref{eq:LaurentEval}.
 In analogy to Example~\ref{ex:FreeLaurent}, as a $\widetilde{\mathcal{C}}$-module, $\mathrm{IDR}_\Q(\mathcal{R})$ is isomorphic to $\widetilde{\mathcal{C}}\otimes\mathcal{R}[\ln(x)]$.
\end{example}

\begin{example}
 If $(\mathcal{R},\Der,\Q)$ is an integro-differential ring, then $\mathcal{R}_\T$ and hence $\widetilde{M}_1$ are the trivial module.
 Consequently, $\widetilde{\mathcal{C}}_1=\mathcal{C}$.
 Together with $\mathcal{C}_2=\mathcal{C}$ and $\mathcal{T}=\mathcal{C}\et$ obtained in Example~\ref{ex:FreeSurjective}, this yields $\mathrm{IDR}_\Q(\mathcal{R})=\iota(\mathcal{R})$, i.e.\ we recover the integro-differential ring we started with by Theorem~\ref{thm:IDRext}.\ref{item:embeddingExt}.
\end{example}

\subsection{Integro-differential ring closure without new constants}
\label{sec:NNC}

When extending the given differential ring $(\mathcal{R},\Der)$, it is sometimes advantageous to avoid the introduction of new constants altogether.
To achieve this, we can factor the free integro-differential ring $\mathrm{IDR}(\mathcal{R})$ by any ideal $J$ that is generated by an appropriate generating set of the $\mathcal{C}$-algebra $\bar{\mathcal{C}}$ such that $J\cap\mathcal{C}=\{0\}$.
Trivially, we can choose the generating set
\begin{equation}\label{eq:gensCbar}
 \mathcal{C}\otimes\ep(M_1)\otimes\mathcal{C}\otimes\mathcal{C}\et \enspace\cup\enspace \mathcal{C}\otimes\mathcal{C}\otimes\ep(M_2)\otimes\mathcal{C}\et \enspace\subseteq\enspace \bar{\mathcal{C}}.
\end{equation}

Instead of factoring $\mathrm{IDR}(\mathcal{R})/J$ based on \eqref{eq:gensCbar}, we can construct this ring more directly by removing all new constants from the construction of $\mathrm{IDR}(\mathcal{R})$.
Fixing $(\mathcal{R},\Der)$ and a quasi-integration $\Q$, we thereby straightforwardly construct
\begin{equation}\label{eq:DefSDR}
 \mathrm{IDR}_{\Q,\shuffle}(\mathcal{R}) := \mathcal{R}\otimes\mathcal{T}
\end{equation}
as tensor product of $\mathcal{C}$-algebras, where $\mathcal{T}$ as in \eqref{eq:DefTensors} is equipped with the shuffle product $\shuffle$, and define the derivation on $\mathrm{IDR}_{\Q,\shuffle}(\mathcal{R})$ via
\begin{equation}\label{eq:DefDerivNNC}
 \Der(f_0\otimes{f}):=
 \begin{cases}(\Der{f_0})\otimes{f}&n=0\\
 (\Der{f_0})\otimes{f}+(f_0f_1)\otimes{f_2^n}&n>0\end{cases}
\end{equation}
for $f_0\in\mathcal{R}$ and pure tensors $f\in\mathcal{R}_\T^{\otimes{n}}$.
We also let $\widetilde{\mathcal{T}}\subseteq\mathcal{T}$ as in \eqref{eq:DefTtilde}.
Since $\mathcal{C}\et$ is a complemented $\mathcal{C}$-submodule of $\mathcal{T}$, it follows that $\iota:\mathcal{R}\to\mathrm{IDR}_{\Q,\shuffle}(\mathcal{R})$ defined by $\iota(f):=f\otimes\et$ is an injective ring homomorphism, which even commutes with $\Der$ by \eqref{eq:DefDerivNNC}.
We will show that the differential ring $(\mathrm{IDR}_{\Q,\shuffle}(\mathcal{R}),\Der)$ indeed has no new constants in addition to $\iota(\mathcal{C})$.
An integration is defined on $\mathrm{IDR}_{\Q,\shuffle}(\mathcal{R})$ via
\begin{equation}\label{eq:DefIntNNC}
 \Int\!_\Q(f_0\otimes{f}):=
 \begin{cases}(\Q{f_0})\otimes{f}+1\otimes((f_0-\Der\Q{f_0})\otimes{f})&n=0\\[\smallskipamount]
 (\Q{f_0})\otimes{f}-\Int\!_\Q(((\Q{f_0})f_1)\otimes{f_2^n})+1\otimes((f_0-\Der\Q{f_0})\otimes{f})&n>0\end{cases}.
\end{equation}

For this construction, we obtain the following analogues of Lemma~\ref{lem:Ext}, Theorem~\ref{thm:IDRext}, and Corollary~\ref{cor:universalExt}.
Note the simplified formulae \eqref{eq:DefIntNNC} and \eqref{eq:EvaluationNNC} compared to the integrations and evaluations discussed for $\mathrm{IDR}(\mathcal{R})$ and $\mathrm{IDR}_\Q(\mathcal{R})$.

\begin{lemma}\label{lem:NNC}
 For $\mathrm{IDR}_{\Q,\shuffle}(\mathcal{R})$, $\Der$, and $\Int\!_\Q$ defined above, we have the following properties.
 \begin{enumerate}
  \item\label{item:LeibnizNNC} The $\mathcal{C}$-linear map $\Der$ defined by \eqref{eq:DefDerivNNC} is a derivation on $\mathrm{IDR}_{\Q,\shuffle}(\mathcal{R})$.
  \item\label{item:RightInverseNNC} On $\mathrm{IDR}_{\Q,\shuffle}(\mathcal{R})$, the $\mathcal{C}$-linear map $\Int\!_\Q$ defined by \eqref{eq:DefIntNNC} is a right inverse of $\Der$.
  \item\label{item:EvaluationNNC} The $\mathcal{C}$-linear map $\E_\Q:=\id-\Int\!_\Q\Der$ acts on $f_0\otimes{f} \in \mathrm{IDR}_{\Q,\shuffle}(\mathcal{R})$, where $f \in \mathcal{R}_\T^{\otimes{n}}$ and $n \in \{0,1,2,\dots\}$, by
   \begin{equation}\label{eq:EvaluationNNC}
    \E_\Q(f_0\otimes{f})=
    \begin{cases}\big(\underbrace{f_0-\Q\Der{f_0}}_{\in \mathcal{C}}\big)\otimes{f}&n=0\\
    0&n>0\end{cases}.
   \end{equation}
  \item\label{item:constNNC} The constants of $\mathrm{IDR}_{\Q,\shuffle}(\mathcal{R})$ are given by $\const_\Der(\mathrm{IDR}_{\Q,\shuffle}(\mathcal{R}))=\iota(\mathcal{C})$.
 \end{enumerate}
\end{lemma}
\begin{proof}
 The proof of Lemma~\ref{lem:Ext} carries over and becomes slightly simpler.
\end{proof}

\begin{theorem}\label{thm:IDRNNC}
 With the above definitions, the following hold.
 \begin{enumerate}
  \item $(\mathrm{IDR}_{\Q,\shuffle}(\mathcal{R}),\Der,\Int\!_\Q)$ is a commutative integro-differential ring.
  \item\label{item:embeddingNNC} $\iota:\mathcal{R}\to\mathrm{IDR}_{\Q,\shuffle}(\mathcal{R})$ defined by $\iota(f):=f\otimes\et$ is an injective differential ring homomorphism such that $\Int\!_\Q\iota(f)=\iota(\Q{f})$ holds for all $f \in \Der\mathcal{R}$.
   In addition, for all $f_0 \in \mathcal{R}_\mathrm{J}$ and all $f_1,\dots,f_n,g_1,\dots,g_m \in \mathcal{R}_\T$ where $n,m>0$, we have
   \begin{gather*}
    \E_\Q\iota(f_0)\Int\!_\Q\iota(f_1)\dots\Int\!_\Q\iota(f_n)=0 \quad\text{and}\\
    \E_\Q\big(\Int\!_\Q\iota(f_1)\dots\Int\!_\Q\iota(f_n)\big)\Int\!_\Q\iota(g_1)\dots\Int\!_\Q\iota(g_m)=0.
   \end{gather*}
  \item Let $J\subseteq\mathrm{IDR}(\mathcal{R})$ be the integro-differential ideal generated by \eqref{eq:gensCbar}.
   Then, there is a unique bijective integro-differential ring homomorphism $\psi:\mathrm{IDR}(\mathcal{R})/J\to\mathrm{IDR}_{\Q,\shuffle}(\mathcal{R})$ such that $\psi(\pi(\iota(f)))=\iota(f)$, where $\pi:\mathrm{IDR}(\mathcal{R})\to\mathrm{IDR}(\mathcal{R})/J$ is the canonical projection.
 \end{enumerate}
\end{theorem}
\begin{proof}
 Based on Lemma~\ref{lem:NNC} it immediately follows that $(\mathrm{IDR}_{\Q,\shuffle}(\mathcal{R}),\Der,\Int\!_\Q)$ is a commutative integro-differential ring and that $\iota$ is an injective differential ring homomorphism.
 Moreover, by \eqref{eq:DefIntNNC}, it follows that $\Int\!_\Q\iota(\Der{f})=\iota(\Q\Der{f})$ holds for all $f \in \mathcal{R}$.
 For pure tensors $f \in \mathcal{R}_\T^{\otimes{n}}$ with $n>0$, we obtain $\Int\!_\Q\iota(f_1)\dots\Int\!_\Q\iota(f_n)=1\otimes{f}$ by induction using \eqref{eq:DefIntNNC}.
 Note that, $\mathcal{R}\otimes\widetilde{\mathcal{T}} \subseteq \mathrm{IDR}_{\Q,\shuffle}(\mathcal{R})$ is an ideal, since the product used in $\mathcal{T}$ is the shuffle product.
 By \eqref{eq:EvaluationNNC}, $\E_\Q$ vanishes on this ideal, which implies the remaining identities claimed in statement \ref{item:embeddingNNC}.
\par
 Finally, we show that $\psi:\mathrm{IDR}(\mathcal{R})/J\to\mathrm{IDR}_{\Q,\shuffle}(\mathcal{R})$ satisfying the claimed properties exists.
 By Theorem~\ref{thm:universal}, there exists a unique integro-differential homomorphism $\bar{\pi}:\mathrm{IDR}(\mathcal{R})\to\mathrm{IDR}_{\Q,\shuffle}(\mathcal{R})$ such that $\bar{\pi}(\iota(f))=\iota(f)$ for all $f\in\mathcal{R}$.
 We have
 \[
  \bar{\pi}(f_0\otimes1\otimes1\otimes{f})=\iota(f_0)\Int\!_\Q\iota(f_1)\dots\Int\!_\Q\iota(f_n)=f_0\otimes{f}
 \]
 for all $f_0 \in \mathcal{R}$ and all pure tensors $f=f_1\otimes\dots\otimes{f_n} \in \mathcal{T}$.
 In particular, it follows that the restriction of $\bar{\pi}$ to $\mathcal{R}\otimes\mathcal{C}\otimes\mathcal{C}\otimes\mathcal{T}$ is bijective.
 Moreover, for all $f_0 \in \mathcal{R}_\mathrm{J}$ and all pure tensors $f=f_1\otimes\dots\otimes{f_n} \in \mathcal{T}$ and $s=s_1\otimes\dots\otimes{s_m},t=t_1\otimes\dots\otimes{t_l} \in \widetilde{\mathcal{T}}$, we obtain
 \[
  \bar{\pi}(1\otimes\ep(f_0\otimes{f})\otimes1\otimes\et)=\bar{\pi}(\E(f_0\otimes1\otimes1\otimes{f}))=\E_\Q(f_0\otimes{f})=0
 \]
 as well as
 \begin{multline*}
  \bar{\pi}(1\otimes1\otimes\ep(s\odot{t})\otimes\et)=\bar{\pi}(\E(1\otimes1\otimes1\otimes{s})(1\otimes1\otimes1\otimes{t}))\\
  =\E_\Q(1\otimes{s})(1\otimes{t})=\E_\Q(1\otimes(s\shuffle{t}))=0
 \end{multline*}
 by \eqref{eq:EvaluationNNC}.
 Therefore, $J\subseteq\ker(\bar{\pi})$.
 Hence, by factoring out $J$, $\bar{\pi}$ induces a unique integro-differential ring homomorphism $\psi:\mathrm{IDR}(\mathcal{R})/J\to\mathrm{IDR}_{\Q,\shuffle}(\mathcal{R})$ with $\psi\circ\pi=\bar{\pi}$.
 Since $\mathrm{IDR}(\mathcal{R})=(\mathcal{R}\otimes\mathcal{C}\otimes\mathcal{C}\otimes\mathcal{T})\oplus{J}$, bijectivity of the restriction of $\bar{\pi}$ to $\mathcal{R}\otimes\mathcal{C}\otimes\mathcal{C}\otimes\mathcal{T}$ implies that $\psi$ is bijective as well.
 Finally, $\psi$ is unique with $\psi\circ\pi\circ\iota=\iota$, since any such integro-differential ring homomorphism $\psi$ arises from the corresponding $\bar{\psi}:=\psi\circ\pi$ by factoring out $J$.
\end{proof}

\begin{corollary}\label{cor:universalNNC}
 Let $(\mathcal{S},\Der,\Int)$ be a commutative integro-differential ring and let $\varphi:\mathcal{R}\to\mathcal{S}$ be a differential ring homomorphism such that $\varphi(\Q\Der{f})=\Int\Der\varphi(f)$ for all $f \in \mathcal{R}$.
 Assume in addition that
 \begin{gather*}
  \E\varphi(f_0)\Int\varphi(f_1)\dots\Int\varphi(f_n)=0\\
  \E\big(\Int\varphi(f_1)\dots\Int\varphi(f_n)\big)\Int\varphi(g_1)\dots\Int\varphi(g_m)=0
 \end{gather*}
 for all $f_0 \in \mathcal{R}_\mathrm{J}$ and all $f_1,\dots,f_n,g_1,\dots,g_m \in \mathcal{R}_\T$ where $n,m>0$.
 Then, there exists a unique integro-differential ring homomorphism $\eta:\mathrm{IDR}_{\Q,\shuffle}(\mathcal{R})\to\mathcal{S}$ such that $\eta(\iota(f))=\varphi(f)$ for all $f\in\mathcal{R}$.
\end{corollary}
\begin{proof}
 The proof of Corollary~\ref{cor:universalExt} carries over, using Theorem~\ref{thm:IDRNNC} instead of Theorem~\ref{thm:IDRext}.
 To show $\bar{\eta}(J)=0$, we observe in addition that the assumptions on $\varphi$ allow us to compute $\bar{\eta}(1\otimes\ep(f_0\otimes{f})\otimes1\otimes\et) = 0$ and $\bar{\eta}(1\otimes1\otimes\ep(f\odot{g})\otimes\et) = 0$ for all $f_0 \in \mathcal{R}_\mathrm{J}$ and all pure tensors $f \in \mathcal{R}_\T^{\otimes{n}}$ and $g \in \mathcal{R}_\T^{\otimes{m}}$ with $n,m>0$.
\end{proof}

The construction of $\mathrm{IDR}_{\Q,\shuffle}(\mathcal{R})$ is particularly relevant in cases where the following minimality property plays a role.
\begin{lemma}\label{lem:minimalDR}
 There is no proper differential subring of $(\mathrm{IDR}_{\Q,\shuffle}(\mathcal{R}),\Der)$ that contains $\iota(\mathcal{R})$ and has surjective derivation.
\end{lemma}
\begin{proof}
 Assume $\mathcal{S} \subseteq \mathrm{IDR}_{\Q,\shuffle}(\mathcal{R})$ is a differential subring with $\Der\mathcal{S}=\mathcal{S}$ and $\iota(\mathcal{R})\subseteq\mathcal{S}$.
 Then, we show $\mathcal{R}\otimes\mathcal{R}_\T^{\otimes{n}}\subseteq\mathcal{S}$ by induction.
 For $n=0$, we have $\iota(\mathcal{R})\subseteq\mathcal{S}$.
 For $n>0$, we obtain $\mathcal{R}_\T\otimes\mathcal{R}_\T^{\otimes(n-1)}\subseteq\mathcal{S}$ by the induction hypothesis.
 By $\mathcal{S}\subseteq\Der\mathcal{S}$ and $\const_\Der(\mathrm{IDR}_{\Q,\shuffle}(\mathcal{R}))=\iota(\mathcal{C})\subseteq\mathcal{S}$, this implies $\mathcal{C}\otimes\mathcal{R}_\T^{\otimes{n}}\subseteq\mathcal{S}$.
 Since $\mathcal{S}$ is a ring, the claim follows by $\iota(\mathcal{R})\subseteq\mathcal{S}$.
\end{proof}

As indicated by the notation, $\Int\!_\Q$ and $\E_\Q$ depend on $\Q$ in an essential way.
For example, this is illustrated by the following property of $\E_\Q$.
\begin{lemma}
 The induced evaluation $\E_\Q$ in $(\mathrm{IDR}_{\Q,\shuffle}(\mathcal{R}),\Der,\Int\!_\Q)$ is multiplicative if and only if $\id-\Q\Der$ is multiplicative on $\mathcal{R}$.
\end{lemma}
\begin{proof}
 Since multiplication on $\mathcal{T}$ above is just the shuffle product, we have that $\mathcal{R}\otimes\widetilde{\mathcal{T}} \subseteq \mathrm{IDR}_{\Q,\shuffle}(\mathcal{R})$ is an ideal.
 By $\mathrm{IDR}_{\Q,\shuffle}(\mathcal{R}) = \iota(\mathcal{R}) \oplus (\mathcal{R}\otimes\widetilde{\mathcal{T}})$, any $f,g \in \mathrm{IDR}_{\Q,\shuffle}(\mathcal{R})$ can be written as $f=\iota(f_0)+\tilde{f}$ and $g=\iota(g_0)+\tilde{g}$ with $f_0,g_0 \in \mathcal{R}$ and $\tilde{f},\tilde{g} \in \mathcal{R}\otimes\widetilde{\mathcal{T}}$.
 In the same way, we have $fg=\iota(f_0g_0)+(\iota(f_0)\tilde{g}+\iota(g_0)\tilde{f}+\tilde{f}\tilde{g})$.
 By \eqref{eq:EvaluationNNC}, $\E_\Q$ is zero on $\mathcal{R}\otimes\widetilde{\mathcal{T}}$.
 For shorter notation, we abbreviate $e_\Q:=\id-\Q\Der$.
 Now, we compute $\E_\Q{f}=\iota(e_\Q{f_0})$, $\E_\Q{g}=\iota(e_\Q{g_0})$, and $\E_\Q{fg}=\iota(e_\Q{f_0g_0})$ by \eqref{eq:EvaluationNNC}.
 This implies the claim.
\end{proof}

Without any further modifications, this construction above can, for example, recover the integro-differential ring $(\mathcal{C}((x))[\ln(x)],\frac{d}{dx},\Int)$ of Example~\ref{ex:IDR}.

\begin{example}\label{ex:LaurentNNC}
 Consider $(\mathcal{R},\Der)=(\mathcal{C}((x)),\frac{d}{dx})$ with $\mathbb{Q}\subseteq\mathcal{C}$ and $\Q$ as in Example~\ref{ex:FreeLaurent}.
 In particular, $\mathcal{R}_\T=\mathcal{C}\frac{1}{x}$.
 Let $(\mathcal{S},\Der,\Int)=(\mathcal{C}((x))[\ln(x)],\frac{d}{dx},\Int)$ as in Example~\ref{ex:IDR}.
 Then, $\mathrm{IDR}_{\Q,\shuffle}(\mathcal{R})$ is isomorphic to $\mathcal{S}=\mathcal{R}[\ln(x)]$ by $\mu(f\otimes(\frac{1}{x})^{\otimes{n}})=f\frac{\ln(x)^n}{n!}$.
 In fact, it is straightforward to verify that this $\mu:\mathrm{IDR}_{\Q,\shuffle}(\mathcal{R})\to\mathcal{S}$ is even an integro-differential ring isomorphism.
\end{example}

\begin{example}\label{ex:Hurwitz}
 To construct the integro-differential ring of Hurwitz polynomials with coefficients in $\mathcal{C}$, which is a finite version of Hurwitz series \cite{Keigher}, we may start from any commutative ring of constants $(\mathcal{R},\Der)=(\mathcal{C},0)$, as in Example~\ref{ex:FreeConstant}, and take $\mathrm{IDR}_{0,\shuffle}(\mathcal{C})$.
\end{example}

\subsection{Internal closure in a commutative integro-differential ring}
\label{sec:closure}

In this section, we assume that $(\mathcal{S},\Der,\Int)$ is a commutative integro-differential ring such that $(\mathcal{S},\Der)$ is a differential ring extension of a fixed commutative differential ring $(\mathcal{R},\Der)$.
In other words, the inclusion map of $\mathcal{R}$ into $\mathcal{S}$ is a differential ring homomorphism.
Additionally, based on Theorem~\ref{thm:universal}, we fix $\eta:\mathrm{IDR}(\mathcal{R})\to\mathcal{S}$ as the unique integro-differential ring homomorphism such that $\eta(\iota(f))=f$ for all $f \in \mathcal{R}$.
As explained in Remark~\ref{rem:closure1}, $\eta$ maps the free integro-differential ring $\mathrm{IDR}(\mathcal{R})$ onto the internal integro-differential ring closure of $\mathcal{R}$ in $(\mathcal{S},\Der,\Int)$ and this internal closure $\bar{\mathcal{R}}:=\eta(\mathrm{IDR}(\mathcal{R}))$ is isomorphic to the quotient $\mathrm{IDR}(\mathcal{R})/\ker(\eta)$.
So, we provide a more explicit characterization of $\ker(\eta)$ in the following.
To this end, we additionally assume
\begin{equation}\label{eq:NNC}
 \const_\Der(\bar{\mathcal{R}})=\mathcal{C},
\end{equation}
i.e.\ that the internal closure has no new constants, in all that follows.

In the above setting, we can assign to any new constant from $\bar{\mathcal{C}}\setminus\iota(\mathcal{C})$ a value in $\mathcal{C}$ by computing the corresponding constant in $(\mathcal{S},\Der,\Int)$.
Explicitly, we have the following straightforward consequence.

\begin{lemma}\label{lem:closure}
 Let $J\subseteq\mathrm{IDR}(\mathcal{R})$ be the ideal generated by the constants
 \begin{equation}\label{eq:constantvalues}
  \begin{gathered}
   1\otimes\big(\ep(f)-\E{f_0}\Int{f_1}\dots\Int{f_n}\big)\otimes1\otimes\et \in \bar{\mathcal{C}} \quad\text{and}\\
   1\otimes1\otimes\big(\ep(s\odot{t})-\E\big(\Int{s_1}\dots\Int{s_m}\big)\Int{t_1}\dots\Int{t_l}\big)\otimes\et \in \bar{\mathcal{C}}
  \end{gathered}
 \end{equation}
 for all pure tensors $f=f_0\otimes\dots\otimes{f_n} \in M_1$ and $s=s_1\otimes\dots\otimes{s_m},t=t_1\otimes\dots\otimes{t_l} \in\widetilde{\mathcal{T}}$.
 Then, $J\subseteq\ker(\eta)$ and there is a unique integro-differential ring homomorphism $\tilde{\eta}:\mathrm{IDR}(\mathcal{R})/J\to\mathcal{S}$ with $\tilde{\eta}([\iota(f)]_J)=f$ for all $f \in \mathcal{R}$.
\end{lemma}
\begin{proof}
 In \eqref{eq:constantvalues}, the evaluations $\E{f_0}\Int{f_1}\dots\Int{f_n}$ and $\E\big(\Int{s_1}\dots\Int{s_m}\big)\Int{t_1}\dots\Int{t_l}$ in $(\mathcal{S},\Der,\Int)$ yield constants and hence lie in $\mathcal{C}$ by \eqref{eq:NNC}.
 Consequently, the quantities constructed in \eqref{eq:constantvalues} are indeed elements of $\bar{\mathcal{C}}$ for all pure tensors $f=f_0\otimes\dots\otimes{f_n} \in M_1$ and $s=s_1\otimes\dots\otimes{s_m},t=t_1\otimes\dots\otimes{t_l} \in\widetilde{\mathcal{T}}$.
 Therefore, the ideal $J\subseteq\mathrm{IDR}(\mathcal{R})$ is well-defined.
\par
 Note that $\Int{f_1}\dots\Int{f_n}=\eta(1\otimes1\otimes1\otimes{f_1^n})$.
 We clearly have $J\subseteq\ker(\eta)$, since $\eta$ maps every generator \eqref{eq:constantvalues} of $J$ to zero by definition, cf.\ the proof of Theorem~\ref{thm:universal}.
 Hence, $\eta$ induces a unique integro-differential ring homomorphism $\tilde{\eta}:\mathrm{IDR}(\mathcal{R})/J\to\mathcal{S}$ with $\tilde{\eta}([\iota(f)]_J)=f$ for all $f \in \mathcal{R}$.
\end{proof}

The following lemma provides a characterization of the equality $J=\ker(\eta)$ for the above ideal generated by the constants \eqref{eq:constantvalues}.
Recall from Section~\ref{sec:independence} the definitions of $\sigma$ and $\tau$, which allow us to abbreviate nested integrals
\[
 \sigma(W)=\Int{a_{w_1}}\dots\Int{a_{w_{|W|}}} \in \mathcal{S}
\]
and tensors $\tau(W)=a_{w_1}\otimes\dots\otimes{a_{w_{|W|}}} \in \mathcal{T}$ in terms of words $W=w_1\dots{w_{|W|}} \in \langle{I}\rangle$ using a fixed subset $\{a_i\ |\ i\in{I}\}\subseteq\mathcal{R}_\mathrm{T}$.
Note that we have $\sigma(W)=\eta(1\otimes1\otimes1\otimes\tau(W))$.
If $\linspan_\mathcal{C}\{a_i\ |\ i\in{I}\}=\mathcal{R}_\mathrm{T}$, then with these definitions it is clear that the ideal $J\subseteq\mathrm{IDR}(\mathcal{R})$ given in Lemma~\ref{lem:closure} is generated by all
\begin{equation}\label{eq:constantvaluesW}
 \begin{gathered}
  1\otimes\big(\ep(f\otimes\tau(W))-\E{f}\sigma(W)\big)\otimes1\otimes\et \in \bar{\mathcal{C}}\quad\text{and}\\
  1\otimes1\otimes\big(\ep(\tau(U)\odot\tau(V))-\E\sigma(U)\sigma(V)\big)\otimes\et \in \bar{\mathcal{C}}
 \end{gathered}
\end{equation}
with $f \in \mathcal{R}_\mathrm{J}$, $W \in \langle{I}\rangle$, and $U,V \in \langle{I}\rangle^*$.

\begin{lemma}\label{lem:closure2}
 Assume that $\mathcal{R}_\T$ is a free $\mathcal{C}$-module with basis $\{a_i\ |\ i\in{I}\}\subseteq\mathcal{R}_\mathrm{T}$.
 Let the ideal $J\subseteq\mathrm{IDR}(\mathcal{R})$ be generated by \eqref{eq:constantvaluesW}.
 Then, $J=\ker(\eta)$ holds if and only if the nested integrals $\{\sigma(W)\ |\ W \in \langle{I}\rangle\}\subseteq\mathcal{S}$ are $\mathcal{R}$-linearly independent.
\end{lemma}
\begin{proof}
 Let $\tilde{\eta}$ be given by Lemma~\ref{lem:closure}.
 Since $J\subseteq\ker(\eta)$, the equality $J=\ker(\eta)$ is equivalent to $\tilde{\eta}$ being injective.
 By \eqref{eq:constantvalues}, factoring by $J$ amounts to assigning to any new constant from $\bar{\mathcal{C}}\setminus\iota(\mathcal{C})$ a value in $\mathcal{C}$ in such a way that the resulting canonical projection $\tilde{\pi}:\mathrm{IDR}(\mathcal{R})\to\mathrm{IDR}(\mathcal{R})/J$ is bijective when restricted to the $\mathcal{C}$-submodule
 \begin{equation}\label{eq:defM}
  M:=\mathcal{R}\otimes\mathcal{C}\otimes\mathcal{C}\otimes\mathcal{T}\subseteq\mathrm{IDR}(\mathcal{R}).
 \end{equation}
 Altogether, by $\eta=\tilde{\eta}\circ\tilde{\pi}$, we have $J=\ker(\eta)$ if and only if $\eta$ restricted to $M$ is injective.
 It remains to show that injectivity of $\eta|_M$ is equivalent to $\mathcal{R}$-linear independence of $\{\sigma(W)\ |\ W \in \langle{I}\rangle\}$.
\par
 By the assumption on $\mathcal{R}_\T$, $\mathcal{T}$ is a free $\mathcal{C}$-module with basis $\{\tau(W)\ |\ W \in \langle{I}\rangle\}$.
 So, by Lemma~\ref{lem:tensorproduct}, we can write any $f \in M$ uniquely as $f=\sum_{i=1}^nf_i\otimes1\otimes1\otimes\tau(W_i)$, where $W_i\in\langle{I}\rangle$ are pairwise distinct and $f_i\in\mathcal{R}$ are nonzero, and we note that $\eta(f)=\sum_{i=1}^nf_i\sigma(W_i)$.
 If $\eta|_M$ is injective and $\sum_{i=1}^nf_i\sigma(W_i)$ is zero in $\mathcal{S}$ for some nonzero $f_i\in\mathcal{R}$ and pairwise distinct $W_i\in\langle{I}\rangle$, then $\sum_{i=1}^nf_i\otimes1\otimes1\otimes\tau(W_i)=0$.
 So, $n=0$ follows by uniqueness.
 Conversely, assume that $\{\sigma(W)\ |\ W \in \langle{I}\rangle\}\subseteq\mathcal{S}$ is $\mathcal{R}$-linearly independent and $f \in M\cap\ker(\eta)$.
 By $0=\eta(f)=\sum_{i=1}^nf_i\sigma(W_i)$, it then follows that $n=0$ and hence $f=0$.
\end{proof}

In order to make use of Lemma~\ref{lem:closure2}, we now focus on a large class of differential rings $(\mathcal{R},\Der)$, for which Theorem~\ref{thm:independence} shows that nested integrals of this form are indeed $\mathcal{R}$-linearly independent.
If $\mathcal{R}$ is a field of characteristic zero, then independence also follows by \cite[Thm.~1]{DeneufchatelEtAl}.

\begin{theorem}\label{thm:closure}
 Assume that $\mathcal{R}_\T$ is a free $\mathcal{C}$-module with basis $\{a_i\ |\ i\in{I}\}\subseteq\mathcal{R}_\mathrm{T}$.
 Let the ideal $J \subseteq \mathrm{IDR}(\mathcal{R})$ be generated by all constants \eqref{eq:constantvaluesW}
 with $f \in \mathcal{R}_\mathrm{J}$, $W \in \langle{I}\rangle$, and $U,V \in \langle{I}\rangle^*$.
 Assume further that the differential ring $(\mathcal{R},\Der)$ is such that, for every nonzero element, the differential ideal generated by it contains a nonzero constant.
 Then, $\tilde{\eta}:\mathrm{IDR}(\mathcal{R})/J\to\bar{\mathcal{R}}$ given by Lemma~\ref{lem:closure} is an integro-differential ring isomorphism.
\end{theorem}
\begin{proof}
 By construction, $\tilde{\eta}:\mathrm{IDR}(\mathcal{R})/J\to\bar{\mathcal{R}}$ is a surjective integro-differential ring homomorphism.
 By Theorem~\ref{thm:independence}, $\{\sigma(W)\ |\ W \in \langle{I}\rangle\}\subseteq\mathcal{S}$ is $\mathcal{R}$-linearly independent.
 So, Lemma~\ref{lem:closure2} implies $J=\ker(\eta)$, i.e.\ $\tilde{\eta}$ is injective on $\mathrm{IDR}(\mathcal{R})/J$.
\end{proof}

In principle, an analog of Theorem~\ref{thm:closure} could also be worked out for $\mathrm{IDR}_\Q(\mathcal{R})$ instead of $\mathrm{IDR}(\mathcal{R})$, if $\Q$ agrees with $\Int$ on $\Der\mathcal{R} \subseteq\mathcal{S}$.
Analogously to what we have done in Section~\ref{sec:ExtQIDR} explicitly, one could modify the construction of $\mathrm{IDR}(\mathcal{R})$ to directly respect the relations of constants imposed by $J$ instead of taking the quotient $\mathrm{IDR}(\mathcal{R})/J$.
The construction shown in Section~\ref{sec:NNC} is in fact a special case of this.

Finally, we now look at concrete internal integro-differential closures.
We start with a rather simple case, continuing Example~\ref{ex:FreeLaurent}, where $\mathcal{R}_\T$ is cyclic and all relevant evaluations involving nested integrals are in fact zero in $\mathcal{S}$.
The second example will have neither of these properties and will treat the integro-differential ring formed from rational functions and their nested integrals mentioned at the very beginning of the paper.

\begin{example}\label{ex:LaurentClosure}
 With $\mathbb{Q}\subseteq\mathcal{C}$, let $(\mathcal{R},\Der):=(\mathcal{C}((x)),\frac{d}{dx})$ and $(\mathcal{S},\Der,\Int):=(\mathcal{R}[\ln(x)],\frac{d}{dx},\Int)$ as in Example~\ref{ex:FreeLaurent}.
 Note that $\bar{\mathcal{R}}=\mathcal{S}$ in this case.
 Setting $a_0:=\frac{1}{x}$ and $I:=\{0\}$, we have $\sigma(W)=\frac{\ln(x)}{n!}$ for $W \in \langle{I}\rangle$ where $n=|W|$.
 Consequently, $\E{f}\sigma(U)=0$ and $\E\sigma(U)\sigma(V)=0$ by \eqref{eq:LaurentEval} for all $f\in\mathcal{R}$ and $U,V \in \langle{I}\rangle^*$.
 Since $\mathcal{R}_\mathrm{J}$ was defined in Example~\ref{ex:FreeLaurent} as the set of series without constant term, we also have $\E{f}\sigma(W)=0$ by \eqref{eq:LaurentEval} for all $f \in \mathcal{R}_\mathrm{J}$ when $W$ is the empty word.
 This explains all evaluations occurring in \eqref{eq:constantvaluesW}.
 Then, if $\mathcal{C}$ is Noetherian, Theorem~\ref{thm:closure} can be applied based on Lemma~\ref{lem:Noetherian} and shows that the quotient $\mathrm{IDR}(\mathcal{R})/\ker(\eta)$ can be obtained by setting all constants $1\otimes{c_1}\otimes1\otimes\et$ and $1\otimes1\otimes{c_2}\otimes\et$ with $c_1\in\ep(M_1)$ and $c_2\in\ep(M_2)$ to zero, which agrees with the choice in \eqref{eq:gensCbar}.
 So, this quotient is isomorphic to $\mathrm{IDR}_{\Q,\shuffle}(\mathcal{R})$ as discussed in Example~\ref{ex:LaurentNNC}.
 Since $\eta$ is surjective onto $\mathcal{S}$, this quotient is isomorphic to $(\mathcal{S},\Der,\Int)$ as integro-differential ring.
\end{example}

If $J=\ker(\eta)$ holds as in Lemma~\ref{lem:closure2}, then the quotient $\mathrm{IDR}(\mathcal{R})/\ker(\eta)$ is canonically isomorphic to $\mathrm{IDR}_{\Q,\shuffle}(\mathcal{R})$ as a differential module.
However, we do not obtain an integro-differential ring homomorphism this way in general, since the constants \eqref{eq:gensCbar} need not be contained in $\ker(\eta)$.

\begin{example}\label{ex:hyperlogs}
 Let $\mathcal{C}$ be a field of characteristic zero, let $(\mathcal{R},\Der)$ be the field of rational functions $(\mathcal{C}(x),\frac{d}{dx})$, and let $(\mathcal{S},\Der,\Int):=(\mathcal{C}((x))[\ln(x)],\frac{d}{dx},\Int)$ as in Example~\ref{ex:IDR}.
 On $\mathcal{R}$, we consider the quasi-integration $\Q$ defined in Example~\ref{ex:Q}.
 Again, $\mathrm{IDR}(\mathcal{R})/\ker(\eta)$ is isomorphic to the internal integro-differential ring closure $\bar{\mathcal{R}}$ of $\mathcal{C}(x)$ in $(\mathcal{S},\Der,\Int)$ and we determine $\ker(\eta)$ based on Theorem~\ref{thm:closure}.
 A particular basis of $\mathcal{R}_\T=\ker(\Q)$ is given by all $\frac{x^k}{p}$ such that $p\in\mathcal{C}[x]$ is monic irreducible and $0\le{k}<\deg(p)$.
 We take a closer look at the evaluations appearing in \eqref{eq:constantvaluesW}.
 Observe that $\eta(1\otimes1\otimes1\otimes{t}) \in \mathcal{C}[[x]][\ln(x)]$ for all $t \in \widetilde{\mathcal{T}}$, since any element of $\mathcal{R}_\T=\ker(\Q)$ lies in $\frac{1}{x}\mathcal{C}[[x]]\subseteq\mathcal{S}$.
 By \eqref{eq:Evaluation}, we also have $\E\eta(1\otimes1\otimes1\otimes{t})=0$ for all $t \in \widetilde{\mathcal{T}}$.
 Now, \eqref{eq:LaurentEval} implies that $\E\eta(f\otimes1\otimes1\otimes{t})=0$ and $\E\eta(1\otimes1\otimes1\otimes{s})\eta(1\otimes1\otimes1\otimes{t})=0$ for all $f \in \mathcal{C}(x)\cap\mathcal{C}[[x]]$ and $s,t \in \widetilde{\mathcal{T}}$.
 Moreover, $\eta(1\otimes1\otimes1\otimes(t\otimes\frac{1}{x})) \in \ln(x)\mathcal{C}[[x]][\ln(x)]$ for all $t \in \mathcal{T}$ implies $\E\eta(f\otimes1\otimes1\otimes(t\otimes\frac{1}{x}))=0$ for all $f \in \mathcal{C}(x)$ and $t \in \mathcal{T}$.
 Computing $\E{f}$ for $f \in \mathcal{R}_\mathrm{J}$ via the partial fraction decomposition of $f$, we obtain from \eqref{eq:LaurentEval} that $\E\frac{1}{x^k}=0$ for $k\ge2$ and that $\E\frac{x^k}{p^m}=0$ for $k\ge1$ and monic irreducible $p \in \mathcal{C}[x]$ with $p\neq{x}$, while $\E\frac{1}{p^m}=\frac{1}{p(0)^m}$ is nonzero.
 Altogether, this explains many evaluations appearing in \eqref{eq:constantvaluesW} to be zero.
 The remaining ones involve more computation and are given by $\E\frac{1}{x^k}\Int{f_1}\dots\Int{f_n}$ for $f_1,\dots,f_n \in \mathcal{R}_\T$ with $k,n\ge1$ and $f_n \in \mathcal{C}[[x]]$.
 In particular, they are not all zero, as $\E\tfrac{1}{x}\Int\tfrac{1}{x+1} = 1$ illustrates.
 By Theorem~\ref{thm:closure}, $J=\ker(\eta)$ is generated by the constants in \eqref{eq:constantvaluesW} and factoring yields an integro-differential ring isomorphism $\tilde{\eta}:\mathrm{IDR}(\mathcal{R})/J\to\bar{\mathcal{R}}$.
\end{example}

\subsubsection*{Acknowledgements}
This work was funded in part by the Austrian Science Fund (FWF): 10.55776/P27229, 10.55776/P31952, and 10.55776/P37258.
\href{https://www.fwf.ac.at/en/about-us/what-we-do/open-science/open-access-policy/open-access-policy-for-peer-reviewed-publications}{For open access purposes}, the author has applied a CC BY public copyright license to any author accepted manuscript version arising from this submission.
Part of this work was done while both authors were at the Institute for Algebra of the Johannes Kepler University Linz, Austria.
The authors would like to thank the anonymous reviewer for the suggestions to improve the presentation of this work.

\end{document}